\documentclass[12pt]{amsart}
\usepackage{graphicx}
\usepackage{amsmath, amscd, amssymb, amsthm}
\usepackage[subnum]{cases}
\usepackage{changepage}
\usepackage{xcolor}
\usepackage{marginnote}
\usepackage{hyperref}
\usepackage{cleveref}
\usepackage[all]{xy}
\usepackage{tabularx}
\usepackage{ltablex}
\usepackage[T1]{fontenc}

\newtheorem{theorem}{Theorem}[section]
\newtheorem{lemma}[theorem]{Lemma}

\newtheorem{proposition}[theorem]{Proposition}
\newtheorem{corollary}[theorem]{Corollary}
\newtheorem{conjecture}[theorem]{Conjecture}

\theoremstyle{definition}
\newtheorem{definition}[theorem]{Definition}
\newtheorem{remark}[theorem]{Remark}
\numberwithin{equation}{section}
\newtheorem{example}[theorem]{Example}
\newtheorem{assumption}[theorem]{Assumption}
\newtheorem{setting}[theorem]{Setting}

\usepackage{fancyhdr}
\begin{document}

\normalfont

\title{Analytic Geometry and Hodge-Frobenius Structure}
\author{Xin Tong}

\maketitle

\begin{abstract}
\rm  In this paper, we study Frobenius structures in higher dimensional $p$-adic analytic geometry and the corresponding $p$-adic functional analysis. This will build up foundations for further study on some generalized cohomology of Frobenius modules and the corresponding generalized Iwasawa theory and generalized noncommutative Tamagawa number conjectures in the spirit of Burns-Flach-Fukaya-Kato and Nakamura (as well as certainly the original noncommutative Tamagawa number conjectures as observed by Pal-Z\'abr\'adi). We will work in the program proposed by Carter-Kedlaya-Z\'abr\'adi and after Pal-Z\'abr\'adi, and we will follow closely the approach from Kedlaya-Pottharst-Xiao to investigate the corresponding deformation of the generalized $p$-adic Hodge structures. 
\end{abstract}

\newpage

\tableofcontents

\newpage

\section{Introduction}

\subsection{Higher Dimensional Analytic Geometry}

Higher dimensional analytic geometry in the nonarchimedean setting and corresponding $p$-adic analysis received very rapid development in the past decades, dated back to the last century to some point. Its complex analog, which is to say the complex analysis in higher dimension of several variables or the study of complex analytic varieties could be dated back even earlier. For instance searching for the automorphism groups of higher dimensional domains is still attracting extensive attention. On the other hand in the $p$-adic setting, considering higher dimensional analytic domains will be a very natural generalization (especially in the rigid analytic geometry, or more general analytic geometry). However, the approaches do not generalize very smoothly to the higher dimension, even, one will immediately need to find new approaches. This is actually already clear in the complex setting, for instance one may believe that complex analysis of several variables over complex polydiscs is natural generalization of one variable situation, for instance the Cauchy integration formula and maximal value issues. However in the $p$-adic setting, this might be more complicated as discussed in the following.\\

\indent The famous $p$-adic local monodromy theorem, proved by Tsuzuki in the unit-root setting, and then proved by Andr\'e-Kedlaya-Mebkhout \cite{Ked1}, \cite{And1} and \cite{Me} in the full general setting, is already a very difficult theorem over one-dimensional $p$-adic annulus, which roughly speaks that over such annulus, each finite locally free sheaf with connection and Frobenius (not necessarily absolute) structure could be made simpler with respect to the connection structure. For instance in the unit-root setting, the picture is even simpler, since then we only need to consider Galois representations with finite monodromy which is to say potentially unramified somewhat: we have an equivalence between the category of unit-root overconvergent isocrystals and the one of all the potentially unramified \'etale $\mathbb{Q}_p$-local systems, over some space. Actually Kedlaya's approach established in \cite{Ked1} used very deep theorem on the slope filtration over one-dimensional $p$-adic annulus to reduce to Tsuzuki's approach.\\

\indent To generalize to the higher dimensional situation, we have many ways to do so which are encoded in the following projects. First one could consider the picture in the relative $p$-adic Hodge theory, for instance established in \cite{KL1} and \cite{KL2}, or the multidimensional program of \cite{CKZ18} and \cite{PZ19}. Also we would like to mention some relativization programs in \cite{BC1}, \cite{Ked2}, \cite{Liu1} where some attempt to generalize the slope filtration theorem and the local monodromy theorem to the relative Frobenius situation had been successful. The applications from \cite{Ked1} to the rigid cohomology and isocrystals have already considered some sort of multidimensional construction as in \cite{Ked3},\cite{Ked4},\cite{Ked5},\cite{Ked6},\cite{Ked7},\cite{Ked8}. However, one could still only get information deduced from the monodromy theorem or the slope filtration theorem essentially from the one mentioned above. The theory of partial differential equations have been regarded as non-trivial and significant generalization of the story relevant above. To be more precise, these are reflected in the following work: \cite{Xiao1}, \cite{Xiao2} and \cite{KX1}, \cite{Ked10}, \cite{Ked11}, \cite{Ked12}.\\

\indent That being all said, we will only focus on mixed-characteristic situation in this paper. Here let us do some summarization in a uniform way on the known pictures which one might want to follow being local or global, being \'etale or non-\'etale:

\begin{setting}
Mixed-characteristic local \'etale: Berger's theorem for potentially semi-stable Galois representations \cite{Ber1}.	
\end{setting}

\begin{setting}
Mixed-characteristic local non-\'etale: Again Berger's theorem for $(\varphi,\Gamma)$-modules.
\end{setting}

\begin{setting}
Mixed-characteristic global \'etale: Relative $p$-adic Hodge theory after Faltings, Andreatta, Iovita, Brinon, Scholze, Kedlaya, Liu and so on, in \cite{Fal}, \cite{A1}, \cite{AI1}, \cite{AI2}, \cite{AB1}, \cite{AB2}, \cite{AB3}, \cite{Sch1}, \cite{KL1} and \cite{KL2}.
\end{setting}

\begin{setting}
Mixed-characteristic global non-\'etale: Relative $p$-adic Hodge theory in the style of Kedlaya-Liu \cite{KL1} and \cite{KL2}.	
\end{setting}

\indent In this paper, we first consider the program of \cite{CKZ18} and \cite{PZ19}, and we will work in the generality of \cite{KPX}.  The generalization itself is already interesting enough and actually we need to address some fundamental issues around. The multidimensionalization we studied here represents some similarity to the globalization and multidimensionalization mentioned above in the relative $p$-adic Hodge theory and some globalization of the study of isocrystals in the style of Kedlaya's theorem on Shiho's conjecture. Actually we consider slightly different actions within the module structures, which is to say that we will have multi-Frobenius actions, multi-Lie group actions, multi-differentials. 

\subsection{Results and Applications}

\indent Here is our first main result, while the corresponding proof will be given in \cref{section4.1}:

\begin{theorem}
With the same assumptions as in \cref{assumption1} as below, we set $M$ to be the global section of a locally free coherent sheaf of module over the sheaf of the Robba ring $\Pi_{\mathrm{an},\mathrm{con},I,X}(\pi_{K_I})$ over an affinoid domain $X$ attached to an affinoid algebra $A$ in the rigid analytic geometry, carrying mutually commuting actions of $(\varphi_I,\Gamma_{K_I})$. Then we have the corresponding Herr complex $C^\bullet_{\varphi_I,\Gamma_{K_I}}(M)$ lives then in the derived category $\mathbb{D}^\flat_{\mathrm{perf}}(A)$\footnote{One should be able to work in the $(\infty,1)$-derived categories as well, see \cref{definition4.5}, though we just use the regular derived categories to formulate our results without the enhancement.}, and we have that the corresponding complex $C^\bullet_{\psi_I}(M)$ has all the cohomology groups which are coadmissible over $\Pi_{\mathrm{an},\infty,I,A}(\Gamma_{K_I})$. Moreover when $A=\mathbb{Q}_p$, $C^\bullet_{\psi_I}(M)$ lives in $\mathbb{D}^\flat_{\mathrm{perf}}(\Pi_{\mathrm{an},\infty,I,A}(\Gamma_{K_I}))$.
\end{theorem}

\indent Another motivation in our situation comes from the equivariant Iwasawa theory proposed by Kedlaya-Pottharst in \cite{KP1}. Here the idea will be working over more general distribution algebra after Berthelot, Schneider and Teitelbaum in the following sense. The idea is from the following observation. Now suppose $G$ is a nice $p$-adic Lie group for instance as above we considered the product of $\mathbb{Z}_p$ or $\mathbb{Z}_p^\times$. Then we consider the corresponding distribution algebra which we will use the notation $\Pi_{\mathrm{an},\infty,A}(G)$ to denote this ring (note that here we have already considered the even more general relative setting). Now consider $L/K$ a $p$-adic Lie extension. Then we have that one could consider the projection $G_{K}\rightarrow G(L/K)$ to define the corresponding Galois representation with coefficients in $\Pi_{\mathrm{an},\infty,A}(G)$, which then gives rise to a $(\varphi,\Gamma_K)$-module over $\Pi_{\mathrm{an},\infty,A}(G)$. Then we take the external tensor product of any $(\varphi,\Gamma)$-module $M$ with this $p$-adic Lie deformation $\mathbf{DfmLie}$ we will get a module $M\boxtimes \mathbf{DfmLie}$ which is actually $(\varphi,\Gamma)$-module over $\Pi_{\mathrm{an},\infty,A}(G)$. Then we are going to define the Herr complex of this module, which will be called the $G$-equivariant Iwasawa cohomology. It is natural to investigate the fundamental properties of this cohomology which is over $\mathrm{Max}A\times \mathrm{Max}\Pi_{\mathrm{an},\infty}(G)$. For instance over $\mathbb{Z}_p^{\times,d}$ as above, we will have to study the cohomology of the big $(\varphi,\Gamma)$-module over the product of the ring $\Pi_{\mathrm{an},\mathrm{con},A}(\pi_K)$ and the ring $\Pi_{\mathrm{an},\infty}(G)$. Therefore through the previous theorem we have the following coherence over more general Fr\'echet-Stein algebras, while the corresponding proof will be given in \cref{section4.1}:

\begin{theorem}
Let $X_\infty(G)$ be the quasi-Stein rigid analytic space attached to the $p$-adic Lie group $G:=\Gamma_{K_J}$ for some set $J$ possibly different from $I$ as in \cref{setting3.1}. Let $A$ be an affinoid algebra over $\mathbb{Q}_p$. Suppose $\mathcal{F}$ is a locally free coherent sheaf over the sheaf of Robba ring $\Pi_{\mathrm{an},\mathrm{con},I,X_\infty(G)\widehat{\otimes} A}(\pi_{K_I})$, where we use the same notation to denote the global section and we assume that the global section is finite projective, carrying mutually commuting actions of $(\varphi_I,\Gamma_{K_I})$. Then we have that the cohomology groups of the Herr complex $C^\bullet_{\varphi_I,\Gamma_I}(\mathcal{F})$ are coadmissible over $X_\infty(G)\widehat{\otimes}A$, and the corresponding cohomology groups of Iwasawa complex $C^\bullet_{\psi_I} (\mathcal{F})$ are coadmissible over $\Pi_{\mathrm{an},\infty,I,X_\infty(G)\widehat{\otimes} A}(\Gamma_{K_I})$. Moreover when $A=\mathbb{Q}_p$, $C^\bullet_{\psi_I}(\mathcal{F})$ lives in $\mathbb{D}^\flat_{\mathrm{perf}}(\Pi_{\mathrm{an},\infty,I,X_\infty(G)\widehat{\otimes} A}(\Gamma_{K_I}))$ and $C^\bullet_{\varphi_I,\Gamma_I}(\mathcal{F})$ lives in $\mathbb{D}^\flat_{\mathrm{perf}}(X_\infty(G)\widehat{\otimes} A)$.
\end{theorem}

\begin{remark}
The two theorems are around commutative period rings. Actually we conjectured that the corresponding analogs in the noncommutative setting should hold. In other words not only we would expect that one can follow the ideas in this work to establish some finiteness theorems when we have the coefficients that are noetherian noncommutative Banach rings as those discussed in \cref{section5.1}, but also we should as well have some analogs of the results on the (noncommutative) equivariant Iwasawa cohomology. For more detail, see \cref{conjecture5.22}, \cref{conjecture5.23}, \cref{conjecture5.24}. 
\end{remark}

\indent It is actually not expected (due to some difficulty in high dimensional $p$-adic analysis) that one could derive the finiteness or perfectness in some ambient derived categories by using analogue of Kedlaya's key ingredient to prove the theorem mentioned in setting 1.2 above, which is to say there is no d\'evissage possibility transparent for us to use in order to reduce all the proof in the non-\'etale setting to the main results of \cite{CKZ18} and \cite{PZ19}, in the \'etale setting. The coherence of the cohomology is significant in our theory, for instance if one would like to study the global triangulation of $(\varphi,\Gamma)$-modules over any rigid analytic spaces, for more application see the discussion below.\\

\indent Now we discuss the applications of our results. The multidimensionalization represented in the program of \cite{PZ19} and \cite{CKZ18} is aimed at some motivation from $p$-adic Local Langlands program and relative Fukaya-Kato-Nakamura's local Tamagawa number conjecture or the $\varepsilon$-isomorphism conjecture. One could somehow regard the two conjectures as closely related to each other to some extent. Roughly speaking we have the following conjecture:

\begin{conjecture} \mbox{\bf (Nakamura \cite{Nak2})} In the determinant category attached to any $(\varphi,\Gamma)$-modules over relative Robba rings, for each such $M$ there exists an isomorphism between the object $\mathbf{1}$ with the fundamental line attached to $M$. Moreover this isomorphism is compatible with algebraic functional equation, duality and the de Rham isomorphism.
\end{conjecture}

The program initiated in \cite{PZ19} and \cite{CKZ18} gives rise to some possible approach to this conjecture, by first generalizing this to higher dimensional situation, and then consider some sort of diagonal embedding in the group theoretic consideration. Again there should be a corresponding relative version of the conjectures of Pal-Z\'abr\'adi \cite{PZ19}. \\

%
%
%
%
%
%
%
%
%
%
%
%
%

\indent Another application will be deformation theory and the variation of the triangulation property represented already in the one-dimensional situation, see \cite{KPX}. As mentioned above, local triangulation or global triangulation are significant properties in the study of Galois theoretic eigenvarieties. We should mention that actually our approach which is parallel to \cite{KPX} is different from the one in \cite{Liu1}. The latter used essentially some local spreading properties of the slope filtrations, which is tested by considering the sensibility of the neighbourhood around a good point on the slope filtrations. All of these are not directly expected in the context of \cite{KPX} and the one we are considering in this paper. For the extended and detailed exposition, see the corresponding \cref{triangulation}.\\

\subsection{What's Next}

The framework we studied here is just some partial aspects of the whole story in the higher dimensional analytic geometry. One can definitely study more general analytic geometry even in the archimedean setting where one can search for more very typical multi-Hodge structures such as those coming from the hyper-K\"ahler geometry. We have only worked over polydisks and polyannuli in the strictly affinoid setting in regular rigid analytic geometry which we believe of course should be the most similar setting to the classical complex analysis of several variables. But that being said, these restrictions are not necessary at all. One can obviously ask how general the spaces which we could handle are. We certainly believe in some very general sense that the spaces we could potentially handle are those allowing us to do the suitable $p$-adic functional analysis and $p$-adic complex analysis of several variables.\\

The corresponding construction we established here targets directly at the noncommutative Tamagawa number conjectures. We should then follow \cite{Nak1}, \cite{Nak2}, \cite{Zah1}, \cite{PZ19} to amplify our programs in our mind from the philosophy in \cite{BF1}, \cite{BF2} and \cite{FK}. And we should believe that our approach could be also  applied in many other contexts where \cite{Kie1} and \cite{KL3} will provide insights and applications, which certainly might not have to be Iwasawa theoretic. Certainly in the context of \cite{PZ19} one still has to investigate in more detail the corresponding $p$-adic Hodge theoretic properties of objects studied in this paper, such as the corresponding possible analogs of the main results in \cite{Ked1}, \cite{And1}, \cite{Me} and \cite{Ber1}.\\

\newpage 

\begin{center}
\begin{tabularx}{\linewidth}{lX}
Notation & Description (These are mainly starting participating in the discussion in Chapter 2 in the commutative setting, while in Chapter 5 in the noncommutative setting)\\
\hline
$p$    & A prime number which is $>2$.\\
$I$    & A finite set. \\
$K_\alpha$   & Finite extension of $\mathbb{Q}_p$ for each $\alpha\in I$, with chosen $\pi_{K,\alpha}$.\\
$G_{K_{\alpha}}$ & The corresponding Galois group of $\mathrm{Gal}(\overline{\mathbb{Q}}_p/K_\alpha)$ for each $\alpha\in I$.\\
$\Gamma_{K_{\alpha}}$ & The corresponding $\Gamma$ group of $K_\alpha$ for each $\alpha\in I$.\\
$s_I,r_I$ & Multidimensional radii $0<s_I\leq r_I$ with coordinates which are rational numbers.\\
$A$  & An affinoid algebra (which could be noncommutative).\\
$\Pi_{r_I,I}$ & Multidimensional convergent ring with respect to the radius $r_I$.\\
$\Pi_{[s_I,r_I],I}$ & Multidimensional analytic ring with respect to the radii in $[s_I,r_I]$.\\
$\Pi_{\mathrm{an},r_I,I}$ & Multidimensional analytic ring with respect to the radius $r_{I}$.\\
$\Pi_{\mathrm{an},\mathrm{con},I}$ & Multidimensional Robba rings.\\
$\Pi_{r_I,I,A}$ & Multidimensional convergent ring with respect to the radius $r_I$ with coefficients in an affinoid algebra $A$.\\
$\Pi_{[s_I,r_I],I,A}$ & Multidimensional analytic ring with respect to the radii in $[s_I,r_I]$ with coefficients in an affinoid algebra $A$.\\
$\Pi_{\mathrm{an},r_I,I,A}$ & Multidimensional analytic ring with respect to the radius $r_{I}$ with coefficients in an affinoid algebra $A$.\\
$\Pi_{\mathrm{an},\mathrm{con},I,A}$ & Multidimensional Robba rings with coefficients in an affinoid algebra $A$.\\

$(C_\alpha,l_\alpha)_{\alpha\in I}$  & Collections of pair $(C_\alpha,l_\alpha)$ each of which consists of a multiplicative topological group which is abelian with finite torsion part and free part isomorphic to $\mathbb{Z}_p$, and a continuous homomorphism from $C_\alpha$ to $\mathbb{Q}_p$ for each $\alpha\in I$.  \\
$\Pi_{\mathrm{an},r_I,I}(C_I)$ & Multidimensional analytic ring with respect to the radius $r_{I}$, with new variables from $C_I$.\\
$\Pi_{\mathrm{an},\mathrm{con},I}(C_I)$ & Multidimensional Robba rings, with new variables from $C_I$.\\ 	
$\Pi_{\mathrm{an},r_I,I,A}(C_I)$ & Multidimensional analytic ring with respect to the radius $r_{I}$ with coefficients in an affinoid algebra $A$, with new variables from $C_I$.\\
$\Pi_{\mathrm{an},\mathrm{con},I,A}(C_I)$ & Multidimensional Robba rings with coefficients in an affinoid algebra $A$, with new variables from $C_I$.\\ 	
$\Pi_{\mathrm{an},\infty,I}(C_I)$ & Multidimensional analytic ring with respect to the radius $r_{I}$, with new variables from $C_I$.\\
$\Pi_{\mathrm{an},\infty,I,A}(C_I)$ & Multidimensional analytic ring with respect to the radius $r_{I}$ with coefficients in an affinoid algebra $A$, with new variables from $C_I$.\\

\end{tabularx}
\end{center}

\begin{center}
\begin{tabularx}{\linewidth}{lX}
Notation & Description (These are mainly starting participating in the discussion in Chapter 3 and 4 in the commutative setting, while in Chapter 5 in the noncommutative setting)\\
\hline

$\Pi_{\mathrm{an},r_I,I}(\pi_{K,I})$ & Multidimensional analytic ring with respect to the radius $r_{I}$, with new variables in $\pi_{K,I}$.\\
$\Pi_{[s_I,r_I],I}(\pi_{K,I})$ & Multidimensional analytic ring with respect to the radii in $[s_I,r_{I}]$, with new variables in $\pi_{K,I}$.\\
$\Pi_{\mathrm{an},\mathrm{con},I}(\pi_{K,I})$ & Multidimensional Robba rings, with new variables in $\pi_{K,I}$.\\ 	
$\Pi_{\mathrm{an},r_I,I,A}(\pi_{K,I})$ & Multidimensional analytic ring with respect to the radius $r_{I}$ with coefficients in an affinoid algebra $A$, with new variables in $\pi_{K,I}$.\\
$\Pi_{[s_I,r_I],I,A}(\pi_{K,I})$ & Multidimensional analytic ring with respect to the radii in $[s_I,r_{I}]$ with coefficients in an affinoid algebra $A$, with new variables in $\pi_{K,I}$.\\
$\Pi_{\mathrm{an},\mathrm{con},I,A}(\pi_{K,I})$ & Multidimensional Robba rings with coefficients in an affinoid algebra $A$, with new variables in $\pi_{K,I}$.\\

$\Pi_{\mathrm{an},\infty,I}(\Gamma_{K,I})$ & Multidimensional analytic ring with respect to the radius $r_{I}$, with new variables from $\Gamma_K$.\\

$\Pi_{\mathrm{an},\infty,I,A}(\Gamma_{K,I})$ & Multidimensional analytic ring with respect to the radius $r_{I}$ with coefficients in an affinoid algebra $A$, with new variables from $\Gamma_K$.\\

\end{tabularx}
\end{center}

\newpage

\section{Fr\'echet Rings and Sheaves in the Commutative Case}

\subsection{Fr\'echet Objects}

\indent In this subsection we are going to define the main objects we are going to study ring theoretically. Everything will be modeling the one-dimensional situation. As indicated in the title of the section, we will consider Fr\'echet-Stein objects in the sense of \cite{ST1}. We first proceed by defining the following rings:

\begin{definition} \mbox{\bf{(After KPX \cite[Notation 2.1.1]{KPX})}}
We consider the multi radii as mentioned in the notation list, where $0<s_I\leq r_I$ are all rational numbers or $\infty$. Under this assumption first let $I$ is a finite set. Then we consider the multidimensional polynomial rings generalizing the one-dimensional situation with variables $T_1^{\pm 1},..,T_I^{\pm 1}$ which is to say the ring:
\begin{displaymath}
\mathbb{Q}_p[T_1^{\pm 1},..,T_I^{\pm 1}].	
\end{displaymath}
Now we endow this rings with some specific Gauss valuations and norms as in the following way for any multi radii $t_I>0$:
\begin{displaymath}
v_{t_I}(f)=v_{t_I}(\sum_{n_I}a_{n_I}T_I^{i_I}):= \inf_{n_I}\left\{v_p(a_{n_I})+\sum_{\alpha\in I}t_\alpha i_\alpha\right\},
\end{displaymath}
\begin{displaymath}
\left\|f\right\|_{t_I}=\left\|\sum_{n_I}a_{n_I}T_I^{i_I}\right\|_{t_I}:= \sup_{n_I}\left\{|a_{n_I}|p^{-\sum_{\alpha\in I}t_\alpha i_\alpha/(p-1)}\right\}.
\end{displaymath}
We are going to use the Gauss valuations and norms with respect to the multi-interval $[s_I,r_I]$ in the following sense:
\begin{displaymath}
v_{[s_I,r_I]}(f)=v_{[s_I,r_I]}(\sum_{n_I}a_{n_I}T_I^{i_I}):= \min_{t_\alpha\in\{s_I,r_I\}}\left\{v_p(a_{n_I})+\sum_{\alpha\in I}t_\alpha i_\alpha\right\},
\end{displaymath}
\begin{displaymath}
\left\|f\right\|_{[s_I,r_I]}=\left\|\sum_{n_I}a_{n_I}T_I^{i_I}\right\|_{[s_I,r_I]}:= \max_{t_\alpha\in\{s_I,r_I\}}\left\|\sum_{n_I}a_{n_I}T_I^{i_I}\right\|_{t_\alpha}.	
\end{displaymath}
The latter will give rise to the following definition:
\begin{displaymath}
\Pi_{[s_I,r_I],I}	
\end{displaymath}
which is defined to the completion of the ring $\mathbb{Q}_p[T_1^{\pm 1},..,T_I^{\pm 1}]$ with respect to the norms defined above with respect to the multi-interval $[s_I,r_I]$. For the polynomial rings:
\begin{displaymath}
\mathbb{Q}_p[T_1,..,T_I]	
\end{displaymath}
we then instead define for some $t_I>0$:

\begin{displaymath}
v_{t_I}(f)=v_{t_I}(\sum_{n_I}a_{n_I}T_I^{i_I}):= \inf_{n_I}\left\{v_p(a_{n_I})+\sum_{\alpha\in I}t_\alpha i_\alpha\right\},
\end{displaymath}
\begin{displaymath}
\left\|f\right\|_{t_I}=\left\|\sum_{n_I}a_{n_I}T_I^{i_I}\right\|_{t_I}:= \sup_{n_I}\left\{|a_{n_I}|p^{-\sum_{\alpha\in I}t_\alpha i_\alpha/(p-1)}\right\}.
\end{displaymath}
This will then give rise to the rings corresponding to the closed polydiscs and open polydiscs:
\begin{displaymath}
\Pi_{[t_I,\infty_I],I}, \Pi_{\mathrm{an},\infty_I,I}:=\bigcap_{t_I>0}\Pi_{[t_I,\infty_I],I}.	
\end{displaymath}
Then correspondingly we define the following analytic rings corresponding to the full polyannuli:
\begin{displaymath}
\Pi_{\mathrm{an},r_I,I}:=\bigcap_{0<s_I\leq r_I} \Pi_{[s_I,r_I],I},\Pi_{\mathrm{an},\mathrm{con},I}:=\bigcup_{r_I>0}\bigcap_{0<s_I\leq r_I} \Pi_{[s_I,r_I],I}. 
\end{displaymath}

\end{definition}

\begin{remark}
One could actually define the rings by considering all the (rigid) analytic functions over the corresponding open polydiscs, closed polydiscs, polyannuli and full polyannuli, which will be more transparent.
\end{remark}

\begin{remark}
Actually one could see that in the definition, one could find that especially in the multi-interval situation, the rightmost radii could be mixed-type with respect to the infinity which is to say that one could have a genuine proper subset of indexes where the rightmost radii are exactly infinite. We will mainly focus on the situation where such proper subset does not emerge in our discussion later, if not specified otherwise.	
\end{remark}

\indent Now we consider the relative version of the definition as in the one dimensional situation in \cite[Notation 2.1.1]{KPX}.

\begin{definition} \mbox{\bf{(After KPX \cite[Notation 2.1.1]{KPX})}}
Let $A$ be an affinoid algebra over $\mathbb{Q}_p$ in the sense of rigid analytic spaces. Then we are going to define the relative version of the rings defined as above in the absolute situation:
\begin{displaymath}
\Pi_{[s_I,r_I],I,A}, 
\end{displaymath}
as the analytic functions relative to $A$ over the products:
\begin{displaymath}
\mathrm{Max}A\times A^I[s_I,r_I],0<s_I\leq r_I.
\end{displaymath}
Then we define:
\begin{displaymath}
\Pi_{\mathrm{an},r_I,I,A}:=\bigcap_{0<s_I\leq r_I} \Pi_{[s_I,r_I],I,A},\Pi_{\mathrm{an},\mathrm{con},I,A}:=\bigcup_{r_I>0}\bigcap_{0<s_I\leq r_I} \Pi_{[s_I,r_I],I,A}.	
\end{displaymath}
Correspondingly almost in the parallel and similar way one could define the following ring of analytic functions relative to the space $A$:
\begin{displaymath}
\Pi_{[s_I,\infty_I],I,A},\Pi_{\mathrm{an},\infty_I,I,A}:=\bigcap_{s_I>0}\Pi_{[s_I,\infty_I],I,A}.	
\end{displaymath}
\end{definition}

\indent Eventually we are going to work within the framework of Fr\'echet-Stein algebras over rigid analytic spaces. Now we are going to consider the construction over the affinoid algebras. First let $r_{I,0}$ be a fixed radii (where we do not consider the situation that there is a finite proper set of $I$ where the radii approach the infinity). Then we consider a decreasing sequence of multi radii $\{r_{I,n}\}_{n\geq 0}$ which is to say that for each $\alpha\in I$, $r_{\alpha,n+1}\leq r_{\alpha,n}$, and when $n=0$ the radii is the one fixed above. And we require that for each $\alpha\in I$, $r_{\alpha,1}< r_{\alpha,0}$.

\begin{remark}
Actually this looks a little bit strange, but one could eventually choose in the following way that for each $n\geq 1$ one just set $r_{\alpha,n}=r_n$ for each $\alpha$.	
\end{remark}

\indent Then we consider the following result which in some non-trivial way generalizes the one dimensional situation:

\begin{proposition}
For all $n\geq 1$, the rings $\Pi_{[r_{I,n},r_{I,0}],I,A}$ and $\Pi_{[r_{I,n},\infty_I],I,A}$ are noetherian as Banach $A$-algebras.	
\end{proposition}

\begin{proof}
We are going to follow ideas in \cite{PZ19} and \cite{Ked9} to prove this. This will be some application of a theorem due to Fulton in \cite{Ful1}. We are going to show the statement in the situation where $A$ is just the field $\mathbb{Q}_p$, and leave the general situation to the reader. We are going to only choose to prove the statement for the ring $\Pi_{[r_{I,n},r_{I,0}],I,A}$, and we leave the check for the other ring to the readers. We are going to show that the ring $\Pi_{[r_{I,n},r_{I,0}],I,A}$ is weakly complete algebras over the corresponding adic completion of the ring $\mathbb{Z_p}[T_\alpha,\alpha\in I]$ with the ideal $(p)$ generated by the element $p$ , and a theorem due to Fulton will imply that the ring $\Pi_{[r_{I,n},r_{I,0}],I,A}$ is actually noetherian.\\
\indent Following Monsky-Washnitzer, we recall as in the following the basic definition of a ring $R$ being relatively weakly complete algebras over $A$. This means that for any element $f\in R$ and any degree $d\geq 0$, one could find $k$ elements $t_1,...,t_k$ and a polynomial $p_d$ with $k$-variables with the coefficients in $(p)^d$, such that there exists a constant $C>0$ with $\mathrm{deg}p_d\leq C(d+1)$ and with that:
\begin{displaymath}
f=\sum_{d\geq 0}p_d(t_1,...,t_k).	
\end{displaymath}
Therefore we are going to find some essential estimate over the degree of the polynomial with respect to some given element $f$. To get start for any such given $f$ and any given $d$ in the condition recalled as above, we consider a degree $\partial_d$ positive which is maximal in the sense that $f$ is inside $T_I^{-\partial_d}\mathbb{Z}_p/(p)^d$ and the index is the maximal throughout all such indexes. Then we are going to use the notation $\partial_I$ to denote the corresponding multi-index (which is possibly not the same as $-\partial_d^I$). Now we take a constant $c_1>1$ such that $\left\|f\right\|_{[r_{I,n},r_{I,0}]}\leq c_1$. Then we compute:
\begin{displaymath}
c_1\geq \left\|f\right\|_{[r_{I,n},r_{I,0}]} \geq \left\|f\right\|_{r_{I,0}} \geq |\mathrm{coeff}_{\partial_I}(f)|p^{\partial_d r_{\alpha,0}/(p-1)}\prod_{\beta\in I\backslash \{\alpha\}}p^{-r_{\beta,0} i_\beta/(p-1)}
\end{displaymath}
which gives rise to that for some $c_2>0$:
\begin{displaymath}
|\mathrm{coeff}_{\partial_I}(f)|p^{\partial_d r_{\alpha,0}/(p-1)}\leq c_2.	
\end{displaymath}
Then we compute as in the following:
\begin{align}
|\mathrm{coeff}_{\partial_I}(f)|p^{\partial_d r_{\alpha,0}/(p-1)}&\leq c_2\\	
\mathrm{log}_p|\mathrm{coeff}_{\partial_I}(f)|+\partial_d r_{\alpha,0}/(p-1)&\leq \mathrm{log}_pc_2\\
(1-d)+\partial_d r_{\alpha,0}/(p-1)&\leq \mathrm{log}_pc_2\\
\partial_d r_{\alpha,0}&\leq (p-1)\mathrm{log}_pc_2+(d-1)(p-1)\\
\partial_d &\leq (p-1)\mathrm{log}_pc_2/r_{\alpha,0}+(d-1)(p-1)/r_{\alpha,0}
\end{align}
which gives rise to the desired estimate, which then further finishes the proof by considering Fulton's theorem in \cite{Ful1} which says that the desired relative weakly complete algebra is now noetherian.
\end{proof}

 \indent Then we consider the following statement which sends us to the framework of \cite{ST1}, which generalizes the one dimensional situation in \cite[Section 2.1]{KPX}:
 
\begin{proposition}
For all $n\geq 1$, the rings $\Pi_{\mathrm{an},r_{I,0},I,A}$ and $\Pi_{\mathrm{an},\infty_I,I,A}$ are Fr\'echet-Stein in the sense of \cite{ST1}.
\end{proposition}

\begin{proof}
We keep the notations simple so we just consider the proof for the first ring and leave the task for checking the second case to the reader. Now indeed, we first note that $\Pi_{[r_{I,n},r_{I,0}],I,A}$ is Fr\'echet algebra in the sense that it is the ring of the analytic functions over a union of the relative annuli as in the following:
\begin{displaymath}
\bigcup_{0<r_{I,n}\leq r_{I,0}} \mathrm{Max}A\times A^I[r_{I,n},r_{I,0}]. 
\end{displaymath}
Furthermore we note that this is also the increasing union which is of the form of:
\begin{displaymath}
\bigcup_{0<r_{I,n}\leq r_{I,0}} \mathrm{Max}A\times A^I[r_{I,n},r_{I,0}],
\end{displaymath}
where $ \mathrm{Max}A\times A^I[r_{I,n},r_{I,0}]\subset  \mathrm{Max}A\times A^I[r_{I,n+1},r_{I,0}]$. The corresponding transition map over the algebras then with respect to this increasing union is actually flat since it is \'etale by the production construction and the base change properties. Finally the transition map is of topologically dense image again by the product construction.	
\end{proof}

\indent Before we proceed to study more in more detail on the representation and structure of the algebras defined above, we first define the vector bundles over polyannuli generalizing the one dimensional result in \cite[Definition 2.1.3]{KPX}. We mimick the definition, which gives rise to the following:

\begin{definition} \mbox{\bf{(After KPX \cite[Definition 2.1.3]{KPX})}}
We define coherent sheaves of modules and the global sections over the ring $\Pi_{\mathrm{an},r_{I_0},I,A}$ in the following way. First a coherent sheaf $\mathcal{M}$ over the relative multidimensional Robba ring $\Pi_{\mathrm{an},r_{I_0},I,A}$ is defined to be a family $(M_{[s_I,r_I]})_{[s_I,r_I]\subset (0,r_{I,0}]}$	of modules of finite type over each relative $\Pi_{[s_I,r_I],I,A}$ satisfying the following two conditions as in the more classical situation:
\begin{displaymath}
M_{[s_I',r_I']}\otimes_{\Pi_{[s'_I,r'_I],I,A}}\Pi_{[s_I'',r_I''],I,A}\overset{\sim}{\rightarrow} M_{[s''_I,r''_I]}	
\end{displaymath}
for any multi radii satisfying $0<s_I'\leq s_I''\leq r_I''\leq r_I'\leq r_{I,0}$, with furthermore 
\begin{displaymath}
(M_{[s'_I,r'_I]}\otimes_{\Pi_{[s'_I,r'_I],I,A}}\Pi_{[s''_I,r''_I],I,A}	)\otimes_{\Pi_{[s''_I,r''_I],I,A}}\Pi_{[s'''_I,r'''_I],I,A}\overset{\sim}{\rightarrow}\Pi_{[s'''_I,r'''_I],I,A},
\end{displaymath}
for any multi radii satisfying the following condition:
\begin{displaymath}
0<s_I'\leq s_I''\leq s_I'''\leq r_I'''\leq r_I''\leq r_I'\leq r_{I,0}.	
\end{displaymath}
Then we define the corresponding global section of the coherent sheaf $\mathcal{M}$, which is usually denoted by $M$ which is defined to be the following inverse limit:
\begin{displaymath}
M:=\varprojlim_{s_I\rightarrow 0_I^+} M_{[s_I,r_{I,0}]}.	
\end{displaymath}
We are going to follow \cite[Definition 2.1.3]{KPX} to call any module defined over $\Pi_{\mathrm{an},r_{I_0},I,A}$ coadmissible it comes from a global section of a coherent sheaf $\mathcal{M}$ in the sense defined above. 
\end{definition}

\indent Then we could collect the following properties form more general theory from \cite{ST1} as in \cite[Lemma 2.1.4]{KPX} in the one dimensional situation.

\begin{proposition} \mbox{\bf{(After KPX \cite[Lemma 2.1.4]{KPX})}} \label{prop2.9}
We have the following properties for the coherent sheaves defined above over the relative multi-dimensional Robba rings $\Pi_{\mathrm{an},r_{I_0},I,A}$:\\
I.	For each multi radii $0< s_I\leq r_I\leq r_{I,0}$ as above, we have that the global section $M$ is dense in the corresponding section $M_{[s_I,r_I]}$ where $r_{I,0}$ is finite, where the statement is also true for $M[1/T_I]$ for $r_{I,0}$ infinite but with $r_I$ not;\\
II. For each multi radii $0< s_I\leq r_I\leq r_{I,0}$ as above, we have the following comparison:
\begin{displaymath}
M\otimes_{\Pi_{\mathrm{an},r_{I_0},I,A}}\Pi_{[s_I,r_I],I,A}\overset{\sim}{\rightarrow}	M_{[s_I,r_I]};\\
\end{displaymath}
III. For each multi radii $0< s_I \leq r_{I,0}$ we have the following higher vanishing result:
\begin{displaymath}
R^1\varprojlim_{s_I\rightarrow 0^+}M_{[s_I,r_{I,0}]}=0;\\	
\end{displaymath}
IV. The kernel or cokernel of a morphism between coadmissible modules over $\Pi_{\mathrm{an},r_{I_0},I,A}$ is again coadmissible;\\
V. Any finite presented module over $\Pi_{\mathrm{an},r_{I_0},I,A}$ is coadmissible;\\
VI. Any submodule of finite type of any coadmissible  $\Pi_{\mathrm{an},r_{I_0},I,A}$-module is again coadmissible;\\ 
VII. For each multi radii $0< s_I\leq r_I\leq r_{I,0} $, and for the global section defined above, we have that $M\otimes_{\Pi_{\mathrm{an},r_{I_0},I,A}}\Pi_{[s_I,r_I],I,A}$ is flat over $\Pi_{[s_I,r_I],I,A}$. 

\end{proposition}

\begin{proof}
See \cite[Lemma 2.1.4]{KPX} and \cite{ST1}. Again let us mention that actually in this setting initially we only have the corresponding results for $r_I$ fixed which could for instance be $r_{I,0}$, however this could be overcome by the same method as in the one dimensional situation in \cite[Lemma 2.1.4]{KPX}.
\end{proof}

\indent We then follow the ideas and approaches from \cite{KPX} to 
study further the structure theory of the coherent sheaves defined above. The following is then a generalization of the corresponding result in the one-dimensional situation in \cite[Corollary 2.1.5]{KPX}.

\begin{lemma} \mbox{\bf{(After KPX \cite[Corollary 2.1.5]{KPX})}}
The rings $\Pi_{\mathrm{an},r_{I_0},I,A}$, and $\Pi_{\mathrm{an},\mathrm{con},I,A}$ are flat over $A$ as in \cite[Corollary 2.1.5]{KPX}.	
\end{lemma}

\begin{proof}
We adapt the corresponding method of proof in \cite[Corollary 2.1.5]{KPX} to prove this generalization. Indeed one could then prove this by the same method of \cite[Corollary 2.1.5]{KPX} since then by taking a suitable basis we will have $\Pi_{\mathrm{an},r_{I_0},I,\mathbb{Q}_p}\overset{\simeq}{\rightarrow}\widehat{\oplus}_{\eta}\mathbb{Q}_pe_\eta$. Then one could just finish as in \cite[Corollary 2.1.5]{KPX}. 	
\end{proof}

\begin{definition} \mbox{\bf{(After KPX \cite[Definition 2.1.3]{KPX})}}
Consider a coherent sheaf $M_{r_{I,0}}$ over $\Pi_{\mathrm{an},r_{I_0},I,A}$ defined as above, then we are going to call this sheaf a vector bundle if for each $[s_I,r_I]$ suitably located multi interval we have $M_{[s_I,r_I]}$ is flat over $\Pi_{\mathrm{an},[s_I,r_I],I,A}$.	
\end{definition}

\indent This will give us a chance to consider the following result:

\begin{proposition} \mbox{\bf{(After KPX \cite[Lemma 2.1.6]{KPX})}} \label{prop2.12}
For any coherent sheaf $M_{r_{I,0}}$ over $\Pi_{\mathrm{an},r_{I_0},I,A}$ which admits a vector bundle structure in our situation, we then have that the global section $M$ of it is finitely projective if and only if it is finitely presented. (Actually one can even derive the corresponding sufficient and necessary condition to be just being finitely generated by applying \cite[Corollary 2.6.8]{KL2}). 	
\end{proposition}

\begin{proof}
Basic representation theory tells us that the module involved is finitely projective if and only if it is finitely presented and flat. Therefore back to our proposition one direction of the implications is trivial, which is to say that we could now assume that the corresponding module is finitely presented, and it suffices to show that it is flat. Therefore as in \cite[Lemma 2.1.6]{KPX} we consider the map $I\otimes_{\Pi_{\mathrm{an},r_{I_0},I,A}} M\rightarrow M$ for some finitely generated ideal $I$. Then we are going to show that this is injective. First choose some presentation of $M$ which is of type $(m,n)$ over the base ring $\Pi_{\mathrm{an},r_{I_0},I,A}$ then translate this to the corresponding statement for the ideal $I$ which is to say $I^m\rightarrow I^n$. This actually gives us a chance to derive the statement that $I\otimes_{\Pi_{\mathrm{an},r_{I_0},I,A}} M$ is also coadmissible by the \cref{prop2.9}. Now again by \cref{prop2.9} we have that it suffices to show that:
\begin{displaymath}
I\otimes_{\Pi_{\mathrm{an},r_{I_0},I,A}} M\rightarrow M	
\end{displaymath}
is injective over some base change to some $\Pi_{\mathrm{an},[s_I,r_I],I,A} $. Then as in \cite[Lemma 2.1.6]{KPX} one could finish the proof since this is nothing but the following map:
\begin{displaymath}
(I\otimes_{\Pi_{\mathrm{an},r_{I_0},I,A}} \Pi_{\mathrm{an},[s_I,r_I],I,A})\otimes  M_{[s_I,r_I]} \rightarrow M_{[s_I,r_I]}	
\end{displaymath}
which is injective by the flatness from \cite[Lemma 2.1.6]{KPX}.
\end{proof}

\indent We then have the following corollary which is a direct analog of the corresponding results in \cite[Corollary 2.1.7]{KPX}:

\begin{corollary}
Let $A$ be an affinoid algebra which is reduced. Now consider a module $M$ over the Robba ring with respect to a specific radius $r_I$, defined over $\mathrm{Max}(A)\times \mathbb{A}^n(0,r_I]$ and we assume that for any maximal ideal $\mathfrak{m}_\eta$ of $A$ the fiber $M_{\mathfrak{m}_\eta}$ has the same rank. Then we have that $M$ is finite projective if we further assume that $M$ is finitely presented (as mentioned in the previous proposition one can even assume that it is finitely generated).	
\end{corollary}

\begin{proof}
This is a corollary of the previous proposition and the corresponding noetherian property as in \cite[Corollary 2.1.7]{KPX}.	
\end{proof}

\indent Before we study the relationship between the finiteness of the global sections and the coherent sheaves, we generalize the uniform finiteness in the one dimensional situation which is established in \cite[Definition 2.1.9]{KPX}.

\begin{definition} \mbox{\bf{(After KPX \cite[Definition 2.1.9]{KPX})}} 
We first generalize the notion of admissible covering in our setting, which is higher dimensional generalization of the situation established in \cite[Definition 2.1.9]{KPX}. For any covering $\{[s_I,r_I]\}$ of $(0,r_{I,0}]$, we are going to specify those admissible coverings if the given covering admits refinement finite locally (and each corresponding member in the covering has the corresponding interior part which is assumed to be not empty with respect to each $\alpha\in I$). Then it is very natural in our setting that we have the corresponding notations of $(m,n)$-finitely presentedness and $n$-finitely generatedness for any $m,n$ positive integers. Then based on these definitions we have the following generalization of the corresponding uniform finiteness. First we are going to call a coherent sheaf $\{M_{s_I,r_I}\}_{[s_I,r_I]}$ over $\Pi_{\mathrm{an},r_{I,0},I,A}$ uniformly $(m,n)$-finitely presented if there exists an admissible covering $\{[s_I,r_I]\}$ and a pair of positive integers $(m,n)$ such that each module $M_{[s_I,r_I]}$ defined over $\Pi_{\mathrm{an},[s_I,r_I],I,A}$ is $(m,n)$-finitely presented. Also we have the notion of uniformly $n$-finitely generated for any positive integer $n$ by defining that to mean under the existence of an admissible covering we have that each member $M_{[s_I,r_I]}$ in the family defined over $\Pi_{\mathrm{an},[s_I,r_I],I,A}$ is $n$-finitely generated. Sometimes we are also going to use the notions of being uniformly finitely generated and being uniformly finitely presented to mean the corresponding objects when the corresponding uniform numbers of the generators are well-understood.
\end{definition}

\indent After defining these notions one could consider the following generalization of the corresponding results in \cite[Lemma 2.1.10]{KPX} specifically in the one-dimensional situation. Before establishing this kind of generalization we first consider the following:

\begin{lemma} \mbox{\bf{(After KPX \cite[Lemma 2.1.10]{KPX})}} \label{lemma2.15}
Consider now the following short exact sequence of coherent sheaves over $\Pi_{\mathrm{an},r_{I,0},I,A}$:
\[
\xymatrix@R+2pc@C+0pc{
0\ar[r] \ar[r] &(M^1_{[s_I,r_I]})_{\{[s_I,r_I]\}} \ar[r] \ar[r] &(M_{[s_I,r_I]})_{\{[s_I,r_I]\}} \ar[r] \ar[r] &(M^2_{[s_I,r_I]})_{\{[s_I,r_I]\}} \ar[r] \ar[r] &0.} 
\]
Then we have that if $(M^1_{[s_I,r_I]})_{\{[s_I,r_I]\}}$ and one of the rest two sheaves are uniformly finitely presented so is the third, moreover we have that if $(M_{[s_I,r_I]})_{\{[s_I,r_I]\}},(M^2_{[s_I,r_I]})_{\{[s_I,r_I]\}}$ are uniformly finitely presented then the third one will be uniformly finitely generated.
\end{lemma}

\begin{proof}
See \cite[Lemma 2.1.10]{KPX}. 	
\end{proof}


\begin{lemma}  \mbox{\bf{(After KPX \cite[Lemma 2.1.11]{KPX})}}
Let $\{M_{[s_I,r_I]}\}$	be a coherent sheaf over $\Pi_{\mathrm{an},r_{I,0},I,A}$ with global section which we will denote it by $M$ (with respect to an admissible covering in our context). Suppose that we have a set of generators $\{\mathbf{e}_1,...,\mathbf{e}_n\}$ which generates $M_{[s_I,r_I]}$ for each $[s_I,r_I]$ involved, then we have that this set of generators will generate $M$ as well.

\end{lemma}

\begin{proof}
As in the one dimensional situation (see \cite[Lemma 2.1.11]{KPX}) one could prove this by first exhibiting a map by using the given basis $(\Pi_{\mathrm{an},r_{I,0},I,A})^n\rightarrow M$, then arguing that this will give rise to the desired generating set by considering the corresponding projection which is again a well-defined surjection and comparison under the assumption mentioned above.
\end{proof}

\begin{proposition} \mbox{\bf{(After KPX \cite[Proposition 2.1.13]{KPX})}} \label{prop1}
Consider any arbitrary coherent sheaf $\{M_{[s_I,r_I]}\}$ over $\Pi_{\mathrm{an},r_{I,0},I,A}$ with the global section $M$. Then we have the following corresponding statements:\\
I.	The coherent sheaf $\{M_{[s_I,r_I]}\}$ is uniformly finitely generated iff the global section $M$ is finitely generated;\\
II. The coherent sheaf $\{M_{[s_I,r_I]}\}$ is uniformly finitely generated iff the global section $M$ is finitely presented;\\
III. The coherent sheaf $\{M_{[s_I,r_I]}\}$ is uniformly finitely presented vector bundle iff the global section $M$ is finite projective.
\end{proposition}

\begin{proof}
One could derive the the second and the third statements by using the \cref{lemma2.15} and \cref{prop2.12}. Therefore it suffices for us to prove the first statement. Indeed one could see that one direction in the statement is straightforward, therefore it suffices now to consider the other direction. So now we assume that the coherent sheaf $\{M_{[s_I,r_I]}\}$ is uniformly finitely generated for some $n$ for instance. Then we are going to show that one could find corresponding generators for the global section $M$ from the given uniform finitely generatedness. 

Now by definition we consider an admissible covering taking the form of $\{[s_I,r_I]\}$. In \cite[Proposition 2.1.13]{KPX} the corresponding argument goes by using special decomposition of the covering in some refined way, mainly by reorganizing the initial admissible covering into two single ones, where each of them will be made into the one consisting of those well-established intervals with no overlap. In our setting, we need to use instead the corresponding upgraded nice decomposition of the given admissible covering in our setting, which is extensively discussed in \cite[2.6.14-2.6.17]{KL2}. Namely as in \cite[2.6.14-2.6.17]{KL2} there is a chance for us to extract $2^{|I|}$ families $\{\{[s_{I,\delta_i},r_{I,\delta_i}]\}_{i=0,1,...}\}_{\{1_i\},\{2_i\},...,\{N_i\},}$ ($N=2^{|I|}$) of intervals, where there is a chance to have the situation where for each such family the corresponding intervals in it could be made to be disjoint in pairs. Then we are in the situation of \cite[Proposition 2.6.17]{KL2}, namely whenever we have a $2^{|I|}$-uniform covering, then the uniformly finiteness throughout all the coverings involved will guarantee the corresponding finiteness of the corresponding global section which proves the proposition. To be more precise we make the argument in \cite[Proposition 2.6.17]{KL2} more precise in our setting. First we write the corresponding space $\Lambda^I(0,r_{I,0}]_A$ into the following form:
\begin{displaymath}
\Lambda^I(0,r_{I,0}]_A:=\bigcup_{\alpha\geq 0} \Lambda^I[s_{I,\alpha},r_{I,0}]_A.	
\end{displaymath}
For $\Lambda^I[s_{I,-1},r_{I,0}]_A$ we choose that to be empty. Then for a chosen family $\{[s_{I,\delta_i},r_{I,\delta_i}]\}_{i=0,1,...}$, we use the notation $\{[s_{I,\gamma},r_{I,\gamma}]\}_{\Gamma}$ to denote this specific family where $\Gamma$ is endowed with the corresponding well ordering as in \cite[proposition 2.6.17]{KL2}. Now we consider for each $\gamma$ the corresponding index $\alpha(\gamma)$ such that the intersection between $\Lambda^I[s_{I,\alpha(\gamma)},r_{I,0}]_A$ and $\Lambda^I[s_{I,\gamma},r_{I,\gamma}]_A$ is empty and meanwhile this index is the largest among all the candidates such that the intersection like this is empty. Then we consider the corresponding subspace in the following way:
\begin{displaymath}
B_\gamma:=	\Lambda^I[s_{I,\alpha(\gamma)},r_{I,0}]_A\cup \bigcup_{\gamma'<\gamma,\gamma'\in \Gamma}\Lambda^I[s_{I,\gamma},r_{I,\gamma}]_A.
\end{displaymath}
Then we can see that one can find an element in the section $\mathcal{O}(B_\gamma)$ which by considering the corresponding restriction gives rise to some element which represents an invertible of an element which is topologically nilpotent over the section over $\Lambda^I[s_{I,\gamma},r_{I,\gamma}]_A$ but meanwhile vanishes over the section over $B_\gamma$. Then by the density of the global section one can see that one can upgrade this element to be some element in $\mathcal{O}(\Lambda^I(0,r_{I,0}]_A)$ which which represents an invertible of an element which is topologically nilpotent over the section over $\Lambda^I[s_{I,\gamma},r_{I,\gamma}]_A$ but meanwhile represents an element which is topologically nilpotent over the section over $B_\gamma$. We denote this element by $x_\gamma$. Now we can perform the corresponding iteration process in \cite[proposition 2.6.17]{KL2} to generate the following families of elements in $M$:
\begin{align}
&...,\\
&e_{\gamma,1},...,e_{\gamma,k},\\
&...	
\end{align}
whose limit will be the corresponding finite generating set of the global section. We first start from a generating set of $M_{[s_{I,\gamma},r_{I,\gamma}]}$ such $e_1,...,e_k$, then we are considering will be the sequence defined in the following way:
\begin{displaymath}
e_{\gamma,i}:=e_{\gamma,i}+x_\gamma^\eta e_i,  	
\end{displaymath}
where $\eta$ is sufficiently large integer. This process guarantee that we can have converging limit set of generators. The generators actually generate the corresponding section $M_{[s_{I,\gamma},r_{I,\gamma}]}$ since in our situation we can control the corresponding difference between the restriction of our produced set of elements and the actually generator over such section as in \cite[proposition 2.6.17]{KL2} and \cite[Lemma 2.1.12]{KPX}.

\end{proof}

\subsection{Further Discussion and Residue Pairings}

\indent We collect some further continuation of the discussion in the previous section before we consider the residue pairings in our setting, these are higher dimensional generalization of the corresponding results in \cite{KPX}.

\begin{lemma} \mbox{\bf{(After KPX \cite[Lemma 2.1.16]{KPX})}}
Consider a morphism $f:M\rightarrow N$ of modules over $\Pi_{\mathrm{an},r_{I,0},I,A}$ with kernel and cokernel of finite types. Suppose $f:M\rightarrow N$ is injective over $\Pi_{\mathrm{an},\mathrm{con},I,A}$. Then $f$ is injective over $\Pi_{\mathrm{an},r_{I},I,A}$ for some radii $r_I\leq r_{I,0}$. Furthermore if $f:M\rightarrow N$ is surjective over $\Pi_{\mathrm{an},con,I,A}$, then $f$ is surjective over $\Pi_{\mathrm{an},r_{I},I,A}$ for some radii $r_I\leq r_{I,0}$.  	
\end{lemma}

\begin{proof}
See the proof of \cite[Lemma 2.1.16]{KPX} which is a reinterpretation of the properties of being of finite type of the corresponding kernel and cokernel involved.	
\end{proof}

\begin{proposition} \mbox{\bf{(After KPX \cite[Lemma 2.1.18]{KPX})}}
Suppose now that $M$ is a module defined over $\Pi_{\mathrm{an},r_{I,0},I,A}$ and is assumed to be finite over $A$. Then we have that there exists some multi radii $r_I\leq r_{I,0}$ such that the base change of $M$ to $\Pi_{\mathrm{an},r_{I},I,A}$ vanishes. 	
\end{proposition}

\begin{proof}
This is a higher dimensional generalization of the corresponding results in \cite[Lemma 2.1.18]{KPX}. The argument is parallel, we adapt the argument to our setting as in the following. First we consider the situation where $A$ is Artin, then it is easy to see that the proposition is true. Then more generally we consider the copy $A\left<T_I\right>\subset \Pi_{\mathrm{an},r_{I,0},I,A}$, which allows us to go back to the previous situation by taking the reduction. Note that we still satisfy the condition since then the module is also finite over $A\left<T_I\right>$ since it is so over $A$. Then one could finish as in the one dimensional situation, see \cite[Lemma 2.1.18]{KPX}.	
\end{proof}

\indent Then we define the following higher dimensional analog of the differential $\Omega$ over the one dimensional Robba ring relative to the ring $A$. This amounts to considering a higher dimensional residue formula for a function with $I$ variables.

\begin{definition} \mbox{\bf{(After KPX)}}
We first generalize the definition of the differentials of the Robba rings to higher dimensional situation. First consider our full Robba rings $\Pi_{\mathrm{an},\mathrm{con},I,A}$, then we are going to define the corresponding $n$-th differential of this ring relative to $A$ as the following:
\begin{displaymath}
\Omega_{\Pi_{\mathrm{an},\mathrm{con},I,A}/A}:=\wedge^{|I|} \Omega^1_{\Pi_{\mathrm{an},\mathrm{con},I,A}/A}	
\end{displaymath}
which is nothing but the $|I|$-th differential of the ring $\Pi_{\mathrm{an},\mathrm{con},I,A}$ with respect to the affinoid $A$. Here the differentials are assumed to be the continuous ones.	
\end{definition}

\indent Then for any element in the $|I|$-th differential we could define the corresponding residue of this chosen element as in the following:

\begin{definition} \mbox{\bf{(After KPX)}}
For any $f\in \Omega_{\Pi_{\mathrm{an},\mathrm{con},I,A}/A}$ of the general form of $f=\sum_{n_I}a_{n_1,...,n_I}T_1^{n_1}...T_I^{n_I}$, we define the residue of this differential of $|I|$-th order by the following formula:
\begin{displaymath}
\mathrm{res}(f):=\mathrm{res}(\sum_{n_I}a_{n_1,...,n_I}T_1^{n_1}...T_I^{n_I}):= a_{-\mathbf{1}}.	
\end{displaymath}
Here the symbol $\mathbf{1}$ is defined to be the index of $(1,...,1)$ in a uniform way.	
\end{definition}

\begin{proposition} \mbox{\bf{(After KPX \cite[Lemma 2.1.19]{KPX})}}
The residue map defined above in our setting is well-defined, and the differential module $\Omega_{\Pi_{\mathrm{an},\mathrm{con},I,A}/A}$ is free of rank one. We define the following differential pairing over the relative differential module $\Omega_{\Pi_{\mathrm{an},\mathrm{con},I,A}/A}$:
\begin{align}
h:{\Pi_{\mathrm{an},\mathrm{con},I,A}}\times {\Pi_{\mathrm{an},\mathrm{con},I,A}} &\rightarrow A\\
(f,g)	&\mapsto \mathrm{res}(f(T_I)g(T_I)dT_1\wedge...\wedge dT_I).
\end{align}
Then the pairing now induces the following canonical isomorphisms (where the topology depends on the choices of the variables):
\begin{align}
\mathrm{Hom}_{\mathrm{con}}(\Pi_{\mathrm{an},\mathrm{con},I,A},A)\overset{\sim}{\rightarrow}	\Pi_{\mathrm{an},\mathrm{con},I,A}\\
\mathrm{Hom}_{\mathrm{con}}(\Pi_{\mathrm{an},\mathrm{con},I,A}/\sum_{k}\Pi_{\mathrm{an},k,I,A} ,A)\overset{\sim}{\rightarrow}	\Pi_{\mathrm{an},\infty,I,A}.\\
\end{align}
Here the ring $\Pi_{\mathrm{an},k,I,A}$ is defined as
\begin{align}
\Pi_{\mathrm{an},k,I,A}:=\bigcup_{r_i>0,i\neq k, r_k=\infty,k=1,...,|I|}\Pi_{\mathrm{an},r_I,I,A}.	
\end{align}

\end{proposition}

\begin{proof}
As in the one dimensional situation, from the very definitions, one could derive that the residue map is well-defined and the fact that the differential module is generated by $dT_I$. Consider our product construction and the descent process in the one dimensional situation we have that one could then check the issue of well-definedness over the corresponding multivariate convergent rings $\mathcal{E}_{\Delta}$ in \cite{CKZ18} and \cite{PZ19}. Then for the first statement in the isomorphisms, first one could see that one direction (from right to left) which is induced directly from the residue pairing. On the other hand, we consider the following map which is a generalization of the one dimensional situation:
\begin{align}
\mathrm{Hom}_{\mathrm{con}}(\Pi_{\mathrm{an},\mathrm{con},I,A},A) &\overset{}{\rightarrow}	\Pi_{\mathrm{an},\mathrm{con},I,A}\\\mu_I	&\mapsto \sum_{n_I}\mu_I(T_I^{-\mathbf{1}-n_I})T_1^{n_1}...T_I^{n_I}, 
\end{align}
which is easy to see that this gives an inverse mapping by the definitions of the residue pairings. For the second isomorphism, under our construction, see \cite[Proposition 5.5]{Cre} and the corresponding \cite[Lemma 8.5.1]{Ked3}.

\end{proof}

\indent Then as below we are going to contact with the multi Lie groups structures which is a higher dimensional generalization of the usual pictures in \cite[Definition 2.1.20]{KPX}. To achieve so we need to generalize the constructions in \cite[Definition 2.1.20]{KPX} in the following sense:

\begin{definition} \mbox{\bf{(After KPX \cite[Definition 2.1.20]{KPX})}} \label{definition1}
Uniformizing the notations as in our list of notations at the beginning, we use the notation $(C_\alpha,\ell_\alpha)_{\alpha\in I}$ to denote $I$ pairs consisting of $I$ commutative topological groups with finite torsion parts and with the free part uniformly isomorphic to $\mathbb{Z}_p$, and $I$ morphism from $\ell_\alpha:C_\alpha\rightarrow \mathbb{Q}_p$ for each $\alpha\in I$. Then based on these we consider the following generalization of the constructions in \cite[Definition 2.1.20]{KPX}. First we are going to choose now for each $\alpha\in I$ a generator taking the form of $c_\alpha$ of the free part of the $\alpha$-th topological group for all $\alpha\in I$. Then we generalize the corresponding definitions in the usual situation as in the following. First assume that we only have torsion free topological groups in our setting for each $\alpha\in I$. Then we will use the notations $\Pi_{\mathrm{an},r_{I},I,A}(C_I),\Pi_{[s_I,r_I],I,A}(C_I)$ respectively to denote the corresponding higher dimensional period rings defined above namely $\Pi_{\mathrm{an},r_{I},I,A},\Pi_{[s_I,r_I],I,A}$ formally modified by considering the replacement of the variables with $c_\alpha-1$ for each dimension $\alpha\in I$. As in the usual situations in this situation we do not have the issue of the dependence of the choices of the generators. Also we know since by the corresponding one dimensional result we will have that
\begin{displaymath}
\left\|f(...,(1+T_\alpha)^p-1,...)\right\|_{r_I}=\left\|f(T_1,...,T_I)\right\|_{...,pr_\alpha,...}
\end{displaymath} 
for each $\alpha\in I$ which implies that certainly for all the $\alpha\in I$ together. Therefore from this fact we know that for each $\alpha\in I$:
\begin{displaymath}
\Pi_{\mathrm{an},(...,pr_{\alpha},...),I,A}(...,pC_\alpha,...)\otimes_{\mathbb{Z}_p[(...,pC_\alpha,...)]}	\mathbb{Z}_p[C_I]\overset{\sim}{\rightarrow} \Pi_{\mathrm{an},r_{I},I,A}(C_I),
\end{displaymath}
with
\begin{displaymath}
\Pi_{\mathrm{an},(...,p[s_{\alpha},r_{\alpha}],...),I,A}(...,pC_\alpha,...)\otimes_{\mathbb{Z}_p[(...,pC_\alpha,...)]}	\mathbb{Z}_p[C_I]\overset{\sim}{\rightarrow} \Pi_{[s_I,r_{I}],I,A}(C_I).
\end{displaymath}
As in \cite[Definition 2.1.20]{KPX} one could then consider more general situation where $C_I$ may not be torsion free. In this case one considers simultaneously all small $r_I<\mathbf{1}$ or all small $s_I$ when $r_I$ are all infinity, and any $D_\alpha$ a torsion free open subgroup of $C_\alpha$ for each $\alpha\in I$, one has:
\begin{displaymath}
\Pi_{\mathrm{an},([C_1:D_1]r_1,...,[C_I:D_I]r_I),I,A}(D_I)\otimes_{\mathbb{Z}_p[D_I]}	\mathbb{Z}_p[C_I]\overset{\sim}{\rightarrow} \Pi_{\mathrm{an},r_{I},I,A}(C_I),
\end{displaymath}
with
\begin{displaymath}
\Pi_{\mathrm{an},(...,[C_\alpha:D_\alpha][s_{\alpha},r_{\alpha}],...),I,A}(...,D_\alpha,...)\otimes_{\mathbb{Z}_p[(...,D_\alpha,...)]}	\mathbb{Z}_p[C_I]\overset{\sim}{\rightarrow} \Pi_{[s_I,r_{I}],I,A}(C_I).
\end{displaymath}
Then we consider the situation for some general open subgroups $D_I'$ of the groups $D_I$ which are torsion free then for each $\alpha\in I$, and for each $r_\alpha\leq 1/\sharp D_\mathrm{tor}$, or for each $s_\alpha\leq 1/\sharp D_\mathrm{tor}$ for $r_I$ infinite, we define:
\begin{displaymath}
\Pi_{\mathrm{an},([D_1:D'_1]r_1,...,[D_I:D'_I]r_I),I,A}(D'_I)\otimes_{\mathbb{Z}_p[D'_I]}	\mathbb{Z}_p[D_I]\overset{\sim}{\rightarrow} \Pi_{\mathrm{an},r_{I},I,A}(D_I),
\end{displaymath}
with
\begin{displaymath}
\Pi_{\mathrm{an},(...,[D_\alpha:D'_\alpha][s_{\alpha},r_{\alpha}],...),I,A}(...,D'_\alpha,...)\otimes_{\mathbb{Z}_p[(...,D'_\alpha,...)]}	\mathbb{Z}_p[D_I]\overset{\sim}{\rightarrow} \Pi_{[s_I,r_{I}],I,A}(D_I).
\end{displaymath}
As in the one dimensional situation everything is well-defined with respect to the choices involved. Then based on this observation we start to consider the generalization of the differentials defined above as in the following. In our situation we consider the differential $\omega_{\ell_I}$ which is defined to be the wedge product $\omega_{\ell_1}\wedge...\wedge\omega_{\ell_I}$, where each $\omega_{\ell_\alpha}$ is defined as in the situation in \cite[Definition 2.1.20]{KPX}. Then we could then consider any element $f\in \Omega_{\Pi_{\mathrm{an},\mathrm{con},I,A}(C_I)/A}$, and define:
\begin{displaymath}
\mathrm{res}(f):=a_{-\mathbf{1}}.	
\end{displaymath}
Then we define pairing as above in more general setting as:
\begin{align}
h:{\Pi_{\mathrm{an},\mathrm{con},I,A}(C_I)}\times {\Pi_{\mathrm{an},\mathrm{con},I,A}(C_I)} &\rightarrow A\\
(f,g) &\mapsto \mathrm{res}(fg\omega_{\ell_I}).	
\end{align}
And we also have the following analog:
\begin{align}
\mathrm{Hom}_{\mathrm{con}}({\Pi_{\mathrm{an},\mathrm{con},I,A}(C_I)},A)\overset{\sim}{\longrightarrow}	{\Pi_{\mathrm{an},\mathrm{con},I,A}(C_I)}.
\end{align}
\end{definition}

\newpage

\section{Relative Frobenius Sheaves over Affinoids and Fr\'echet Rings}

\noindent Now we introduce main objects in our study and our consideration. We first generalize the setting in \cite{KPX} to the situation we have multi Hodge-Frobenius structure which is quite natural since we consider already some higher dimensional spaces. Then we generalize everything to the setting in \cite{KP1}, which is also very natural in the consideration of noncommutative Iwasawa theory in the style of \cite{KP1}. The latter will be around the Iwasawa cohomology of some $p$-adic Lie deformation of the usual $(\varphi,\Gamma)$-modules. Therefore in the second situation we will consider the usual Hodge structure but we consider the $p$-adic Lie extension, for instance one could consider the product of $\mathbb{Z}_p^\times$ with some power of $\mathbb{Z}_p^\times$ or $\mathbb{Z}_p$. This will be then highly relative since we will work over some relative Robba rings or sheaves taking the form of $\Pi_{\mathrm{an},\mathrm{con},I,\Pi_{\mathrm{an},\infty,A}(G)}$, where $\Pi_{\mathrm{an},\infty}(G)$ for instance could be Berthelot's $A_\infty$-algebra, namely the distribution over $G$.\\

\begin{setting} \label{setting3.1}
We will work over the ring $\Pi_{\mathrm{an},\infty,A}(G)$ which is just the ring defined in the previous section which is abelian and which is the corresponding completed tensor product of the corresponding rings. To be more precise we will consider the Frobenius modules over $\Pi_{\mathrm{an},\mathrm{con},I,\Pi_{\mathrm{an},\infty,J,A}(\Gamma_{K_J})}$. Note that here $J$ could be different from $I$, but we assume that we have the relation $I\subset J$, and for those $K_\alpha$ ($\alpha\in J\backslash I$) we have that $K_\alpha=\mathbb{Q}_p$. Now set $G$ be $\Gamma_{K_J}$ (see the discussion below for the definitions). By considering the philosophy of Burns-Flach-Fukaya-Kato, this is actually already some noncommutative type Hodge structures.	
\end{setting}

\begin{remark}
The noncommutative setting is considered in the 2017 thesis \cite{Zah1}. 	
\end{remark}

\indent Let $K_I$ be $I$ finite extension of the $p$-adic number field $\mathbb{Q}_p$. Now for each $\alpha\in I$, we consider the cyclotomic field $K_{\alpha}(\mu_{p^\infty})$, which gives rise to the the extension $K_{\alpha}(\mu_{p^\infty})/K_\alpha$ with the Galois group $G(K_{\alpha}(\mu_{p^\infty})/K_\alpha)$ which is just the corresponding $\Gamma_{K_\alpha}$ in our situation. Note that for each $K_\alpha$ we have the associated cyclotomic character $\chi_{\alpha}$ whose kernel will be denoted as usual by $H_{K_\alpha}$, for each $\alpha\in I$. And for each $\alpha$ we will use the notation $\widetilde{e}_{K_{\alpha,\infty}}$ to denote the ramification indexes along the towers, namely for each $\alpha$ this is the quotient $[K_{\alpha}(\mu_{p^\infty}):\mathbb{Q}_p(\mu_{p^\infty})]/[k_{K_{\alpha}(\mu_{p^\infty})}:k_{K_\alpha}]$. Then first step now is to add some Hodge-Frobenius structures over the period rings we established in the previous section. To do so we consider the following situation:

\begin{definition} \mbox{\bf{(After KPX \cite[Definition 2.2.2]{KPX})}}
For each $\alpha\in I$, we choose suitable uniformizer $\pi_{K_\alpha}$ in our consideration. Then we are going to use the notation $\Pi_{\mathrm{an},\mathrm{con},I,A}(\pi_{K_I})$ to denote the corresponding period ring constructed from $\Pi_{\mathrm{an},\mathrm{con},I,A}$ just by directly replacing the variables by the corresponding uniformizers as above. And similarly for other period rings we use the corresponding notations taking the same form namely $\Pi_{\mathrm{an},\mathrm{con},I,A}(\pi_{K_I})$, $\Pi_{\mathrm{an},r_{I,0},I,A}(\pi_{K_I})$, $\Pi_{\mathrm{an},[s_I,r_I],I,A}(\pi_{K_I})$. As in the one dimensional situation we consider the situation where the radii are all sufficiently small. Then one could define the multiple Frobenius actions from the multi Frobenius $\varphi_I$ over each the ring mentioned above. For suitable radii $r_I$ we define a $\varphi_I$-module to be a finite projective $\Pi_{\mathrm{an},r_{I},I,A}(\pi_{K_I})$-module with the requirement that for each $\alpha\in I$ we have $\varphi_\alpha^*M_{r_I}\overset{\sim}{\rightarrow}M_{...,r_\alpha/p,...}$ (after suitable base changes). Then we define $M:=M_{r_I}\otimes_{\Pi_{\mathrm{an},r_{I},I,A}(\pi_{K_I})}\Pi_{\mathrm{an},\mathrm{con},I,A}(\pi_{K_I})$ to define a $\varphi_I$-module over the full relative Robba ring in our situation. Furthermore we have the notion of $\varphi_I$-bundles in our situation which consists of a family of $\varphi_I$-modules $\{M_{[s_I,r_I]}\}$ where each module $M_{[s_I,r_I]}$ is defined to be finite projective over $\Pi_{\mathrm{an},[s_I,r_I],I,A}(\pi_{K_I})$ satisfying the action formula taking the form of $\varphi_\alpha^*M_{[s_I,r_I]}\overset{\sim}{\rightarrow}M_{...,[s_\alpha/p,r_\alpha/p],...}$ (after suitable base changes). Furthermore for each $\alpha$ we have the corresponding operator $\varphi_\alpha:M_{[s_I,r_I]}\rightarrow M_{...,[s_\alpha/p,r_\alpha/p],...}$ and we have the corresponding operator $\psi_\alpha$ which is defined to be $p^{-1}\varphi_\alpha^{-1}\circ\mathrm{Trace}_{M_{...,[s_\alpha/p,r_\alpha/p],...}/\varphi_\alpha(M_{[s_I,r_I]})}$. Certainly we have the corresponding operator $\psi_\alpha$ over the global section $M_{r_I}$. Note that here we require that the Hodge-Frobenius structures are commutative in the sense that all the Frobenius are commuting with each other, and they are semilinear.
\end{definition}

\indent So now based on the definitions above we can consider the corresponding objects over the ring $\Pi_{\mathrm{an},\infty,A}(G)$, which are actually complicated. The basic reason is that over the infinite level we will lose some control of the corresponding finiteness, also very significantly we lose the control of the behavior of the radius throughout the whole variation over the Stein space attached to $\Pi_{\mathrm{an},\infty,A}(G)$. Let $X(\Pi_{\mathrm{an},\infty,A}(G))$ be the corresponding Fr\'echet-Stein space attached to the group $G$, which gives rise to the following represetation:
\begin{align}
X(\Pi_{\mathrm{an},\infty,A}(G))=\varinjlim_k X(\Pi_{\mathrm{an},\infty,A}(G))_k,\\
\mathcal{O}_{X(\Pi_{\mathrm{an},\infty,A}(G))}= \varprojlim_k \mathcal{O}_{X(\Pi_{\mathrm{an},\infty,A}(G))_k	},\\
\Pi_{\mathrm{an},\infty,A}(G)=: \varprojlim_k \Pi_{\mathrm{an},\infty,A}(G)_k.
\end{align}

\begin{definition} 
Over $\square:=\Pi_{\mathrm{an},\mathrm{con},I,\Pi_{\mathrm{an},\infty,A}(G)}(\pi_{K_I})$, or $\Pi_{\mathrm{an},r_{I,0},I,\Pi_{\mathrm{an},\infty,A}(G)}(\pi_{K_I})$, or $\Pi_{\mathrm{an},[s_I,r_I],I,\Pi_{\mathrm{an},\infty,A}(G)}(\pi_{K_I})$ as in the above (with the notations for the radii as considered above) we have the notion of $\varphi_I$-modules over $\square$. Be careful that we will define a $\varphi_I$-module $M$ over $\square$ to be the following projective limit:
\begin{align}
M:=\varprojlim_k M_k	
\end{align}
where $M_k$ is a corresponding object over $\square_k$ as defined above as a corresponding $\varphi_I$-module $M$ (finite projective) over $\square_k$. Here $\square_k:=\Pi_{\mathrm{an},\mathrm{con},I,\Pi_{\mathrm{an},\infty,A}(G)_k}(\pi_{K_I})$, or $\Pi_{\mathrm{an},r_{I,0},I,\Pi_{\mathrm{an},\infty,A}(G)_k}(\pi_{K_I})$, or $\Pi_{\mathrm{an},[s_I,r_I],I,\Pi_{\mathrm{an},\infty,A}(G)_k}(\pi_{K_I})$ respectively. One defines the $\varphi_I$-bundles over $\Pi_{\mathrm{an},\mathrm{con},I,\Pi_{\mathrm{an},\infty,A}(G)}(\pi_{K_I})$ in the same fashion.
\end{definition}

\indent We also have the corresponding notions of families of the corresponding objects considered above in the Fr\'echet-Stein situation.

\begin{definition} 
Over $\square:=\Pi_{\mathrm{an},\mathrm{con},I,X(\Pi_{\mathrm{an},\infty,A}(G))}(\pi_{K_I})$, or $\Pi_{\mathrm{an},r_{I,0},I,X(\Pi_{\mathrm{an},\infty,A}(G))}(\pi_{K_I})$, or\\
 $\Pi_{\mathrm{an},[s_I,r_I],I,X(\Pi_{\mathrm{an},\infty,A}(G))}(\pi_{K_I})$ as in the above (with the notations for the radii as considered above) we have the notion of $\varphi_I$-modules over $\square$. Be careful that we will define a $\varphi_I$-module $M$ over $\square$ to be the following the sheaf taking each quasi-compact $Y\subset X(\Pi_{\mathrm{an},\infty,A}(G))$ to $M(Y)$ such that this is a corresponding object over $\square_Y$ as defined above as a corresponding $\varphi_I$-module $M$ (finite projective) over $\square_Y$. Here $\square_Y:=\Pi_{\mathrm{an},\mathrm{con},I,\mathcal{O}_{X(\Pi_{\mathrm{an},\infty,A}(G))}(Y)}(\pi_{K_I})$, or $\Pi_{\mathrm{an},r_{I,0},I,\mathcal{O}_{X(\Pi_{\mathrm{an},\infty,A}(G))}(Y)}(\pi_{K_I})$, or
 $\Pi_{\mathrm{an},[s_I,r_I],I,\mathcal{O}_{X(\Pi_{\mathrm{an},\infty,A}(G))}(Y)}(\pi_{K_I})$ respectively. One defines the families of $\varphi_I$-bundles over $\Pi_{\mathrm{an},\mathrm{con},I,X(\Pi_{\mathrm{an},\infty,A}(G))}(\pi_{K_I})$ in the same fashion. Here one can take the corresponding $Y$ as some member in the filtration:
\begin{displaymath}
X(\Pi_{\mathrm{an},\infty,A}(G))=\varinjlim_k X(\Pi_{\mathrm{an},\infty,A}(G))_k.	
\end{displaymath}
 
\end{definition}

\begin{proposition} \mbox{\bf{(After KPX \cite[Proposition 2.2.7]{KPX})}} \label{proposition3.5}
We have that natural functor from families of $\varphi_I$-modules (in the sense of the previous definition without the finite condition on the global section) over $\Pi_{\mathrm{an},r_{I,0},I,\Pi_{\mathrm{an},\infty,A}(G)}(\pi_{K_I})$ to the corresponding families of $\varphi_I$-bundles (in the sense of the previous definition without the finite condition on the global section) is an equivalence.	Here by a family of $\varphi_I$-module we mean a $\varphi_I$-module over $\Pi_{\mathrm{an},r_{I,0},I,X(\Pi_{\mathrm{an},\infty,A}(G))}(\pi_{K_I})$, which is the same to a family of $\varphi_I$-bundles over $\Pi_{\mathrm{an},r_{I,0},I,\Pi_{\mathrm{an},\infty,A}(G)}(\pi_{K_I})$.
\end{proposition}

\begin{proof}
Here we need to consider the corresponding Fr\'echet sheaves as in the previous definition since in our situation the relative Robba rings are defined over some non-noetherian Stein space. The statement is obviously true for each $M_n$ as in the usual situation. First the faithfulness is straightforward, then for the surjectivity part, we only need to show that the global section is finite projective, which is a directly consequence of \cref{prop1} as long as in our situation one considers the simultaneous application of our multiple Frobenius to shift around to vary through all the intervals.	
\end{proof}

\begin{definition} \mbox{\bf{(After KPX \cite[Definition 2.2.12]{KPX})}}
Then we add the corresponding more Lie group action in our setting, namely the multi semiliear $\Gamma_{K_I}$-action. Again in this situation we require that all the actions of the $\Gamma_{K_I}$ are commuting with each other and with all the semilinear Frobenius actions defined above. We define the corresponding $(\varphi_I,\Gamma_I)$-modules over $\Pi_{\mathrm{an},r_{I,0},I,\Pi_{\mathrm{an},\infty,A}(G)}(\pi_{K_I})$ to be finite projective module with mutually commuting semilinear actions of $\varphi_I$ and $\Gamma_{K_I}$. For the latter action we require that to be continuous.
\end{definition}

\indent The continuity condition here could be made more explicit, as in \cite[Proposition 2.2.14]{KPX}. In our setting, we just need to generalize the result in \cite[Proposition 2.2.14]{KPX} to multiple actions from $\Gamma_{K_I}$.

\begin{lemma} \mbox{\bf{(After KPX \cite[Proposition 2.2.14]{KPX})}}
For suitable $r_{I,0}$, consider $M_{r_{I,0}}$ a module (carrying a semilinear action of $\Gamma_{K_I}$) of finite type over $\Pi_{\mathrm{an},r_{I,0},I,\Pi_{\mathrm{an},\infty,A}(G)}(\pi_{K_I})$ with the generating set $\mathbf{e}_1,...,\mathbf{e}_m$ (namely we look at a sheaf over $\Pi_{\mathrm{an},r_{I,0},I,X(\Pi_{\mathrm{an},\infty,A}(G))}(\pi_{K_I})$ with global section finitely generated). We will follow \cite[Proposition 2.2.14]{KPX} to differentiate the corresponding Banach module norm $|.|_{[s_I,r_{I,0}]}$ and the operator norm	$\left\|.\right\|_{[s_I,r_{I,0}]}$ when we consider the module over the polydisc with respect to the corresponding multi-interval (and the corresponding member in the Fr\'echet family). Then we have: I. The action of the group $\Gamma_{K_I}$ is continuous if and only if that for each $\gamma\in \Gamma_{K_I}$ and each element in the generating set we have $\lim_{\gamma\rightarrow 1} \gamma \mathbf{e}_j=\mathbf{e}_j$; The continuity of the action implies that: 
II. $\lim_{\gamma\rightarrow 1}\left\|\gamma-1\right\|_{[s_I,r_{I,0}]}=0$ for any $s_I$; 
III. Whenever we have the subgroups (with the notations in \cite[Proposition 2.2.14]{KPX}) $\Gamma_{n_I}\subset \Gamma_{K_I}$ ($\forall n_I\geq 0$), one can find separately $g_{\alpha,s_I}$ with the fact that $\left\|(\gamma'_{n_\alpha}-1)^{g_{\alpha,s_I}}\right\|_{[s_I,r_{I,0}]}< p^{-1/(p-1)}$ for each $\alpha\in I$.
\end{lemma}

\begin{proof}
First we need to work over some affinoid domain with respect to the Fr\'echet-Stein algebra in the coefficient, but this will be reduced to the corresponding affinoid situation. We then follow the strategy of \cite[Proposition 2.2.14]{KPX}, where we just need to consider each group in the product $\Gamma_{K_I}$. The first statement could be derived in the same way as in \cite[Proposition 2.2.14]{KPX}. For the second statement therefore we assume the limit identity in I, then prove the corresponding results on the identity in II. As in the proof in \cite[Proposition 2.2.14]{KPX}, for I and II, we reduce ourselves to the situation where the base field is $\mathbb{Q}_p$ for each $K_\alpha$ from the induction properties, and correspondingly we reduce us to the situation where all the uniformizers $\pi_I$ are just same as a product taking the form of $\pi^I$ as considered in the proof in \cite[Proposition 2.2.14]{KPX}. First is the following observation on the Newton-Leibniz style formula for each decomposition of elements with respect to given basis $\sum r_i e_i$ and for any $\gamma\in \Gamma_{K_I}$:
\begin{displaymath}
(\gamma-1)(r_i e_i)=\gamma(r_i)	(\gamma-1)(e_i)+(\gamma-1)r_i e_i.
\end{displaymath}
It suffices to look at the corresponding last term. We then further reduce ourselves to the situation where the module is just free of rank one over the base ring, which directly implies it suffices to consider the element in the ring taking the general form of $\pi_1^{n_1}...\pi_I^{n_I}$. Since we consider further the group element $\gamma_\alpha$ separately for each $\alpha\in I$, the result follows from then directly observation that $\gamma_\alpha(\pi_\beta)-\pi_\beta=f(\pi_\beta)((1+\pi_\beta)^{ap^k}-1)$ as well as $\gamma_\alpha(\pi_\beta^{-1})-\pi_\beta^{-1}=f'(\pi_\beta)((1+\pi_\beta)^{ap^k}-1)$ ($\alpha=\beta$ otherwise we have trivial action) where $f'$ actually depends on $k$ but could be controlled by some $g'$ which does not depend. Then from this for each $\beta,\alpha$ one gets the result. For the third statement we again argue by separately considering the corresponding group $\Gamma_\alpha$ where $\alpha\in I$. By first stage of reduction, we reduce ourselves to the situation where the base fields are all $\mathbb{Q}_p$ by induction and the uniformizer are all $\pi$ for each $\alpha \in I$. Since for each $\alpha$, as in \cite[Proposition 2.2.14]{KPX} the Banach module norm $|.|_{[s_I,r_{I,0}]}$ is invariant under the action of $\Gamma_{n_\alpha}$ for some $n_\alpha$. Note that this is based on the corresponding observation that:
\begin{align}
\gamma_\alpha(\pi^{i_1}_1...\pi_\alpha...\pi^{i_I}_I)-\pi^{i_1}_1...\pi_\alpha...\pi^{i_I}_I=\pi^{i_1}_1...f(\pi_\alpha)((1+\pi_\alpha)^{ap^k}-1)...\pi^{i_I}_I\\
\gamma_\alpha(\pi^{i_1}_1...\pi^{-1}_\alpha...\pi^{i_I}_I)-\pi^{i_1}_1...\pi_\alpha^{-1}...\pi^{i_I}_I=\pi^{i_1}_1...f'(\pi_\alpha)((1+\pi_\alpha)^{ap^k}-1)...\pi^{i_I}_I,	
\end{align}
and note that $|.|_{[s_I,r_{I,0}]}(\pi^{i_1}_1...((1+\pi_\alpha)^{ap^k}-1)...\pi^{i_I}_I)\leq |.|_{[s_I,r_{I,0}]}(\pi^{i_1}_1...\pi_\alpha...\pi^{i_I}_I)$. Then we consider the corresponding average which is invariant as in \cite[Proposition 2.2.14]{KPX} for some $n_\alpha'\geq n_\alpha$:
\begin{displaymath}
|.|^{-}_{[s_I,r_{I,0}]}:=\frac{1}{[\Gamma_{n_\alpha}:\Gamma_{n_\alpha}']}\sum_{\widetilde{\gamma}_\alpha\in \Gamma_{n_\alpha}/\Gamma_{n_\alpha}'}|.|_{[s_I,r_{I,0}]}\circ \widetilde{\gamma}_\alpha	
\end{displaymath}
which is equivalent to the original one:
\begin{displaymath}
C^{-1}|.|^{-}_{[s_I,r_{I,0}]}\leq |.|_{[s_I,r_{I,0}]} \leq C|.|^{-}_{[s_I,r_{I,0}]}	
\end{displaymath}
for some positive number $C>0$. First for the average norm we have that by taking suitable $n_\alpha'$ one could get $\left\|\gamma'_{n_\alpha'}-1\right\|^{-}_{[s_I,r_{I,0}]}$ is small enough, then take the corresponding $p^{n_\alpha'-n_\alpha}$-power of the element $\gamma'_{n_\alpha}-1$ one could further make this term smaller than $p^{-1/(p-1)}$ under the corresponding operator norm, namely the term $(\gamma'_{n_\alpha'}-1)^{p^{n_\alpha'-n_\alpha}}$ as in \cite[Proposition 2.2.14]{KPX}. Then by the inequality above between the original norms and the average ones, we have that by taking suitable power $g_{s_\alpha}$ one could make $\left\|(\gamma'_{n_\alpha}-1)^{g_{s_\alpha}}\right\|_{[s_I,r_{I,0}]}$ smaller than $p^{-1/(p-1)}$. To be more precise by raising the power $g_{n_\alpha}$ one will have $\left\|(\gamma'_{n_\alpha}-1)^{g_{s_\alpha}}\right\|_{[s_I,r_{I,0}]}\leq C\left\|(\gamma'_{n_\alpha}-1)^{g_{s_\alpha}}\right\|_{[s_I,r_{I,0}]}^{-}\leq CC_{g_{s_\alpha}}$ where $C_{g_{s_\alpha}}$ could be made sufficiently small along this process.
\end{proof}

\newpage

\section{Der Endlichkeitssatz}

\subsection{Finiteness through $p$-adic Functional Analaysis} \label{section4.1}

\noindent In this section we are going to study the cohomology of our relative Frobenius modules defined in the previous sections. We start from a $(\varphi_I,\Gamma_I)$-module over $\Pi_{\mathrm{an},\mathrm{con},I,A}(\pi_{K_I})$ as well a $(\psi_I,\Gamma_I)$-module over $\Pi_{\mathrm{an},\mathrm{con},I,A}(\pi_{K_I})$. We are going to use $M$ to denote such a module. First we are going to use the notation $\triangle_I$ to denote the $p$-torsion part of $\Gamma_{I}$ in the product form, which will be used in the definitions of the corresponding cohomologies as below.

\begin{assumption} \label{assumption1}
We choose to consider the story in which $I$ consists of two elements. We note that actually everything could be carried over to more general setting. We will then use the notation $\{1,2\}$ to denote our set which indicates the corresponding Frobenius actions and Lie group actions with respect to each index.	
\end{assumption}

\begin{definition} \mbox{\bf{(After KPX \cite[2.3]{KPX})}}
For any $(\varphi_I,\Gamma_I)$-module $M$ over $\Pi_{\mathrm{an},\mathrm{con},I,A}(\pi_{K_I})$ we define the complex $C^\bullet_{\Gamma_I}(M)$ of $M$ to be the total complex of the following complex through the induction:
\[
\xymatrix@R+0pc@C+0pc{
[ C^\bullet_{\Gamma_{I\backslash\{I\}}}(M) \ar[r]^{\Gamma_I} \ar[r] \ar[r]  & C^\bullet_{\Gamma_I\backslash\{I\}}(M)]
.}
\]
Then we define the corresponding double complex $C^\bullet_{\varphi_I}C^\bullet_{\Gamma_I}(M)$ again by taking the corresponding totalization of the following complex through induction:
\[
\xymatrix@R+0pc@C+0pc{
[C^\bullet_{\varphi_{I\backslash\{I\}}}C^\bullet_{\Gamma_{I}}(M) \ar[r]^{\varphi_I} \ar[r] \ar[r]  & C^\bullet_{\varphi_{I\backslash\{I\}}}C^\bullet_{\Gamma_I}(M)]
.}
\]
Then we define the complex $C^\bullet_{\varphi_I,\Gamma_I}(M)$ to be the totalization of the double complex defined above, which is called the complex of a $(\varphi_I,\Gamma_I)$-module $M$. Similarly we define the corresponding complex $C^\bullet_{\psi_I,\Gamma_I}(M)$ to be the totalization of the double complex define  d in the same way as above by replacing $\varphi_I$ by then the operators $\psi_I$. For later comparison we also need the corresponding $\psi_I$-cohomology, which is to say the complex $C^\bullet_{\psi_I}(M)$ which could be defined by the totalization of the following complex through induction:
\[
\xymatrix@R+0pc@C+0pc{
[ C^\bullet_{\psi_{I\backslash\{I\}}}(M) \ar[r]^{\psi_I} \ar[r] \ar[r]  & C^\bullet_{\psi_I\backslash\{I\}}(M)]
.}
\]
\end{definition}

\indent One can actually compare the two total complexes defined above in the following way, as in the situation where $I$ is a singleton. First for the $(\varphi_I,\Gamma_I)$-complex we have the following initial double complex:
\[
\xymatrix@R+6pc@C+0pc{
 &0\ar[d] \ar[d] \ar[d]  &0 \ar[d] \ar[d] \ar[d]  & 0 \ar[d] \ar[d] \ar[d] \\
0\ar[r] \ar[r] \ar[r] &M^{\triangle_I}\ar[d]^{(\gamma_1-1,\gamma_2-1)} \ar[d] \ar[d] \ar[r]^{(\varphi_1-1,\varphi_2-1)} \ar[r] \ar[r] &M^{\triangle_I}\oplus M^{\triangle_I}\ar[d] \ar[d] \ar[d]\ar[r]^{(\varphi_2-1)+(1-\varphi_1)} \ar[r] \ar[r]  & M^{\triangle_I}\ar[d] \ar[d] \ar[d] \ar[r] \ar[r] \ar[r] &0\\
0\ar[r] \ar[r] \ar[r] &M^{\triangle_I}\oplus M^{\triangle_I}\ar[d]^{(\gamma_2-1)+(1-\gamma_1)} \ar[d] \ar[d] \ar[r] \ar[r] \ar[r] &M^{\triangle_I}\oplus M^{\triangle_I}\oplus M^{\triangle_I}\oplus M^{\triangle_I}\ar[d] \ar[d] \ar[d]\ar[r] \ar[r] \ar[r]  & M^{\triangle_I}\oplus M^{\triangle_I}\ar[d] \ar[d] \ar[d] \ar[r] \ar[r] \ar[r] &0\\
0\ar[r] \ar[r] \ar[r] &M^{\triangle_I}\ar[d] \ar[d] \ar[d] \ar[r] \ar[r] \ar[r] &M^{\triangle_I}\oplus M^{\triangle_I}\ar[d] \ar[d] \ar[d]\ar[r] \ar[r] \ar[r]  & M^{\triangle_I}\ar[d] \ar[d] \ar[d] \ar[r] \ar[r] \ar[r] &0\\
&0 &0  & 0
.}
\]
Here we followed the corresponding convention as in \cite{KPX}. And then correspondingly we have the following double complex for the corresponding $(\psi_I,\Gamma_I)$-module:
\begin{center}
\[
\xymatrix@R+6pc@C+0pc{
 &0\ar[d] \ar[d] \ar[d]  &0 \ar[d] \ar[d] \ar[d]  & 0 \ar[d] \ar[d] \ar[d] \\
0\ar[r] \ar[r] \ar[r] &M^{\triangle_I}\ar[d]^{(\gamma_1-1,\gamma_2-1)} \ar[d] \ar[d] \ar[r]^{(\psi_1-1,\psi_2-1)} \ar[r] \ar[r] &M^{\triangle_I}\oplus M^{\triangle_I}\ar[d] \ar[d] \ar[d]\ar[r]^{(\psi_2-1)+(1-\psi_1)} \ar[r] \ar[r]  & M^{\triangle_I}\ar[d] \ar[d] \ar[d] \ar[r] \ar[r] \ar[r] &0\\
0\ar[r] \ar[r] \ar[r] &M^{\triangle_I}\oplus M^{\triangle_I}\ar[d]^{(\gamma_2-1)+(1-\gamma_1)} \ar[d] \ar[d] \ar[r] \ar[r] \ar[r] &M^{\triangle_I}\oplus M^{\triangle_I}\oplus M^{\triangle_I}\oplus M^{\triangle_I}\ar[d] \ar[d] \ar[d]\ar[r] \ar[r] \ar[r]  & M^{\triangle_I}\oplus M^{\triangle_I}\ar[d] \ar[d] \ar[d] \ar[r] \ar[r] \ar[r] &0\\
0\ar[r] \ar[r] \ar[r] &M^{\triangle_I}\ar[d] \ar[d] \ar[d] \ar[r] \ar[r] \ar[r] &M^{\triangle_I}\oplus M^{\triangle_I}\ar[d] \ar[d] \ar[d]\ar[r] \ar[r] \ar[r]  & M^{\triangle_I}\ar[d] \ar[d] \ar[d] \ar[r] \ar[r] \ar[r] &0\\
&0 &0  & 0
.}
\]
\end{center}

\indent In our situation in order to relate the two double complexes we consider the corresponding map induced from the following maps for the two first lines of two big complexes where one complex is given as above, while the other is the following one:

\begin{center}
\[
\xymatrix@R+6pc@C+0pc{
 &0\ar[d] \ar[d] \ar[d]  &0 \ar[d] \ar[d] \ar[d]  & 0 \ar[d] \ar[d] \ar[d] \\
0\ar[r] \ar[r] \ar[r] &M^{\triangle_I}\ar[d]^{(\gamma_1-1,\gamma_2-1)} \ar[d] \ar[d] \ar[r]^{(\varphi_1-1,\psi_2-1)} \ar[r] \ar[r] &M^{\triangle_I}\oplus M^{\triangle_I}\ar[d] \ar[d] \ar[d]\ar[r]^{(\psi_2-1)+(1-\varphi_1)} \ar[r] \ar[r]  & M^{\triangle_I}\ar[d] \ar[d] \ar[d] \ar[r] \ar[r] \ar[r] &0\\
0\ar[r] \ar[r] \ar[r] &M^{\triangle_I}\oplus M^{\triangle_I}\ar[d]^{(\gamma_2-1)+(1-\gamma_1)} \ar[d] \ar[d] \ar[r] \ar[r] \ar[r] &M^{\triangle_I}\oplus M^{\triangle_I}\oplus M^{\triangle_I}\oplus M^{\triangle_I}\ar[d] \ar[d] \ar[d]\ar[r] \ar[r] \ar[r]  & M^{\triangle_I}\oplus M^{\triangle_I}\ar[d] \ar[d] \ar[d] \ar[r] \ar[r] \ar[r] &0\\
0\ar[r] \ar[r] \ar[r] &M^{\triangle_I}\ar[d] \ar[d] \ar[d] \ar[r] \ar[r] \ar[r] &M^{\triangle_I}\oplus M^{\triangle_I}\ar[d] \ar[d] \ar[d]\ar[r] \ar[r] \ar[r]  & M^{\triangle_I}\ar[d] \ar[d] \ar[d] \ar[r] \ar[r] \ar[r] &0\\
&0 &0  & 0
}
\]
\end{center}

 (which is a direct generalization of the corresponding morphisms in one dimensional situation in \cite[Definition 2.3.3]{KPX}):

\[
\xymatrix@R+5pc@C+4pc{
0\ar[r] \ar[r] \ar[r] &M^{\triangle_I}\ar[d]^{\mathrm{Id}} \ar[d] \ar[d] \ar[r]^{(\varphi_1-1,\gamma_1-1)} \ar[r] \ar[r] &M^{\triangle_I}\oplus M^{\triangle_I}\ar[d]^{(-\psi_1,\mathrm{Id})} \ar[d] \ar[d]\ar[r]^{(\gamma_1-1)+(1-\varphi_1)} \ar[r] \ar[r]  & M^{\triangle_I}\ar[d]^{-\psi_1} \ar[d] \ar[d] \ar[r] \ar[r] \ar[r] &0\\
0\ar[r] \ar[r] \ar[r] &M^{\triangle_I} \ar[r]^{(\psi_1-1,\gamma_1-1)} \ar[r] \ar[r] &M^{\triangle_I}\oplus M^{\triangle_I}\ar[r]^{(\gamma_1-1)+(1-\psi_1)} \ar[r] \ar[r]  & M^{\triangle_I} \ar[r] \ar[r] \ar[r] &0
.}
\]

We will denote the corresponding morphism in between the two total complex by:

\[
\xymatrix@R+0pc@C+0pc{
\Psi_1: C^\bullet_{\varphi_1,\psi_2,\Gamma_1,\Gamma_2}(M) \ar[r] \ar[r] \ar[r]  & C^\bullet_{\psi_1,\psi_2,\Gamma_1,\Gamma_2}(M)
.}
\]

\begin{proposition} \mbox{\bf{(After KPX \cite[Proposition 2.3.6,Lemma 3.1.2]{KPX})}}
The morphism  
\[
\xymatrix@R+0pc@C+0pc{
\Psi_1': C^\bullet_{\varphi_1,\Gamma_1}(M) \ar[r] \ar[r] \ar[r]  & C^\bullet_{\psi_1,\Gamma_1}(M)
}
\]
is a quasi-isomorphism. This will imply that $\Psi_1$ is a quasi-isomorphism.	
\end{proposition}

\begin{proof}
First the problem is now reduced to the checking of the injectivity of the morphism $\Psi_1'$ defined above. Note now that the kernel of the above morphism is now just the corresponding complex which consists of the invariance taking the form of $M^{\triangle_I,\psi_1=0}$, so it suffices to prove that the operator $\gamma_1-1$ in our situation bas the corresponding property of being invertible. Therefore we will follow the argument in \cite[Proposition 2.3.6]{KPX}, which will be just a direct generalization of the corresponding one dimensional result in \cite[Proposition 2.3.6]{KPX}. First we consider the reduction of the proof to the situation where all the finite extensions $K_I$ over $\mathbb{Q}_p$ are all just the fields $\mathbb{Q}_p$, which could be checked through the induction. Then we assume that we have a corresponding $\Pi_{\mathrm{an},r_{I,0},I,A}(\pi^I)$-module $M_{r_{I,0}}^{\triangle_{\mathbb{Q}_p}}$ and the corresponding invariance $M_{r_{I,0}}^{\triangle_{\mathbb{Q}_p},\psi_1=0}$.
Then we consider the suitable multi integers $n_I\geq 1$ which gives rise to (for suitable multi radii $s_I$ and $r_I$ smaller than the corresponding original multi radii $r_{I,0}$) the corresponding decomposition of the section 
\begin{displaymath}
M_{r_{I,0}}^{\triangle_{\mathbb{Q}_p},\psi_1=0}\otimes_{\Pi_{\mathrm{an},r_{I,0},I,A}(\pi^I)} \Pi_{\mathrm{an},[s_I/p^{n_I},r_I/p^{n_I}],I,A}(\pi^I)	
\end{displaymath}
into the corresponding summation of the corresponding components taking the form of:
\begin{displaymath}
(1+\pi_1)^{i_1}\varphi_1^{n_1}\varphi_2^{n_2}M_{[s_I,r_I]}^{\triangle_{\mathbb{Q}_p}},	
\end{displaymath}
for suitable integers $i_1$ in our generalized situation. Then it now suffices to show that for each $n_I$, the action of the corresponding operators $\gamma_{n_I}-1$ are invertible over each member in the decomposition as above, which is to say the action of each $\gamma_\alpha-1$ is invertible over each $(1+\pi_1)\varphi_1^{n_1}M_{[s_I,r_I]}^{\triangle_{\mathbb{Q}_p}}$, which implies that the $A[\gamma_{n_I}]$-module structure could be promoted by continuation to the corresponding $\Pi_{\mathrm{an},r_{I,0},I,A}(\Gamma_{n_I})$-module structure. Now we look at each member in the decomposition as above, and do the standard computation as in \cite[Lemma 3.1.2]{KPX} as below (which is generalization since we are working with higher dimensional situation) for each $m$ in each member participating in the decomposition as above:
\begin{align}
(\gamma_{n_1}-1)&(1+\pi_1)\varphi_1^{n_1}\varphi_2^{n_2}(m)\\
&=(\gamma_{n_1}-1)(1+\pi_1)\varphi_1^{n_1}\varphi_2^{n_2}(m)\\
&=(1+\pi_1)^{p^{n_1+1}}\varphi_1^{n_1}\varphi_2^{n_2}\gamma_{n_1}(m)-(1+\pi_1)\varphi_1^{n_1}\varphi_2^{n_2}(m)\\
&=(1+\pi_1)((1+\pi_1)^{p^{n_1}}\varphi_1^{n_1}\varphi_2^{n_2}\gamma_{n_1}-\varphi_1^{n_1}\varphi_2^{n_2})(m)\\
&=(1+\pi_1)\varphi_1^{n_1}((1+\pi_1)\varphi_2^{n_2}\gamma_{n_1}-\varphi_2^{n_2})(m)\\
&=(1+\pi_1)\varphi_1^{n_1}\varphi_2^{n_2}\pi_1(1+(1+\pi_1)/\pi_1(\gamma_{n_1}-1))(m)
\end{align}
which implies that the inverse could be taken to be:
\begin{align}
[(1+\pi_1)\varphi_1^{n_1}&\varphi_2^{n_2}\pi_1(1+(1+\pi_1)/\pi_1(\gamma_{n_1}-1))]^{-1}(m)	\\
&=\frac{1}{(1+\pi_1)\varphi_1^{n_1}\varphi_2^{n_2}\pi_1}\cdot \frac{1}{1+(1+\pi_1)/\pi_1(\gamma_{n_1}-1)}m\\
&=\frac{1}{(1+\pi_1)\varphi_1^{n_1}\varphi_2^{n_2}\pi_1}\cdot \sum_{k\geq 0}(-1)^k((1+\pi_1)/\pi_1(\gamma_{n_1}-1))^k m.
\end{align}
\end{proof}

\indent Now in the same fashion one can define the corresponding morphism:
\[
\xymatrix@R+0pc@C+0pc{
\Psi_2: C^\bullet_{\varphi_1,\varphi_2,\Gamma_I}(M) \ar[r] \ar[r] \ar[r]  & C^\bullet_{\varphi_1,\psi_2,\Gamma_I}(M)
}
\]

\begin{lemma}
This morphism $\Psi_2$ is also a quasi-isomorphism. Therefore by composition we have the following quasi-isomorphism:
\[
\xymatrix@R+0pc@C+0pc{
\Psi_I: C^\bullet_{\varphi_1,\varphi_2,\Gamma_I}(M) \ar[r] \ar[r] \ar[r]  & C^\bullet_{\psi_1,\psi_2,\Gamma_I}(M).
}
\]

\end{lemma}

\indent Then now we are at the position to consider the finiteness of the corresponding involved cohomologies. In our situation we consider the following derived categories:

\begin{definition} \mbox{\bf{(After KPX \cite[Notation 4.1.2]{KPX})}} \label{definition4.5}
Now we consider the following derived categories. The first one is the sub-derived category consisting of all the objects in $\mathbb{D}(A)$ which are quasi-isomorphic to those bounded above complexes of finite projective modules over the ring $A$. We denote this category (which is defined in the same way as in \cite[Notation 4.1.2]{KPX}) by $\mathbb{D}^-_\mathrm{perf}(A)$. Over the larger ring $\Pi_{\mathrm{an},\infty,I,A}(\Gamma_{K_I})$ we also have the sub-derived category of $\mathbb{D}(\Pi_{\mathrm{an},\infty,I,A}(\Gamma_{K_I}))$ consisting of all those objects in $\mathbb{D}(\Pi_{\mathrm{an},\infty,I,A}(\Gamma_{K_I}))$ which are quasi-isomorphic to those bounded above complexes of finite projective modules now over the ring $\Pi_{\mathrm{an},\infty,I,A}(\Gamma_{K_I})$. We denote this by $\mathbb{D}^-_\mathrm{perf}(\Pi_{\mathrm{an},\infty,I,A}(\Gamma_{K_I}))$. Similar we define the bounded derived categories $\mathbb{D}^\flat_\mathrm{perf}(A)$, $\mathbb{D}^\flat_\mathrm{perf}(\Pi_{\mathrm{an},\infty,I,A}(\Gamma_{K_I}))$. We then use the notation as following $D_\mathrm{perf}(A)$, $D_\mathrm{perf}(\Pi_{\mathrm{an},\infty,I,A}(\Gamma_{K_I}))$, $D^-_\mathrm{perf}(A)$, $D^-_\mathrm{perf}(\Pi_{\mathrm{an},\infty,I,A}(\Gamma_{K_I}))$, as well as $D^\flat_\mathrm{perf}(A)$, $D^\flat_\mathrm{perf}(\Pi_{\mathrm{an},\infty,I,A}(\Gamma_{K_I}))$ to denote the $(\infty, 1)$-enhancement of these derived categories.
\end{definition}

\indent Then we will investigate the complexes attached to any finite projective $(\varphi_I,\Gamma_I)$-module $M$ over $\Pi_{\mathrm{an},\mathrm{con},I,A}(\pi_{K_I})$:
\begin{displaymath}
C^\bullet_{\varphi_I,\Gamma_I}(M), C^\bullet_{\psi_I,\Gamma_I}(M),C^\bullet_{\psi_I}(M),	
\end{displaymath}
which are living inside the corresponding derived categories:
\begin{displaymath}
\mathbb{D}(A),\mathbb{D}(A),\mathbb{D}(\Pi_{\mathrm{an},\infty,I,A}(\Gamma_{K_I})).	
\end{displaymath}

\begin{proposition} \mbox{\bf{(After KPX \cite[Proposition 4.2.3]{KPX})}} \label{proposition4.4}
Now for any multiplicative character $\eta_I$ of the group $\Gamma_{K_I}$, we have the following isomorphism in the corresponding derived categories defined above:
\begin{align}
C^\bullet_{\psi_I}(M)\otimes^{\mathbb{L}}\Pi_{\mathrm{an},\infty,I,A}(\Gamma_{K_I})/\mathfrak{m}_\eta \Pi_{\mathrm{an},\infty,I,A}(\Gamma_{K_I})&\overset{\sim}{\longrightarrow} C^\bullet_{\psi_I,\Gamma_I}(M\otimes M_{\eta^{-1}})\\
	&\overset{\sim}{\longrightarrow}C^\bullet_{\varphi_I,\Gamma_I}(M\otimes M_{\eta^{-1}}).
\end{align}	
\end{proposition}

\indent Then we are going to study the finiteness of the cohomologies defined above in more details. The methods in \cite{KPX} requires some deep results on the slope filtrations for modules over the one-dimensional Robba rings, where allows \cite{KPX} to reduce the proof in some sense to the \'etale situation. It looks like we have no chance to apply the ideas in \cite{KPX}. Here we adopt the ideas in some development from Kedlaya-Liu in \cite{KL3}. To be more precise we are going to apply some ideas in \cite{KL3} around the application of some $p$-adic functional analysis.


\begin{proposition} \label{proposition4.7}
With the notation and setting up as above we have $h^\bullet(C_{\psi_I}(M))$ are coadmissible over $\Pi_{\mathrm{an},\infty,I,A}(\Gamma_{K_I})$.	
\end{proposition}

\begin{proof}
This can be proved in the \'etale setting after \cite{PZ19} and \cite{CKZ18} in the absolute situation, by using Galois cohomology and the Hochschild-Serre spectral sequences directly. We adopt the Cartan-Serre method involving the completely continuous maps within the $p$-adic functional analysis. By simultaneous induction to the base cases for the field $\mathbb{Q}_p$ from all $K_I$, we can now assume that we are considering the situation where all $K_I$ are $\mathbb{Q}_p$. Then for the complex $C^\bullet_{\psi_I}(M)$, by definition one can then regard this complex as the double complex taking the following form (after considering the  totalization):
\[
\xymatrix@R+5pc@C+5pc{
M \ar[r]^{\psi_1-1} \ar[r] \ar[r] \ar[d]^{\psi_2-1} \ar[d] \ar[d] & M \ar[d] \ar[d] \ar[d]\\
M \ar[r] \ar[r] \ar[r]  & M. 
}
\]
\indent The boundedness of the complex can be observed from the definition and the ideas in \cite[Theorem 3.3, Theorem 3.5]{KL3}. For the finiteness we construct suitable completely continuous maps comparing the quasi-isomorphic complexes for $M$ as below. First we consider some model $M_{r_I}$ over some $\Pi_{\mathrm{an},r_{I,0},I,A}(\pi_{K_I})$. Then as in \cite[Theorem 3.3, Theorem 3.5]{KL3} we can regard the complex for $M_{r_I}$ as equivalently the  complex over $\Pi_{\mathrm{an},r_{I,0},I,\Pi_{\mathrm{an},\infty,I,A}(\Gamma_{K_I})}(\pi_{K_I})$. We are then reduced to considering the complex over $\Pi_{\mathrm{an},r_{I_0},I,\Pi_{\mathrm{an},[t_I,\infty],I,A}(\Gamma_{K_I})}(\pi_{K_I})$ for some $t_I$.\\
\indent Then we consider the complex for some section over some multi interval taking the form of $[s_I,r_I]$ which gives rise to the quasi-isomorphic complex now over 
\begin{center}
$\Pi_{\mathrm{an},[s_I,r_{I_0}],I,\Pi_{\mathrm{an},[t_I,\infty],I,A}(\Gamma_{K_I})}(\pi_{K_I})$ 
\end{center}
which gives rise to the complex:
\[
\xymatrix@R+5pc@C+2pc{
M_{[s_1,r_1]\times[s_2,r_2]} \ar[r]^{\psi_1-1} \ar[r] \ar[r] \ar[d]^{\psi_2-1} \ar[d] \ar[d] & M_{[ps_1,r_1]\times[s_2,r_2]} \ar[d] \ar[d] \ar[d]\\
M_{[s_1,r_1]\times[ps_2,r_2]}\ar[r] \ar[r] \ar[r]  & M_{[ps_1,r_1]\times[ps_2,r_2]}. 
}
\]
Here the quasi-isomorphism among different boxes (e.g. multi intervals) is not actually trivial. For the serious argument after the knowledge on the equivalence between the categories of $\varphi_I$-modules over the full Robba ring, $\varphi_I$-Frobenius bundles and $\varphi_I$-modules over the Robba ring with respect to some single multi interval, which can be realized in our context by using \cref{proposition3.5} as in \cite[Lemma 2.10]{KP1}, we refer to \cite[Lemma 5.2]{KP1} and \cite[Proposition 5.12, Theorem 5.7.11]{KL2} for the comparison on the cohomology groups. \\
\indent To show that the cohomology groups do not change when we switch from a multi-interval $[s_1',r_1']\times [s_2',r_2']$ to a multi-interval $[s_1'',r_1'']\times [s_2'',r_2'']$, we choose to look at an intermediate multi-interval $[s_1'',r_1'']\times [s_2',r_2']$ namely we consider the following in order rings over multi-intervals:
\begin{align}
&\Pi_{\mathrm{an},[s'_1,r'_1]\times [s'_2,r'_2],I,\Pi_{\mathrm{an},[t_I,\infty],I,A}(\Gamma_{K_I})}(\pi_{K_I})\\
&\Pi_{\mathrm{an},[s''_1,r''_1]\times [s'_2,r'_2],I,\Pi_{\mathrm{an},[t_I,\infty],I,A}(\Gamma_{K_I})}(\pi_{K_I})\\
&\Pi_{\mathrm{an},[s''_1,r''_1]\times [s''_2,r''_2],I,\Pi_{\mathrm{an},[t_I,\infty],I,A}(\Gamma_{K_I})}(\pi_{K_I})	
\end{align}
and compare the first two and the last two. To compare with respect to the first two multi-intervals we first consider the $\psi_1$-complex, the cohomology groups can be regarded as the Yoneda extension groups over the twisted polynomial ring in variable of $\psi_1$ with coefficients in the Robba rings, which shows that the $\psi_I$-cohomology groups do not change when we change the boxes. To be more precise, we form the spectral sequences for the two intervals in comparisom coming from the overall double complexes:
\begin{align}
&\mathrm{E}^{i,j}_{*,\Pi_{\mathrm{an},[s'_1,r'_1]\times [s'_2,r'_2],I,\Pi_{\mathrm{an},[t_I,\infty],I,A}(\Gamma_{K_I})}(\pi_{K_I})},\\
&\mathrm{E}^{k,l}_{*,\Pi_{\mathrm{an},[s''_1,r''_1]\times [s'_2,r'_2],I,\Pi_{\mathrm{an},[t_I,\infty],I,A}(\Gamma_{K_I})}(\pi_{K_I})}.
\end{align}
Here one considers the $0$-th page and the $1$-th page respectively, one can see directly that the isomorphism beyond the $1$-page:
\begin{align}
\mathrm{E}^{i,j}_{q,\Pi_{\mathrm{an},[s'_1,r'_1]\times [s'_2,r'_2],I,\Pi_{\mathrm{an},[t_I,\infty],I,A}(\Gamma_{K_I})}(\pi_{K_I})}\overset{\sim}{\longrightarrow}\mathrm{E}^{k,l}_{q,\Pi_{\mathrm{an},[s''_1,r''_1]\times [s'_2,r'_2],I,\Pi_{\mathrm{an},[t_I,\infty],I,A}(\Gamma_{K_I})}(\pi_{K_I})},~q>1.
\end{align}
This implies that:
\begin{align}
\mathrm{E}^{i,j}_{\infty,\Pi_{\mathrm{an},[s'_1,r'_1]\times [s'_2,r'_2],I,\Pi_{\mathrm{an},[t_I,\infty],I,A}(\Gamma_{K_I})}(\pi_{K_I})}\overset{\sim}{\longrightarrow}\mathrm{E}^{k,l}_{\infty,\Pi_{\mathrm{an},[s''_1,r''_1]\times [s'_2,r'_2],I,\Pi_{\mathrm{an},[t_I,\infty],I,A}(\Gamma_{K_I})}(\pi_{K_I})}.
\end{align}
This implies that the cohomologies of the original double complexes are actually preserved when we change the boxes. Then for the remaining similar comparisons in this articles in the following discussion we argue in the parallel way. \\
\indent To show that the cohomology groups do not change when we switch from Robba ring over some fixed multi-radii $r_{1}\times r_{2}$ to some interval $[s_1,r_1]\times [s_2,r_2]$, we consider the  intermediate Robba ring:
\begin{align}
\bigcap_{s_2>0} \Pi_{\mathrm{an},[s_1,r_1]\times [s_2,r_{2}],I,\Pi_{\mathrm{an},[t_I,\infty],I,A}(\Gamma_{K_I})}(\pi_{K_I})	
\end{align}
and then in order the following rings:
\begin{align}
&\bigcap_{s_1>0,s_2>0} \Pi_{\mathrm{an},[s_1,r_{1}]\times [s_2,r_{2}],I,\Pi_{\mathrm{an},[t_I,\infty],I,A}(\Gamma_{K_I})}(\pi_{K_I}),\\
&\bigcap_{s_2>0} \Pi_{\mathrm{an},[s_1,r_{1}]\times [s_2,r_{2}],I,\Pi_{\mathrm{an},[t_I,\infty],I,A}(\Gamma_{K_I})}(\pi_{K_I}),\\
& \Pi_{\mathrm{an},[s_1,r_{1}]\times [s_2,r_{2}],I,\Pi_{\mathrm{an},[t_I,\infty],I,A}(\Gamma_{K_I})}(\pi_{K_I}).
\end{align}
Note that finite projective modules over these are equivalent when carrying the  $\varphi_I$-structure. To compare with respect to the first two situations, we first consider the $\psi_1$-cohomology complex, then by using the Yoneda extension we have the  $\psi_1$-cohomology groups are the same, which implies further that the  $\psi_I$-cohomology groups are the same. To be more precise, we form the spectral sequences for the two intervals in comparisom coming from the overall double complexes:
\begin{align}
&\mathrm{E}^{i,j}_{*,\bigcap_{s_1>0,s_2>0} \Pi_{\mathrm{an},[s_1,r_{1}]\times [s_2,r_{2}],I,\Pi_{\mathrm{an},[t_I,\infty],I,A}(\Gamma_{K_I})}(\pi_{K_I})},\\
&\mathrm{E}^{k,l}_{*,\bigcap_{s_2>0} \Pi_{\mathrm{an},[s_1,r_{1}]\times [s_2,r_{2}],I,\Pi_{\mathrm{an},[t_I,\infty],I,A}(\Gamma_{K_I})}(\pi_{K_I})}.
\end{align}
Here one considers the $0$-th page and the $1$-th page respectively, one can see directly that the isomorphism beyond the $1$-page:
\begin{align}
\mathrm{E}^{i,j}_{q,\bigcap_{s_1>0,s_2>0} \Pi_{\mathrm{an},[s_1,r_{1}]\times [s_2,r_{2}],I,\Pi_{\mathrm{an},[t_I,\infty],I,A}(\Gamma_{K_I})}(\pi_{K_I})}\overset{\sim}{\longrightarrow}\\
\mathrm{E}^{k,l}_{q,\bigcap_{s_2>0} \Pi_{\mathrm{an},[s_1,r_{1}]\times [s_2,r_{2}],I,\Pi_{\mathrm{an},[t_I,\infty],I,A}(\Gamma_{K_I})}(\pi_{K_I})},~q>1.
\end{align}
This implies that:
\begin{align}
\mathrm{E}^{i,j}_{\infty,\bigcap_{s_1>0,s_2>0} \Pi_{\mathrm{an},[s_1,r_{1}]\times [s_2,r_{2}],I,\Pi_{\mathrm{an},[t_I,\infty],I,A}(\Gamma_{K_I})}(\pi_{K_I})}\overset{\sim}{\longrightarrow}\\
\mathrm{E}^{k,l}_{\infty,\bigcap_{s_2>0} \Pi_{\mathrm{an},[s_1,r_{1}]\times [s_2,r_{2}],I,\Pi_{\mathrm{an},[t_I,\infty],I,A}(\Gamma_{K_I})}(\pi_{K_I})}.
\end{align}
This implies that the cohomologies of the original double complexes are actually preserved when we change the boxes. Then for the remaining comparison we argue in the parallel way.
Then to finish, we choose two well-located multi intervals such that one embeds into another (which exists trivially in our situation which is a little bit different from the situation in \cite{KL3}):
\[
\xymatrix@R+6pc@C+3pc{
 0\ar[r] \ar[r] \ar[r]  &C^\bullet_{\psi_{I\backslash \{I\}}}(M_{...,[s'_I,r'_I]}) \ar[r]^{\psi_I-1} \ar[r] \ar[r] \ar[d]\ar[d]\ar[d]  & C^\bullet_{\psi_{I\backslash \{I\}}}(M_{...,[ps'_I,r'_I]})\ar[r] \ar[r] \ar[r] \ar[d]\ar[d]\ar[d]  &0\\
 0\ar[r] \ar[r] \ar[r]  &C^\bullet_{\psi_{I\backslash \{I\}}}(M_{...,[s''_I,r''_I]}) \ar[r]^{\psi_I-1} \ar[r] \ar[r]  & C^\bullet_{\psi_{I\backslash \{I\}}}(M_{...,[ps''_I,r''_I]})\ar[r] \ar[r] \ar[r] &0
,}
\]
or more explicit:
\[
\xymatrix@R+7pc@C+3pc{
M_{[s'_1,r'_1]\times[s'_2,r'_2]}\ar[r] \ar[r] \ar[r] \ar[d] \ar[d] \ar[d] & M_{[ps'_1,r'_1]\times[s'_2,r'_2]}\oplus M_{[s'_1,r'_1]\times[ps'_2,r'_2]}\ar[r] \ar[r] \ar[r]\ar[d] \ar[d] \ar[d] & M_{[ps'_1,r'_1]\times[ps'_2,r'_2]} \ar[d]\ar[d]\ar[d] \\
M_{[s''_1,r''_1]\times[s''_2,r''_2]} \ar[r] \ar[r] \ar[r]  & M_{[ps''_1,r''_1]\times[s''_2,r''_2]}\oplus M_{[s''_1,r''_1]\times[ps''_2,r''_2]}  \ar[r] \ar[r] \ar[r] & M_{[ps''_1,r''_1]\times[ps''_2,r''_2]}.\\
}
\]
Here the maps between objects over different multi intervals are induced exactly from the  inclusion of multi intervals. Then by the fundamental lemma for in \cite[Satz 5.2]{Kie1} and \cite[Lemma 1.10]{KL3} on the continuous completely maps  we get the finiteness when we restrict to $[t_I,\infty]$, since the map according to the embeddings of the multi intervals are quasi-isomorphisms. For the definition of continuous completely maps, we refer to \cite[Definition 1.3]{KL3}. The vertical homomorphisms here are actually continuous completely since we are actually embedding a smaller rigid multi-annuli to a larger one. Then the cohomology groups can be arranged to be a coherent sheaf over each section with respect to each such $[t_I,\infty]$, then the coadmissibility follows.
\end{proof}

\begin{proposition}
Based on the previous proposition we have:
\begin{displaymath}
C_{\psi_I}(M)\in  \mathbb{D}^-_\mathrm{perf}(\Pi_{\mathrm{an},\infty,I,A}(\Gamma_{K_I})),	
\end{displaymath}
for $A=\mathbb{Q}_p$.	
\end{proposition}

\begin{proof}
We have the corresponding coadmissibility from the previous proposition, then we work in the category of all the coadmissible modules over the ring $\Pi_{\mathrm{an},\infty,I,A}(\Gamma_{K_I})$. Then by applying the corresponding result \cite[Theorem 3.10]{Zab1} we can show that the complex is quasi-isomorphic to a corresponding complex being finite projective at each degree in the following. By base change to some $\Pi_{\mathrm{an},[t_I,\infty],I,A}(\Gamma_{K_I})$ for some $t_I$ we have that any such complex over $\Pi_{\mathrm{an},[t_I,\infty],I,A}(\Gamma_{K_I})$ is quasi-isomorphic to a complex being finitely generated at each degree. Therefore the corresponding original complex is quasi-isomorphic to a corresponding complex being coadmissible at each degree. Then in fact at each degree this latter complex is also projective (in the category of all the coadmissible modules over $\Pi_{\mathrm{an},\infty,I,A}(\Gamma_{K_I})$), since this will be further equivalent to that each section over $\Pi_{\mathrm{an},[t_I,\infty],I,A}(\Gamma_{K_I})$ is projective for some $t_I$, which is the case in our situation. Then we are done by \cite[Theorem 3.10]{Zab1} which implies that coadmissible and projective modules are finitely generated.


\end{proof}

\indent We then give a direct proof on the finiteness of the cohomology of $(\varphi_I,\Gamma_I)$-modules:

\begin{proposition} \label{proposition4.9}
With the notations above, we have that 
\begin{displaymath}
C^\bullet_{\varphi_I,\Gamma_{I}}(M)\in \mathbb{D}_\mathrm{perf}^-(A), C^\bullet_{\psi_I,\Gamma_{I}}(M)\in \mathbb{D}_\mathrm{perf}^-(A).
\end{displaymath}
\end{proposition}

\begin{proof}
This can be proved in the \'etale setting after \cite{PZ19} and \cite{CKZ18} in the absolute situation, by using Galois cohomology and the Hochschild-Serre spectral sequences directly. We use the Cartan-Serre method to prove this in the following fashion, modeled on the context of \cite{KL3}. By simultaneous induction to the base cases for the field $\mathbb{Q}_p$ from all $K_I$, we can now assume that we are considering the situation where all $K_I$ are $\mathbb{Q}_p$. We choose to prove the result for $(\psi_I,\Gamma_I)$-modules. In our situation we look at the continuous complexes coming from the  $\Gamma_I$-action, which gives us the complex $C^\bullet_{\mathrm{con}}(\Gamma_{K_I},M)$. Then the idea is to look at the following double complex of the complex $C^\bullet_{\mathrm{con}}(\Gamma_{K_I},M)$ (to be more precise we need to look at the totalization):
\[
\xymatrix@R+6pc@C+2pc{
C^\bullet_{\mathrm{con}}(\Gamma_{K_I},M_{[s_1,r_1]\times[s_2,r_2]}) \ar[r]^{\psi_1-1} \ar[r] \ar[r] \ar[d]^{\psi_2-1} \ar[d] \ar[d] & C^\bullet_{\mathrm{con}}(\Gamma_{K_I},M_{[ps_1,r_1]\times[s_2,r_2]}) \ar[d] \ar[d] \ar[d]\\
C^\bullet_{\mathrm{con}}(\Gamma_{K_I},M_{[s_1,r_1]\times[ps_2,r_2]})\ar[r] \ar[r] \ar[r]  &C^\bullet_{\mathrm{con}}(\Gamma_{K_I},M_{[ps_1,r_1]\times[ps_2,r_2]}). 
}
\]
Here the quasi-isomorphism among different boxes (e.g. multi intervals) is not actually trivial. For the serious argument after the knowledge on the equivalence between the categories of $\varphi_I$-modules over the full Robba ring, $\varphi_I$-Frobenius bundles and $\varphi_I$-modules over the Robba ring with respect to some single multi interval, which can be realized in our context by using \cref{proposition3.5} as in \cite[Lemma 2.10]{KP1}, we refer to \cite[Lemma 5.2]{KP1} and \cite[Proposition 5.12, Theorem 5.7.11]{KL2} for the comparison on the cohomology groups. \\
\indent To show that the cohomology groups do not change when we switch from a multi-interval $[s_1',r_1']\times [s_2',r_2']$ to a multi-interval $[s_1'',r_1'']\times [s_2'',r_2'']$, we choose to look at an intermediate multi-interval $[s_1'',r_1'']\times [s_2',r_2']$ namely we consider the following in order rings over multi-intervals:
\begin{align}
&\Pi_{\mathrm{an},[s'_1,r'_1]\times [s'_2,r'_2],I,A}(\pi_{K_I})\\
&\Pi_{\mathrm{an},[s''_1,r''_1]\times [s'_2,r'_2],I,A}(\pi_{K_I})\\
&\Pi_{\mathrm{an},[s''_1,r''_1]\times [s''_2,r''_2],I,A}(\pi_{K_I})	
\end{align}
and compare the first two and the last two. To compare with respect to the first two multi-intervals we first consider the $\psi_1$-complex, the cohomology groups can be regarded as the  Yoneda extension groups over the twisted polynomial ring in variable of $\psi_1$ with coefficients in the Robba rings, which shows that the $\psi_I$-cohomology groups do not change when we change the boxes. Here we perform the spectral sequential technique twice. First we consider a fixed $m$:
\[
\xymatrix@R+6pc@C+2pc{
C^m_{\mathrm{con}}(\Gamma_{K_I},M_{[s_1,r_1]\times[s_2,r_2]}) \ar[r]^{\psi_1-1} \ar[r] \ar[r] \ar[d]^{\psi_2-1} \ar[d] \ar[d] & C^m_{\mathrm{con}}(\Gamma_{K_I},M_{[ps_1,r_1]\times[s_2,r_2]}) \ar[d] \ar[d] \ar[d]\\
C^m_{\mathrm{con}}(\Gamma_{K_I},M_{[s_1,r_1]\times[ps_2,r_2]})\ar[r] \ar[r] \ar[r]  &C^m_{\mathrm{con}}(\Gamma_{K_I},M_{[ps_1,r_1]\times[ps_2,r_2]}). 
}
\]
Then we consider the corresponding spectral sequence attached to this layer $\square_m$. To be more precise, we form the spectral sequences for the two intervals in comparisom coming from the overall double complexes:
\begin{align}
&\mathrm{E}^{i,j}_{*,[s'_1,r'_1]\times [s'_2,r'_2]}(\square_m),\\
&\mathrm{E}^{k,l}_{*,[s''_1,r''_1]\times [s'_2,r'_2]}(\square_m).
\end{align}
Here one considers the $0$-th page and the $1$-th page respectively, one can see directly that the isomorphism beyond the $1$-page:
\begin{align}
\mathrm{E}^{i,j}_{q,[s'_1,r'_1]\times [s'_2,r'_2]}(\square_m)\overset{\sim}{\longrightarrow}\mathrm{E}^{k,l}_{q,[s''_1,r''_1]\times [s'_2,r'_2]}(\square_m),~q>1.
\end{align}
This implies that:
\begin{align}
\mathrm{E}^{i,j}_{\infty,[s'_1,r'_1]\times [s'_2,r'_2]}(\square_m)\overset{\sim}{\longrightarrow}\mathrm{E}^{k,l}_{\infty,[s''_1,r''_1]\times [s'_2,r'_2]}(\square_m).
\end{align}
This implies that the cohomologies of the double complexes for this layer $\square_m$ are actually preserved when we change the boxes. That is to say  the totalization of this layer $\mathbf{Total}(\square_m)$ will be quasi-isomorphic when we change boxes. Then we consider the double complex $\mathbf{Total}(\square_\bullet)$, then a second stage of the comparison of the $\infty$-pages of the spectral sequence will show eventually that the totalization $\mathbf{Total}\left(\mathbf{Total}(\square_\bullet)\right)$ of $\mathbf{Total}(\square_\bullet)$ will have the isomorphic cohomology groups when we change boxes. Then for the remaining comparison we argue in the parallel way. \\
\indent To show that the cohomology groups do not change when we switch from Robba ring over some fixed multi-radii $r_{1}\times r_{2}$ to some interval $[s_1,r_1]\times [s_2,r_2]$, we consider the intermediate Robba ring:
\begin{align}
\bigcap_{s_2>0} \Pi_{\mathrm{an},[s_1,r_1]\times [s_2,r_{2}],I,A}(\pi_{K_I})	
\end{align}
and then in order the following rings:
\begin{align}
&\Gamma_1:\bigcap_{s_1>0,s_2>0} \Pi_{\mathrm{an},[s_1,r_{1}]\times [s_2,r_{2}],I,A}(\pi_{K_I}),\\
&\Gamma_2:\bigcap_{s_2>0} \Pi_{\mathrm{an},[s_1,r_{1}]\times [s_2,r_{2}],I,A}(\pi_{K_I}),\\
&\Gamma_3: \Pi_{\mathrm{an},[s_1,r_{1}]\times [s_2,r_{2}],I,A}(\pi_{K_I}).
\end{align}
Note that finite projective modules over these are equivalent when carrying the  $\varphi_I$-structure. To compare with respect to the first two situations, we first consider the $\psi_1$-cohomology complex, then by using the  Yoneda extension we have the  $\psi_1$-cohomology groups are the same, which implies further that the  $\psi_I$-cohomology groups are the same. Here we perform the spectral sequential technique twice. First we consider a fixed $m$:
\[
\xymatrix@R+6pc@C+2pc{
C^m_{\mathrm{con}}(\Gamma_{K_I},M) \ar[r]^{\psi_1-1} \ar[r] \ar[r] \ar[d]^{\psi_2-1} \ar[d] \ar[d] & C^m_{\mathrm{con}}(\Gamma_{K_I},M) \ar[d] \ar[d] \ar[d]\\
C^m_{\mathrm{con}}(\Gamma_{K_I},M)\ar[r] \ar[r] \ar[r]  &C^m_{\mathrm{con}}(\Gamma_{K_I},M). 
}
\]
Then we consider the corresponding spectral sequence attached to this layer $\square_m$. To be more precise, we form the spectral sequences for the two intervals in comparisom coming from the overall double complexes:
\begin{align}
\mathrm{E}^{i,j}_{*,\Gamma_1}(\square_m),\\
\mathrm{E}^{k,l}_{*,\Gamma_2}(\square_m).
\end{align}
Here one considers the $0$-th page and the $1$-th page respectively, one can see directly that the isomorphism beyond the $1$-page:
\begin{align}
\mathrm{E}^{i,j}_{q,\Gamma_1}(\square_m)\overset{\sim}{\longrightarrow}\mathrm{E}^{k,l}_{q,\Gamma_2}(\square_m),~q>1.
\end{align}
This implies that:
\begin{align}
\mathrm{E}^{i,j}_{\infty,\Gamma_1}(\square_m)\overset{\sim}{\longrightarrow}\mathrm{E}^{k,l}_{\infty,\Gamma_2}(\square_m).
\end{align}
This implies that the cohomologies of the double complexes for this layer $\square_m$ are actually preserved when we change the boxes. That is to say the totalization of this layer $\mathbf{Total}(\square_m)$ will be quasi-isomorphic when we change boxes. Then we consider the double complex $\mathbf{Total}(\square_\bullet)$, then a second stage of the comparison of the $\infty$-pages of the spectral sequence will show eventually that the totalization $\mathbf{Total}\left(\mathbf{Total}(\square_\bullet)\right)$ of $\mathbf{Total}(\square_\bullet)$ will have the isomorphic cohomology groups when we change boxes. Then for the remaining comparison we argue in the parallel way.\\
\indent Then in order to apply the  fundamental lemma namely \cite[Lemma 1.10]{KL3}, we need to choose suitable intervals in our multi setting where one embeds into the other which gives rise to the following diagram:
\[
\xymatrix@R+6pc@C+0pc{
 0\ar[r] \ar[r] \ar[r]  &C^\bullet_{\psi_{I\backslash \{I\}}}(C^\bullet_{\mathrm{con}}(\Gamma_{K_I},M_{...,[s'_I,r'_I]})) \ar[r]^{\psi_I-1} \ar[r] \ar[r] \ar[d]\ar[d]\ar[d]  & C^\bullet_{\psi_{I\backslash \{I\}}}(C^\bullet_{\mathrm{con}}(\Gamma_{K_I},M_{...,[ps'_I,r'_I]}))\ar[r] \ar[r] \ar[r] \ar[d]\ar[d]\ar[d]  &0\\
 0\ar[r] \ar[r] \ar[r]  &C^\bullet_{\psi_{I\backslash \{I\}}}(C^\bullet_{\mathrm{con}}(\Gamma_{K_I},M_{...,[s''_I,r''_I]})) \ar[r]^{\psi_I-1} \ar[r] \ar[r]  & C^\bullet_{\psi_{I\backslash \{I\}}}(C^\bullet_{\mathrm{con}}(\Gamma_{K_I},M_{...,[ps''_I,r''_I]}))\ar[r] \ar[r] \ar[r] &0
.}
\]
To be more explicit this gives rise to the following diagram:
\[\small
\xymatrix@R+7pc@C+0pc{
C^\bullet_{\Gamma_I}(M_{[s'_1,r'_1]\times[s'_2,r'_2]}) \ar[r] \ar[r] \ar[r] \ar[d] \ar[d] \ar[d] & C^\bullet_{\Gamma_I}(M_{[ps'_1,r'_1]\times[s'_2,r'_2]})\oplus C^\bullet_{\Gamma_I}(M_{[s'_1,r'_1]\times[ps'_2,r'_2]})\ar[r] \ar[r] \ar[r]\ar[d] \ar[d] \ar[d] & C^\bullet_{\Gamma_I}(M_{[ps'_1,r'_1]\times[ps'_2,r'_2]})\ar[d]\ar[d]\ar[d] \\
C^\bullet_{\Gamma_I}(M_{[s''_1,r''_1]\times[s''_2,r''_2]}) \ar[r] \ar[r] \ar[r] & C^\bullet_{\Gamma_I}(M_{[ps''_1,r''_1]\times[s''_2,r''_2]})\oplus C^\bullet_{\Gamma_I}(M_{[s''_1,r''_1]\times[ps''_2,r''_2]})\ar[r] \ar[r] \ar[r] & C^\bullet_{\Gamma_I}(M_{[ps''_1,r''_1]\times[ps''_2,r''_2]}).\\
}
\]
Here the  maps between objects over different multi intervals are induced exactly from the  inclusion of multi intervals. Then we consider the  fundamental lemma \cite[Satz 5.2]{Kie1} and \cite[Lemma 1.10]{KL3} which implies the  finiteness, after one further takes the totalizations of the horizontal complexes in this final diagram. And by our basic construction the result follows. Over $A$, any bounded complex ($C^\bullet$) in $\mathbb{D}^\flat(A)$ will be automatically perfect whenever we have the perfectness of the cohomology groups $h^\bullet(C^\bullet)$ which is the case in our setting as proved above.\end{proof}

\indent Now we consider a finite projective $(\varphi_I,\Gamma_I)$-module $M$ over the ring $\Pi_{\mathrm{an},\mathrm{con},I,\Pi_{\mathrm{an},\infty,A}(G)}$, namely we have:
\begin{displaymath}
M=\varprojlim_k M_k,	
\end{displaymath}
where each object $M_k$ satisfies the corresponding framework in the above discussion.

\begin{definition} \mbox{\bf{(After KPX \cite[2.3]{KPX})}}
For any $(\varphi_I,\Gamma_I)$-module $M$ over $\Pi_{\mathrm{an},\mathrm{con},I,\Pi_{\mathrm{an},\infty,A}(G)}$ we define the complex $C^\bullet_{\Gamma_I}(M)$ of $M$ to be the total complex of the following complex through the induction:
\[
\xymatrix@R+0pc@C+0pc{
[ C^\bullet_{\Gamma_{I\backslash\{I\}}}(M) \ar[r]^{\Gamma_I} \ar[r] \ar[r]  & C^\bullet_{\Gamma_I\backslash\{I\}}(M)]
.}
\]
Then we define the corresponding double complex $C^\bullet_{\varphi_I}C^\bullet_{\Gamma_I}(M)$ again by taking the corresponding totalization of the following complex through induction:
\[
\xymatrix@R+0pc@C+0pc{
[C^\bullet_{\varphi_{I\backslash\{I\}}}C^\bullet_{\Gamma_{I}}(M) \ar[r]^{\varphi_I} \ar[r] \ar[r]  & C^\bullet_{\varphi_{I\backslash\{I\}}}C^\bullet_{\Gamma_I}(M)]
.}
\]
Then we define the complex $C^\bullet_{\varphi_I,\Gamma_I}(M)$ to be the totalization of the double complex defined above, which is called the complex of a $(\varphi_I,\Gamma_I)$-module $M$. Similarly we define the corresponding complex $C^\bullet_{\psi_I,\Gamma_I}(M)$ to be the totalization of the double complex define  d in the same way as above by replacing $\varphi_I$ by then the operators $\psi_I$. For later comparison we also need the corresponding $\psi_I$-cohomology, which is to say the complex $C^\bullet_{\psi_I}(M)$ which could be defined by the totalization of the following complex through induction:
\[
\xymatrix@R+0pc@C+0pc{
[ C^\bullet_{\psi_{I\backslash\{I\}}}(M) \ar[r]^{\psi_I} \ar[r] \ar[r]  & C^\bullet_{\psi_I\backslash\{I\}}(M)]
.}
\]
\end{definition}

\begin{proposition} 
With the notation and setting up as above we have $h^\bullet(C_{\psi_I}(M))$ are coadmissible over $\Pi_{\mathrm{an},\infty,A \widehat{\otimes} \Pi_{\mathrm{an},\infty}(\Gamma_{K_I})}(G)$. 	
\end{proposition}

\begin{proof}
By simultaneous induction to the base cases for the field $\mathbb{Q}_p$ from all $K_I$, we can now assume that we are considering the situation where all $K_I$ are $\mathbb{Q}_p$. Then for the complex $C^\bullet_{\psi_I}(M)$, by definition one could then regard this complex as the double complex taking the following form (after considering the corresponding totalization):
\[
\xymatrix@R+5pc@C+5pc{
M \ar[r]^{\psi_1-1} \ar[r] \ar[r] \ar[d]^{\psi_2-1} \ar[d] \ar[d] & M \ar[d] \ar[d] \ar[d]\\
M \ar[r] \ar[r] \ar[r]  & M. 
}
\]


\indent The boundedness of the complex could be observed from the definition and the corresponding ideas in \cite[Theorem 3.3, Theorem 3.5]{KL3}. For the finiteness we construct suitable completely continuous maps comparing the corresponding quasi-isomorphic complexes for $M$ as below. Write $\Pi_{\mathrm{an},\infty,A}(G)$ as the following form:
\begin{displaymath}
\Pi_{\mathrm{an},\infty,A }(G):=\varprojlim_k \Pi_{\mathrm{an},\infty,A}(G)_k,	
\end{displaymath}
as a Fr\'echet-Stein algebra. Then take some $\Pi_{\mathrm{an},\infty,A}(G)_k$ in the family with the corresponding $M_k$. By our result above for $M_k$ \cref{proposition4.7} we have that $h^\bullet(C_{\psi_I}(M_k))$ are coadmissible over $\Pi_{\mathrm{an},\infty,\Pi_{\mathrm{an},\infty,A}(G)_k}(\Gamma_{K_I})$. Therefore we have the the corresponding cohomology groups $h^\bullet(C_{\psi_I}(M))$ are actually coming from (the global sections of) some coherent sheaves over $\Pi_{\mathrm{an},\infty,\Pi_{\mathrm{an},\infty,A}(G)}(\Gamma_{K_I})$.

\end{proof}

\begin{proposition}
Based on the previous proposition we have:
\begin{displaymath}
C_{\psi_I}(M)\in  \mathbb{D}^-_\mathrm{perf}(\Pi_{\mathrm{an},\infty,A \widehat{\otimes} \Pi_{\mathrm{an},\infty}(\Gamma_{K_I})}(G)),	
\end{displaymath}
for $A=\mathbb{Q}_p$.	
\end{proposition}

\begin{proof}
We have the corresponding coadmissibility from the previous proposition, then we work in the category of all the coadmissible modules over the ring $\Pi_{\mathrm{an},\infty,I,A}(\Gamma_{K_I})$. Then by applying the corresponding result \cite[Theorem 3.10]{Zab1} we can show that the complex is quasi-isomorphic to a corresponding complex being finite projective at each degree in the following. Set:
\begin{align}
\Pi_{\mathrm{an},\infty,A \widehat{\otimes} \Pi_{\mathrm{an},\infty}(\Gamma_{K_I})}(G):=\varprojlim_k \Pi_{\mathrm{an},\infty,A \widehat{\otimes} \Pi_{\mathrm{an},\infty}(\Gamma_{K_I})}(G)_k.	
\end{align}

By base change to some $\Pi_{\mathrm{an},\infty,A \widehat{\otimes} \Pi_{\mathrm{an},\infty}(\Gamma_{K_I})}(G)_k$ for some $k$ we have that any such complex over $ \Pi_{\mathrm{an},\infty,A \widehat{\otimes} \Pi_{\mathrm{an},\infty}(\Gamma_{K_I})}(G)_k$ is quasi-isomorphic to a complex being finitely generated at each degree. Therefore the corresponding original complex is quasi-isomorphic to a corresponding complex being coadmissible at each degree. Then in fact at each degree this latter complex is also projective (in the category of all the coadmissible modules over $\varprojlim_k \Pi_{\mathrm{an},\infty,A \widehat{\otimes} \Pi_{\mathrm{an},\infty}(\Gamma_{K_I})}(G)_k$), since this will be further equivalent to that each section over $\varprojlim_k \Pi_{\mathrm{an},\infty,A \widehat{\otimes} \Pi_{\mathrm{an},\infty}(\Gamma_{K_I})}(G)_k$ is projective for some $k$, which is the case in our situation. Then we are done by \cite[Theorem 3.10]{Zab1} which implies that coadmissible and projective modules are finitely generated.


\end{proof}

\indent We then give a direct proof on the finiteness of the cohomology of $(\varphi_I,\Gamma_I)$-modules:

\begin{proposition}
With the notations above, we have that 
\begin{displaymath}
C^\bullet_{\varphi_I,\Gamma_{I}}(M)\in \mathbb{D}_\mathrm{perf}^-(\Pi_{\mathrm{an},\infty,A }(G)), C^\bullet_{\psi_I,\Gamma_{I}}(M)\in \mathbb{D}_\mathrm{perf}^-(\Pi_{\mathrm{an},\infty,A }(G)).
\end{displaymath}
Here again we assume that the corresponding ring $A$ is $\mathbb{Q}_p$.
\end{proposition}

\begin{proof}
We choose to prove the corresponding result for $(\psi_I,\Gamma_I)$-modules. As above we only have to consider the corresponding situation where $K_I$ are all $\mathbb{Q}_p$ by induction in our current context. Write $\Pi_{\mathrm{an},\infty,A }(G)$ as the limit:
\begin{displaymath}
\Pi_{\mathrm{an},\infty,A }(G)=\varprojlim_k \Pi_{\mathrm{an},\infty,A }(G)_k.	
\end{displaymath}
Working over the ring $\Pi_{\mathrm{an},\infty,A }(G)_k$ we have the corresponding finiteness by the result from \cref{proposition4.9}, which will show that $C^\bullet_{\psi_I,\Gamma_{I}}(M)$ are all coadmissible modules over $\varprojlim_k \Pi_{\mathrm{an},\infty,A }(G)_k$. By base change to some $\Pi_{\mathrm{an},\infty,A }(G)_k$ for some $k$ we have that any such complex over $ \Pi_{\mathrm{an},\infty,A }(G)_k$ is quasi-isomorphic to a complex being finitely generated at each degree. Therefore the corresponding original complex is quasi-isomorphic to a corresponding complex being coadmissible at each degree. Then in fact at each degree this latter complex is also projective (in the category of all the coadmissible modules over $\varprojlim_k \Pi_{\mathrm{an},\infty,A }(G)_k$), since this will be further equivalent to that each section over $\varprojlim_k \Pi_{\mathrm{an},\infty,A }(G)_k$ is projective for some $k$, which is the case in our situation. Then we are done by \cite[Theorem 3.10]{Zab1} which implies that coadmissible and projective modules are finitely generated.\\ 
\end{proof}

\subsection{Application to Triangulation over Rigid Analytic Spaces} \label{triangulation}

\noindent We are now going to establish some application of the finiteness of the corresponding complexes with respect to the action $(\varphi_I,\Gamma_I)$. We will generalize the interesting globalization of the triangulation established in \cite{KPX} to our generalized setting. This has its own interests since it will be very natural to ask about the behavior of some triangulation in our context over the corresponding interesting rigid families of the generalized Hodge structures.\\

\indent Since we only proved the corresponding finiteness in the situation where the set $I$ consists of two elements, therefore in this section we are going to make the same assumption. Then we consider the corresponding finite free rank one $(\varphi_I,\Gamma_I)$-modules coming from the corresponding algebraic characters from the product of $\Gamma_{K_I}$, and moreover from the product of the groups $K_I^\times$. Actually in the one dimensional situation this consideration finishes the corresponding construction of all the finite free rank one $(\varphi_I,\Gamma_I)$-modules in a reasonably well-defined way. We follow largely the corresponding notations in the one dimensional situation. We will assume that over some strong (since we do not have a well-established classifications as mentioned above) pointwisely we have the corresponding strong triangulation (meaning that we have the triangulations by using the parameters strictly coming from the corresponding continuous characters of the group $\Gamma_{K_I}^\times$), then we would like to study the corresponding behavior of the corresponding global triangulations. This means that in order to have some coherent triangulations and spreading effect, one only needs to check things pointwisely.\\

\begin{setting}
In this section we are going to work over a general rigid analytic space $X$ defined over suitable finite extension of $L$, where we assume that $L$ is large enough with respect to the fields $K_I$ generalizing the corresponding situation in \cite[Chapter 6]{KPX}.	
\end{setting}

\indent We start from the following geometric setting:

\begin{definition} \mbox{\bf{(After KPX \cite[Definition 6.2.1]{KPX})}}
First we are going to consider the corresponding sheaves $\Pi_{\mathrm{an},\mathrm{con},I,X}(\pi_{K_I})$ with respect to the analytic space $X$ over $L$, which is defined to be glueing the rings of analytic functions over each affinoid $\mathrm{Max}A$ for the space $X$, namely glueing the rings having the form of $\Pi_{\mathrm{an},\mathrm{con},I,A}(\pi_{K_I})$ with respect to this affinoid. Also as in \cite[Definition 6.2.1]{KPX} one can define the corresponding sheaf with respect to some specific radii which is to say the sheaves taking the form of $\Pi_{\mathrm{an},r_{I},I,X}(\pi_{K_I})$. Then we define the corresponding $(\varphi_I,\Gamma_I)$-modules over these two kinds of sheaves by considering the corresponding families of finite locally free sheaves of modules over $\Pi_{\mathrm{an},\mathrm{con},I,X}(\pi_{K_I})$ or $\Pi_{\mathrm{an},r_{I},I,X}(\pi_{K_I})$ in the well-defined and compatible sense as in the one-dimensional situation in \cite[Definition 6.2.1]{KPX}.

\end{definition}

\begin{setting}
We choose to now consider the following situation where all the fields $K_I$ are assumed to be $\mathbb{Q}_p$. Therefore now the $p$-adic Lie groups we are considering take the following form namely the product of $\mathbb{Q}_p^\times$.	
\end{setting}

\indent Under this assumption namely where all the fields $K_I$ are $\mathbb{Q}_p$ one could directly consider the corresponding construction as generalized below from the one dimensional situation considered in \cite[Notation 6.2.2]{KPX}:

\begin{definition} \mbox{\bf{(After KPX \cite[Notation 6.2.2]{KPX})}}
Consider the product of the $p$-adic Lie group $\mathbb{Q}_p^\times$ which is to say $\prod_{\alpha\in I}\mathbb{Q}_p^\times$. We are going to define the corresponding free rank one $(\varphi_I,\Gamma_I)$-modules of the character types attached to any character taking the form of $\delta_I:\prod_{\alpha\in I}\mathbb{Q}_p^\times\rightarrow \Gamma(X,\mathcal{O}_X)^\times$, to be the free of rank one module $\Pi_{\mathrm{an},\mathrm{con},I,X}(\pi_{K_I})\mathbf{e}$ such that we have for each $\alpha\in I$, the Frobenius $\varphi_\alpha$ acts via $\delta_I(p_\alpha)$ while the group $\Gamma_\alpha$ acts via $\delta_I(\chi_\alpha(\gamma_\alpha))$. Note that when we have that the set $I$ is a singleton the corresponding definition recovers the one defined in \cite[Notation 6.2.2]{KPX}. We will use the notation $\Pi_{\mathrm{an},\mathrm{con},I,X}(\pi_{K_I})(\delta_I)$ to denote such free rank one object in the general setting.	
\end{definition}

\indent On the other hand actually we have no direct classification literally as in the one dimensional situation on the free rank one objects at least up to this definition. Therefore we have to at this moment to fix our objects considered, which is to say that we will focus on all the free rank one objects of character types as defined above, then to study the corresponding geometrization and the variation of the corresponding triangulation. We start from the following basic definition which is a direct generalization of the corresponding one in the one dimensional situation.

\begin{definition} \mbox{\bf{(After KPX \cite[Definition 6.3.1]{KPX})}}
We define the $triangulated$ $(\varphi_I,\Gamma_I)$-modules over the higher dimensional Robba rings over a rigid analytic space $X$ taking the form of $\Pi_{\mathrm{an},\mathrm{con},I,X}(\pi_{K_I})$. Such a module $M$ is defined to be a usual $(\varphi_I,\Gamma_I)$-module over $\Pi_{\mathrm{an},\mathrm{con},I,X}(\pi_{K_I})$ carrying a triangulation with $n$-parameters $\delta_1,...,\delta_n$ which are $n$ characters (we will assume that these are continuous ones) of the group $\prod_{I}\mathbb{Q}_p^\times$ with the values in the corresponding group $\Gamma(X,\mathcal{O}_X)^\times$, which means that we have the following filtration on $M$ by the corresponding sub $(\varphi_I,\Gamma_I)$-modules over the sheaf $\Pi_{\mathrm{an},\mathrm{con},I,X}(\pi_{K_I})$:
\begin{displaymath}
0=M_0\subset M_1\subset...\subset M_n=M	
\end{displaymath}
where $n$ is assumed in our situation to be the rank of $M$, moreover we have that the corresponding each graded piece $M_{i}/M_{i-1}$ for $i=1,...,n$ is a rank one $(\varphi_I,\Gamma_I)$-module with the corresponding parameter $\delta_i$ defined above with some twisted line bundle $\mathbb{L}_i,i=1,...,n$:
\begin{displaymath}
M_{i}/M_{i-1}\overset{\sim}{\longrightarrow}\Pi_{\mathrm{an},\mathrm{con},I,X}(\pi_{K_I})(\delta_i)\otimes \mathbb{L}_i,i=1,...,n.	
\end{displaymath}
\end{definition}

\begin{definition} \mbox{\bf{(After KPX \cite[Definition 6.3.1]{KPX})}}
As in \cite[Definition 6.3.1]{KPX} we generalize the corresponding $strictly~triangulated~modules$ to the higher dimensional situation by defining them to be over $X=\mathrm{Max}(L)$ the triangulated modules as above in a manner of building from each $i$-th filtration to the $i+1$-th uniquely through the corresponding parameters.
\end{definition}

\indent We are going to consider the corresponding integrated behavior of a dense subset of pointwise triangulations. First recall that a subspace $U$ is $Zariski~dense$ if for each member in an admissible covering of the space $X$ the space $U$ is dense inside (see \cite[Chapter 6.3]{KPX}). Then we can now consider the corresponding densely pointwise strictly triangulated families:

\begin{definition} \mbox{\bf{(After KPX \cite[Definition 6.3.2]{KPX})}}
We are going to call a $(\varphi_I,\Gamma_I)$-module over the sheaf $\Pi_{\mathrm{an},\mathrm{con},I,X}(\pi_{K_I})$ a $densely~pointwise~strictly~triangulated~family$ if there exists a Zariski-dense subspace $X_{dpt}$ inside $X$ such that the corresponding the restriction to $X_{dpt}$ of the sheaf $M$ is pointwise strictly triangulated in the sense defined above.	
\end{definition}

\begin{lemma}\mbox{\bf{(After KPX \cite[6.3.3]{KPX})}}
In our situation, we have the cohomology groups $H^\bullet_{\varphi_I,\Gamma_I}(M)$ are coherent sheaves over the rigid analytic space $X$ for $M$ any $(\varphi_I,\Gamma_I)$-modules over the Frobenius sheaves $\Pi_{\mathrm{an},\mathrm{con},I,X}(\pi_{K_I})$ with the corresponding non-vanishing throughout the degrees in $[0,2I]$.	
\end{lemma}

\begin{proof}
The proof of this result is parallel to the corresponding result in the one dimensional situation, see \cite[6.3.3]{KPX}. Essentially one has this result as a direct consequence of the results in the local setting in \cref{section4.1}.
\end{proof}

%
%

\indent We then consider our main results in this section, which is a direct generalization of the results of \cite{KPX} (chapter 6) to the situation where we have higher dimensional Hodge-Frobenius structure. 


\begin{proposition} \mbox{\bf{(After KPX \cite[Theorem 6.3.9]{KPX})}}
Suppose that we have a $(\varphi_I,\Gamma_I)$-module $M$ over the ring $\Pi_{\mathrm{an},\mathrm{con},I,X}(\pi_{K_I})$, where we assume that this module is of rank $k$. Suppose that we have a continuous character from the group $\prod_I \mathbb{Q}_p^\times$ into the group $\Gamma(X,\mathcal{O}_X)^\times$. Suppose that there exists a subset $X_{pdt}\subset X$ over which we have that the cohomology group $H^0_{\varphi_I,\Gamma_I}(M_x^\vee(\delta_x))$ is of dimension one (which is to say for each $x\in X_{pdt}$), with the further assumption that the image of the corresponding ring $\Pi_{\mathrm{an},\mathrm{con},I,X}(\pi_{K_I})$ is saturated under this assumption (namely under the corresponding consideration of the basis). Assume this subset is Zariski-dense. Assume now the space $X$ is reduced, then one could be able to find a morphism from another blow-up $b:Y\rightarrow X$ which is birational and proper and a corresponding map $f:b^*M\rightarrow \Pi_{\mathrm{an},\mathrm{con},I,X}(\pi_{K_I})(\delta)\otimes_{\mathcal{O}_X}\mathbb{L}$ (with some well-defined line bundle $\mathbb{L}$ as in the situation of \cite[Theorem 6.3.9]{KPX}). 


\end{proposition}

\begin{proof}

Indeed, we can adapt the corresponding strategy of \cite[Theorem 6.3.9]{KPX} to prove this. First by taking the well-established (see \cite[Theorem 6.3.9]{KPX} the geometry actually) normalization we are reduced to assume that the space $X$ is normal and connected, and consequently the ranks of the corresponding coherent sheaves preserve. Then we consider the corresponding local coherence of the corresponding cohomology complex $C^\bullet_{\varphi_I,\Gamma_I}(M)$ (for all $\alpha\in I$), and the corresponding local base change results (namely the direct analog of the theorem 4.4.3 of \cite{KPX}), together with the 6.3.6 of \cite{KPX} we have the corresponding groups $H^\bullet_{\varphi_I,\Gamma_I}(M)$ are flat for $\bullet=0$ and having Tor dimension less than or equal to 1 above, after we replace $M$ by some $M_0$ which is the pullback of $M^\vee(\delta)$ along some $f_0:Y_0\rightarrow X$ which is birational and proper. We then glue along the corresponding local construction above as in \cite[Theorem 6.3.9]{KPX}. The first task for us then will be the corresponding construction of the morphism over this $Y_0$ which will be later be promoted to be the desired $Y$ by taking the corresponding blow-up. The construction goes as in the following way which is modeled on \cite[Theorem 6.3.9]{KPX}. One first considers the corresponding subset $Y_1'^c\subset Y_0$ that consists of all the points at which the corresponding forth term in the following exact sequence (by considering the corresponding base change spectral sequence)
\[
\xymatrix@R+0pc@C+0pc{
0\ar[r] \ar[r] \ar[r]  &H^0_{\varphi_I,\Gamma_I}(M_0)\otimes\kappa_y \ar[r] \ar[r] \ar[r]  &H^0_{\varphi_I,\Gamma_I}(M_{0,y}) \ar[r] \ar[r] \ar[r] &\mathrm{Tor}_1(H^1_{\varphi_I,\Gamma_I}(M_0),\kappa_y) \ar[r] \ar[r] \ar[r] &0
.}
\]
vanishes in our situation. Note that this is dense and open in the Zariski topology. Over the subset $Y_1'$ the resulting exact sequence degenerates just to be:
\[
\xymatrix@R+0pc@C+0pc{
0\ar[r] \ar[r] \ar[r]  &H^0_{\varphi_I,\Gamma_I}(M_0)\otimes\kappa_y \ar[r] \ar[r] \ar[r]  &H^0_{\varphi_I,\Gamma_I}(M_{0,y}) \ar[r] \ar[r] \ar[r] &\mathrm{Tor}_1(H^1_{\varphi_I,\Gamma_I}(M_0),\kappa_y) \ar[r] \ar[r] \ar[r] &0
,}
\]
where the forth term does not vanish. Then we note that over the subset $Y_1'^c$ we have that the corresponding resulting cohomology group $H^0_{\varphi_I,\Gamma_I}(M_0)$ is then of dimensional one, which implies that by the topological property of the subspace it is coherent of rank one throughout the whole space $Y_0$. We then have the chance to extract the desired morphism in our situation as below: the corresponding line bundle $\mathbb{L}$ is chosen to be the corresponding dual line bundle. Then we have the result morphism $\Pi_{\mathrm{an},\mathrm{con},I,X}(\pi_{K_I})\otimes \mathbb{L}^\vee\rightarrow M_0$ and consequently we have the corresponding desired morphism which is to say $f_0^*M\rightarrow \Pi_{\mathrm{an},\mathrm{con},I,X}(\pi_{K_I})(\delta)\otimes \mathbb{L}$. \\
\indent Then up to this point, one can then consider 6.3.6 of \cite{KPX} to extract the desired blow-up $Y\rightarrow Y_0$ and define the desired data through the pullback along blow-up, which is to say the corresponding $\mathbb{L}$, $f$ (see \cite[Theorem 6.3.9]{KPX}). 

\end{proof}

\begin{remark}
The corresponding morphism established by the previous proposition is actually expected to satisfy the corresponding properties parallel to the corresponding situation in \cite[Theorem 6.3.9, Corollary 6.3.10]{KPX}. Namely in the context mentioned in the proof of the previous proposition we are now at the position to (thanks to the previous proposition) conjecture: \\
\indent A. We have the corresponding subset $Y_1^c\subset Y$ of all the points $y$ where the corresponding map $f_y:M_y\rightarrow \Pi_{\mathrm{an},\mathrm{con},I,\kappa_y}(\pi_{K_I})(\delta_y)\otimes_{\mathcal{O}_X}\mathbb{L}_y$ is surjective (as the original context of \cite{KPX} we need to assume that the corresponding element generates the zero-th cohomology of $M_y^\vee(\delta_y)$ as well), is open and dense in the Zariski topology; B. The corresponding kernel $\mathbf{Ker}f$ is of rank just $k-1$.

\end{remark}

\newpage

\section{Noncommutative Coefficients}

\subsection{Frobenius Modules over Noncommutative Rings} \label{section5.1}

\noindent We now consider the coefficient where $A$ is a noetherian noncommutative Banach affinoid algebra over $\mathbb{Q}_p$ namely quotient of noncommutative Tate algebra with free variables. The finiteness even in the univariate situation with noncommutative coefficients is actually not known in \cite{Zah1} which is key ingredient in the noncommutative local $p$-adic Tamagawa number conjecture in \cite{Zah1}. We make the corresponding discussion in our situation.


\begin{example} \label{example5.1}
A nontrivial example of such noetherian noncommutative Banach affinoid algebra could be the corresponding noncommutative analog of the commutative Tate algebra $\mathbb{T}_2$ considered in \cite[Section 3]{So1}. To be more precise one considers first the corresponding twisted polynomial ring $\mathbb{Q}_p[X_1,X_2]'$ such that $X_1X_2=aX_2X_1$ for some nonzero element $a\in \mathbb{Q}_p$, then takes the corresponding completion with respect to the Gauss norm $\|.\|$ which is defined by:
\begin{displaymath}
\|\sum_{i_1,i_2}a_{i_1,i_2}X_1^{i_1}X_2^{i_2}\|:=\sup_{i_1,i_2} \{|a_{i_1,i_2}|_p\}.	
\end{displaymath}
Note that as in \cite[Proposition 3.1]{So1} the corresponding categories of finite left $A$-modules and the correspnding Banach ones are actually equivalent, under the forgetful functor. 	
\end{example}

\indent We then consider the families of the corresponding Frobenius modules with coefficients in noncommutative affinoid algebra $A$ considered as above.

\begin{setting}
For the Robba rings with respect to some closed interval $[s_I,r_I]\subset (0,\infty]^{|I|}$ we defined above namely $\Pi_{\mathrm{an},[s_I,r_I],I,\mathbb{Q}_p}$, we take the completed tensor product with the corresponding ring $A$ in this section to define the corresponding noncommutative version of the Robba rings. We keep all the notations compatible with the ones in the commutative setting in the previous discussion in the previous sections.	
\end{setting}

\begin{definition} \mbox{\bf{(After KPX \cite[Definition 2.2.2]{KPX})}}
We keep all the notations (except the coefficient $A$ which is now noncommutative) compatible with the ones in the commutative setting in the previous discussion in the previous sections. For each $\alpha\in I$, we choose suitable uniformizer $\pi_{K_\alpha}$ in our consideration. Then we are going to use the notation $\Pi_{\mathrm{an},\mathrm{con},I,A}(\pi_{K_I})$ to denote the corresponding period ring constructed from $\Pi_{\mathrm{an},\mathrm{con},I,A}$ just by directly replacing the variables by the corresponding uniformizers as above. And similarly for other period rings we use the corresponding notations taking the same form namely $\Pi_{\mathrm{an},\mathrm{con},I,A}(\pi_{K_I})$, $\Pi_{\mathrm{an},r_{I,0},I,A}(\pi_{K_I})$, $\Pi_{\mathrm{an},[s_I,r_I],I,A}(\pi_{K_I})$. As in the one dimensional situation we consider the situation where the radii are all sufficiently small. Then one could define the multiple Frobenius actions from the multi Frobenius $\varphi_I$ over each the ring mentioned above. For suitable radii $r_I$ we define a $\varphi_I$-module to be a finite projective left $\Pi_{\mathrm{an},r_{I},I,A}(\pi_{K_I})$-module with the requirement that for each $\alpha\in I$ we have $\varphi_\alpha^*M_{r_I}\overset{\sim}{\rightarrow}M_{...,r_\alpha/p,...}$ (after suitable base changes). Then we define $M:=\Pi_{\mathrm{an},\mathrm{con},I,A}(\pi_{K_I})\otimes_{\Pi_{\mathrm{an},r_{I},I,A}(\pi_{K_I})}M_{r_I}$ to define a $\varphi_I$-module over the full relative Robba ring in our situation. Furthermore we have the notion of $\varphi_I$-bundles in our situation which consists of a family of $\varphi_I$-modules $\{M_{[s_I,r_I]}\}$ where each module $M_{[s_I,r_I]}$ is defined to be finite projective over $\Pi_{\mathrm{an},[s_I,r_I],I,A}(\pi_{K_I})$ satisfying the action formula taking the form of $\varphi_\alpha^*M_{[s_I,r_I]}\overset{\sim}{\rightarrow}M_{...,[s_\alpha/p,r_\alpha/p],...}$ (after suitable base changes). Furthermore for each $\alpha$ we have the corresponding operator $\varphi_\alpha:M_{[s_I,r_I]}\rightarrow M_{...,[s_\alpha/p,r_\alpha/p],...}$ and we have the corresponding operator $\psi_\alpha$ which is defined to be $p^{-1}\varphi_\alpha^{-1}\circ\mathrm{Trace}_{M_{...,[s_\alpha/p,r_\alpha/p],...}/\varphi_\alpha(M_{[s_I,r_I]})}$. Certainly we have the corresponding operator $\psi_\alpha$ over the global section $M_{r_I}$. Note that here we require that the Hodge-Frobenius structures are commutative in the sense that all the Frobenius are commuting with each other, and they are semilinear.
\end{definition}

\begin{definition} \mbox{\bf{(After KPX \cite[Definition 2.2.2]{KPX})}} \label{definition6.4}
We keep all the notations (except the coefficient $A$ which is now noncommutative) compatible with the ones in the commutative setting in the previous discussion in the previous sections. For each $\alpha\in I$, we choose suitable uniformizer $\pi_{K_\alpha}$ in our consideration. Then we are going to use the notation $\Pi_{\mathrm{an},\mathrm{con},I,A}(\pi_{K_I})$ to denote the corresponding period ring constructed from $\Pi_{\mathrm{an},\mathrm{con},I,A}$ just by directly replacing the variables by the corresponding uniformizers as above. And similarly for other period rings we use the corresponding notations taking the same form namely $\Pi_{\mathrm{an},\mathrm{con},I,A}(\pi_{K_I})$, $\Pi_{\mathrm{an},r_{I,0},I,A}(\pi_{K_I})$, $\Pi_{\mathrm{an},[s_I,r_I],I,A}(\pi_{K_I})$. As in the one dimensional situation we consider the situation where the radii are all sufficiently small. Then one could define the multiple Frobenius actions from the multi Frobenius $\varphi_I$ over each the ring mentioned above. For suitable radii $r_I$ we define a coherent $\varphi_I$-module to be a finitely presented left $\Pi_{\mathrm{an},r_{I},I,A}(\pi_{K_I})$-module with the requirement that for each $\alpha\in I$ we have $\varphi_\alpha^*M_{r_I}\overset{\sim}{\rightarrow}M_{...,r_\alpha/p,...}$ (after suitable base changes). Then we define $M:=\Pi_{\mathrm{an},\mathrm{con},I,A}(\pi_{K_I})\otimes_{\Pi_{\mathrm{an},r_{I},I,A}(\pi_{K_I})}M_{r_I}$ to define a coherent $\varphi_I$-module over the full relative Robba ring in our situation. Furthermore we have the notion of coherent $\varphi_I$-bundles in our situation which consists of a family of coherent $\varphi_I$-modules $\{M_{[s_I,r_I]}\}$ where each module $M_{[s_I,r_I]}$ is defined to be finitely presented over $\Pi_{\mathrm{an},[s_I,r_I],I,A}(\pi_{K_I})$ satisfying the action formula taking the form of $\varphi_\alpha^*M_{[s_I,r_I]}\overset{\sim}{\rightarrow}M_{...,[s_\alpha/p,r_\alpha/p],...}$ (after suitable base changes). Furthermore for each $\alpha$ we have the corresponding operator $\varphi_\alpha:M_{[s_I,r_I]}\rightarrow M_{...,[s_\alpha/p,r_\alpha/p],...}$ and we have the corresponding operator $\psi_\alpha$ which is defined to be $p^{-1}\varphi_\alpha^{-1}\circ\mathrm{Trace}_{M_{...,[s_\alpha/p,r_\alpha/p],...}/\varphi_\alpha(M_{[s_I,r_I]})}$. Certainly we have the corresponding operator $\psi_\alpha$ over the global section $M_{r_I}$. Note that here we require that the Hodge-Frobenius structures are commutative in the sense that all the Frobenius are commuting with each other, and they are semilinear.
\end{definition}

\begin{definition} \mbox{\bf{(After KPX \cite[Definition 2.2.12]{KPX})}}
Then we add the corresponding more Lie group action in our setting, namely the multi semiliear $\Gamma_{K_I}$-action. Again in this situation we require that all the actions of the $\Gamma_{K_I}$ are commuting with each other and with all the semilinear Frobenius actions defined above. We define the corresponding $(\varphi_I,\Gamma_I)$-modules over $\Pi_{\mathrm{an},r_{I,0},I,A}(\pi_{K_I})$ to be finite projective module with mutually commuting semilinear actions of $\varphi_I$ and $\Gamma_{K_I}$. For the latter action we require that to be continuous. We can also define the corresponding actions for coherent Frobenius modules.
\end{definition}


\begin{definition} \mbox{\bf{(After KPX \cite[2.3]{KPX})}}
For any $(\varphi_I,\Gamma_I)$-module $M$ over $\Pi_{\mathrm{an},\mathrm{con},I,A}(\pi_{K_I})$ we define the complex $C^\bullet_{\Gamma_I}(M)$ of $M$ to be the total complex of the following complex through the induction:
\[
\xymatrix@R+0pc@C+0pc{
[ C^\bullet_{\Gamma_{I\backslash\{I\}}}(M) \ar[r]^{\Gamma_I} \ar[r] \ar[r]  & C^\bullet_{\Gamma_I\backslash\{I\}}(M)]
.}
\]
Then we define the corresponding double complex $C^\bullet_{\varphi_I}C^\bullet_{\Gamma_I}(M)$ again by taking the corresponding totalization of the following complex through induction:
\[
\xymatrix@R+0pc@C+0pc{
[C^\bullet_{\varphi_{I\backslash\{I\}}}C^\bullet_{\Gamma_{I}}(M) \ar[r]^{\varphi_I} \ar[r] \ar[r]  & C^\bullet_{\varphi_{I\backslash\{I\}}}C^\bullet_{\Gamma_I}(M)]
.}
\]
Then we define the complex $C^\bullet_{\varphi_I,\Gamma_I}(M)$ to be the totalization of the double complex defined above, which is called the complex of a $(\varphi_I,\Gamma_I)$-module $M$. Similarly we define the corresponding complex $C^\bullet_{\psi_I,\Gamma_I}(M)$ to be the totalization of the double complex define  d in the same way as above by replacing $\varphi_I$ by then the operators $\psi_I$. For later comparison we also need the corresponding $\psi_I$-cohomology, which is to say the complex $C^\bullet_{\psi_I}(M)$ which could be defined by the totalization of the following complex through induction:
\[
\xymatrix@R+0pc@C+0pc{
[ C^\bullet_{\psi_{I\backslash\{I\}}}(M) \ar[r]^{\psi_I} \ar[r] \ar[r]  & C^\bullet_{\psi_I\backslash\{I\}}(M)]
.}
\]
\end{definition}

\indent In order to proceed we now make the following assumption:

\begin{assumption} \label{assumption6.12}
We assume that the ring $A$ in our situation satisfies the following assumption. For any closed interval $[s_I,r_I]$ we assume that the Robba ring $\Pi_{\mathrm{an},[s_I,r_I],I,A}$ is left and right noetherian, namely noetherian in the noncommutative situation. In the following we will discuss some samples.
\end{assumption}

\begin{example} 
First note we point out that the corresponding Tate algebra with several free variables is actually not noetherian, even the corresponding polynomial rings with several free variables are not noetherian at all. For instance in the situation where we have two variables, as in \cite[Introduction]{K1} and \cite[Exercise 1E]{GW1} we know that the ring $k[x,y]^\mathrm{nc}$ noncommutative polynomial ring contains specific corresponding ideal generated by $xy^k$ for each nonnegative integer power $k\geq 0$. Even in the commutative setting, suppose we have a commutative noetherian Banach algebra $B$, it is also definitely not trivial that the affinoid algebra:
\begin{displaymath}
B\{T_1/a_1,...,T_m/a_m,b_1/T_1,...b_m/T_m\}	
\end{displaymath}
is noetherian or not. Actually there does exist some counterexample which makes this not true in general as one might not expect, as in \cite[8.3]{FGK}. 
\end{example}

\begin{example}
One can actually produce another examplification of the \cref{assumption6.12}, namely we consider the following situation. We look at the corresponding (generalized version of) twisted polynomial rings in the sense of \cref{example5.1} (see \cite[Section 3]{So1}). For instance we take a such ring with three variables:
\begin{align}
\mathbb{Q}_p[X_1,X_2,X_3]'	
\end{align}
with:
\begin{align}
X_iX_j=a_{i,j}X_jX_i,i,j\in\{1,2,3\}	
\end{align}
such that $a_{3,1}=a_{1,3}=1$, $a_{3,2}=a_{2,3}=1$ and $a_{2,1}=a_{1,2}\neq 1$. Then after taking the corresponding completion under the Gauss norm we will have the corresponding noncommutative affinoid which is noetherian and satisfies the corresponding \cref{assumption6.12} by suitable rational localization.
\end{example}

Now consider the corresponding rings in the style considered in \cite[Proposition 4.1]{Zab1}:
\begin{displaymath}
D_{[s,r]}(\mathbb{Z}_p,\mathbb{Q}_p), D_{[\rho_1,\rho_2]}(G,K)	
\end{displaymath}
with $G$ a pro-$p$ group as in \cite[Chapter 4]{Zab1} (which is to say we will also assume $G$ to be uniform). The corresponding ring $D_{[\rho_1,\rho_2]}(G,K)$ could be defined through suitable microlocalization with respect to some noncommutative polyannulus with respect to the radii $\rho_1,\rho_2$ which is direct noncommutative generalization of the usual Robba ring we considered in this paper including obviously the corresponding multivariate ones. With the notation in \cite[Proposition 4.1]{Zab1} namely for sufficiently large $s,r,\rho_1,\rho_2\in p^\mathbb{Q}$ and $s,r,\rho_1,\rho_2<1$ we have:

\begin{theorem} \mbox{\bf{(Z\'abr\'adi \cite[Proposition 4.1]{Zab1})}}
The noncommutative Robba rings $D_{[\rho_1,\rho_2]}(G,K)$ is noetherian.	
\end{theorem}

We now take the corresponding completed tensor product we have then:
\begin{displaymath}
D_{[s,r]}(\mathbb{Z}_p,\mathbb{Q}_p)\widehat{\otimes}_{\mathbb{Q}_p}D_{[\rho_1,\rho_2]}(G,K)	
\end{displaymath}
which corresponds directly to our interested situation, namely the noncommutative deformation of the usual Robba ring in one variable.

\begin{proposition}
For suitable radii $s,r,\rho_1,\rho_2\in p^\mathbb{Q}$ such that $s,r,\rho_1,\rho_2<1$ and:
\begin{displaymath}
D_{[s,r]}(\mathbb{Z}_p,\mathbb{Q}_p), D_{[\rho_1,\rho_2]}(G,K)	
\end{displaymath}
are noetherian as in \cite[Proposition 4.1]{Zab1}. Then we have that the corresponding completed tensor product 
\begin{displaymath}
D_{[s,r]}(\mathbb{Z}_p,\mathbb{Q}_p)\widehat{\otimes}_{\mathbb{Q}_p}D_{[\rho_1,\rho_2]}(G,K)	
\end{displaymath}
is also noetherian.

\end{proposition}
	
\begin{remark}
Before the proof, let us mention that this actually provides a very typical example for situation required in \cref{assumption6.12}.	
\end{remark}

\begin{proof}
One can follow the proof of \cite[Proposition 4.1]{Zab1} to do so by applying the criterion in \cite[Proposition I.7.1.2]{LVO}, namely we look at the the graded ring of 
\begin{displaymath}
gr^\bullet_{\|.\|_{\otimes}} \left(D_{[s,r]}(\mathbb{Z}_p,\mathbb{Q}_p)\widehat{\otimes}_{\mathbb{Q}_p}D_{[\rho_1,\rho_2]}(G,K)\right),	
\end{displaymath}
defined by using the bounds of the norm $\|.\|_{\otimes}$, and if this is noetherian then we will have by \cite[Proposition I.7.1.2]{LVO} that the ring $D_{[s,r]}(\mathbb{Z}_p,\mathbb{Q}_p)\widehat{\otimes}D_{[\rho_1,\rho_2]}(G,K)$ is also noetherian. 
It suffices to look at the corresponding graded ring of the dense subring of tensor product (see \cite[Lemma 4.3]{ST1}) 
\begin{displaymath}
gr^\bullet_{\|.\|_{\otimes}} \left(D_{[s,r]}(\mathbb{Z}_p,\mathbb{Q}_p)\otimes_{\mathbb{Q}_p} D_{[\rho_1,\rho_2]}(G,K)\right).	
\end{displaymath}
We do not have the corresponding isomorphism in mind to split the corresponding components to get a product of two single graded rings but we have the corresponding surjective map (see \cite[I.6.13]{LVO}):
\begin{displaymath}
gr^\bullet \left(D_{[s,r]}(\mathbb{Z}_p,\mathbb{Q}_p))\otimes_{gr^\bullet \mathbb{Q}_p} gr^\bullet( D_{[\rho_1,\rho_2]}(G,K)\right)\rightarrow gr_{\|.\|_{\otimes}}^\bullet \left(D_{[s,r]}(\mathbb{Z}_p,\mathbb{Q}_p)\otimes_{\mathbb{Q}_p} D_{[\rho_1,\rho_2]}(G,K)\right)\rightarrow 0.	
\end{displaymath}
This is defined by:
\begin{displaymath}
\sigma_k(x)\otimes \sigma_l(y)\mapsto \sigma_{k+l}(x\otimes y),	
\end{displaymath}
where the principal symbol $\sigma_k(x)$ is defined to be $x+F^{k+} D_{[s,r]}(\mathbb{Z}_p,\mathbb{Q}_p)$ if we we have $x\in F^k D_{[s,r]}(\mathbb{Z}_p,\mathbb{Q}_p) \backslash F^{k+} D_{[s,r]}(\mathbb{Z}_p,\mathbb{Q}_p)$, in the corresponding algebraic microlocalization. The construction for $\sigma_l(y)$ and $ \sigma_{k+l}(x\otimes y)$ is the same. Therefore we look at then the corresponding product:
\begin{displaymath}
gr^\bullet \left(D_{[s,r]}(\mathbb{Z}_p,\mathbb{Q}_p)\right)\otimes gr^\bullet \left( D_{[\rho_1,\rho_2]}(G,K)\right)	
\end{displaymath}
The corresponding elements of $gr^\bullet \left( D_{[\rho_1,\rho_2]}(G,K)\right)$ are those Laurent polynomials over $gr^\bullet K=k[\pi_2,\pi_2^{-1}]$ with variables:
\begin{displaymath}
t_1,...,t_d	
\end{displaymath}
where $t_i=\sigma(b_i)/\pi_2^l$. The corresponding elements of $gr^\bullet \left(D_{[s,r]}(\mathbb{Z}_p,\mathbb{Q}_p) \right)$ are those Laurent polynomials over $gr^\bullet \mathbb{Q}_p=\mathbb{F}_p[\pi_1,\pi_1^{-1}]$ with variable:
\begin{displaymath}
t=\sigma(b)/\pi_1^{k}.
\end{displaymath}
We followed the notations in \cite[Proposition 4.1]{Zab1}, where $b$ is the corresponding local coordinate for $D_{[s,r]}(\mathbb{Z}_p,\mathbb{Q}_p)$ while $\sigma$ represents the corresponding principal symbol in the corresponding nonarchimedean microlocalization, while we have that $b_1,...,b_d$ are local coordinates for $D_{[\rho_1,\rho_2]}(G,K)$. The idea is that this ring is actually noetherian then we have the ring
\begin{displaymath}
gr^\bullet \left(D_{[s,r]}(\mathbb{Z}_p,\mathbb{Q}_p)\otimes D_{[\rho_1,\rho_2]}(G,K)\right)	
\end{displaymath}
is also noetherian. Then one can endow 
\begin{displaymath}
D_{[s,r]}(\mathbb{Z}_p,\mathbb{Q}_p)\widehat{\otimes} D_{[\rho_1,\rho_2]}(G,K)
\end{displaymath}
with the corresponding filtration which will imply that the ring 
\begin{displaymath}
D_{[s,r]}(\mathbb{Z}_p,\mathbb{Q}_p)\widehat{\otimes}  D_{[\rho_1,\rho_2]}(G,K)
\end{displaymath}
is noetherian as well by \cite[Proposition I.7.1.2]{LVO}. Indeed as in \cite[Proposition 4.1]{Zab1} what we could do is to consider the corresponding projection to the corresponding Laurent polynomial ring over $k$ in the variables $\pi_2,t,t_1,...,t_d$ as in the proof of \cite[Proposition 4.1]{Zab1}:
\begin{displaymath}
\mathbb{F}_p[\pi_1,\pi_1^{-1},t,t^{-1}]\otimes_{\mathbb{F}_p[\pi_1,\pi_1^{-1}]} k[\pi_2,\pi_2^{-1},t_1,t_1t_2^{-1},...,t_1t_d^{-1},t_2t_1^{-1},...,t_dt_1^{-1}]	
\end{displaymath}
which is noetherian, then the corresponding ring
\begin{displaymath}
gr^\bullet \left(D_{[s,r]}(\mathbb{Z}_p,\mathbb{Q}_p))\otimes_{gr^\bullet\mathbb{Q}_p} gr^\bullet( D_{[\rho_1,\rho_2]}(G,K)\right)	
\end{displaymath}
is noetherian since one can then lift the generators for the corresponding projection of any ideal through this construction as in the same way of \cite[Proposition 4.1]{Zab1}.
\end{proof}


\begin{setting}
Now under the corresponding noetherian assumption in \cref{assumption6.12}, we actually have that the corresponding multivariate Robba ring $\Pi_{\mathrm{an},r_{I,0},I,A}$	is Fr\'echet-Stein in the sense of \cite[Chapter 3]{ST1}. 
\end{setting}

\begin{definition} \mbox{\bf{(After KPX \cite[Definition 2.1.3]{KPX})}}
We define coherent sheaves of left modules and the global sections over the ring $\Pi_{\mathrm{an},r_{I_0},I,A}$ in the following way. First a coherent sheaf $\mathcal{M}$ over the relative multidimensional Robba ring $\Pi_{\mathrm{an},r_{I_0},I,A}$ is defined to be a family $(M_{[s_I,r_I]})_{[s_I,r_I]\subset (0,r_{I,0}]}$of left modules of finite type over each relative $\Pi_{[s_I,r_I],I,A}$ satisfying the following two conditions as in the more classical situation:
\begin{displaymath}
\Pi_{[s_I'',r_I''],I,A}\otimes_{\Pi_{[s'_I,r'_I],I,A}}M_{[s_I',r_I']}\overset{\sim}{\rightarrow} M_{[s''_I,r''_I]}	
\end{displaymath}
for any multi radii satisfying $0<s_I'\leq s_I''\leq r_I''\leq r_I'\leq r_{I,0}$, with furthermore 
\begin{displaymath}
\Pi_{[s'''_I,r'''_I],I,A}\otimes_{\Pi_{[s''_I,r''_I],I,A}}(\Pi_{[s''_I,r''_I],I,A}	\otimes_{\Pi_{[s'_I,r'_I],I,A}}M_{[s'_I,r'_I]}) \overset{\sim}{\rightarrow}\Pi_{[s'''_I,r'''_I],I,A},
\end{displaymath}
for any multi radii satisfying the following condition:
\begin{displaymath}
0<s_I'\leq s_I''\leq s_I'''\leq r_I'''\leq r_I''\leq r_I'\leq r_{I,0}.	
\end{displaymath}
Then we define the corresponding global section of the coherent sheaf $\mathcal{M}$, which is usually denoted by $M$ which is defined to be the following inverse limit:
\begin{displaymath}
M:=\varprojlim_{s_I\rightarrow 0_I^+} M_{[s_I,r_{I,0}]}.	
\end{displaymath}
We are going to follow \cite[Definition 2.1.3]{KPX} to call any module defined over $\Pi_{\mathrm{an},r_{I_0},I,A}$ coadmissible it comes from a global section of a coherent sheaf $\mathcal{M}$ in the sense defined above. 
\end{definition}

\begin{remark}
Since we are in the framework of \cite[Chapter 3]{ST1}, so now by \cite[Chapter 3, Theorem]{ST1} for $r_{I,0}$ finite (namely for all $\alpha\in I$ the radius $r_{\alpha,0}$ is finite) we have that the corresponding global section in the previous definition is dense in each member participating in the coherent sheaf over some multi interval. We also have the same result for the situation where $r_{I,0}=\infty^{|I|}$ after inverting the corresponding variables in the multivariate Robba rings as in the commutative situation we considered in \cref{prop2.9}. Also we have the higher vanishing of the corresponding derived inverse limit on coherent sheaves in the sense defined above.	
\end{remark}

\begin{definition} \mbox{\bf{(After KPX \cite[Definition 2.1.9]{KPX})}} 
We first generalize the notion of admissible covering in our setting, which is higher dimensional generalization of the situation established in \cite[Definition 2.1.9]{KPX}. For any covering $\{[s_I,r_I]\}$ of $(0,r_{I,0}]$, we are going to specify those admissible coverings if the given covering admits refinement finite locally (and each corresponding member in the covering has the corresponding interior part which is assumed to be not empty with respect to each $\alpha\in I$). Then it is very natural in our setting that we have the corresponding notations of $(m,n)$-finitely presentedness and $n$-finitely generatedness for any $m,n$ positive integers. Then based on these definitions we have the following generalization of the corresponding uniform finiteness. First we are going to call a coherent sheaf $\{M_{s_I,r_I}\}_{[s_I,r_I]}$ over $\Pi_{\mathrm{an},r_{I,0},I,A}$ uniformly $(m,n)$-finitely presented if there exists an admissible covering $\{[s_I,r_I]\}$ and a pair of positive integers $(m,n)$ such that each module $M_{[s_I,r_I]}$ defined over $\Pi_{\mathrm{an},[s_I,r_I],I,A}$ is $(m,n)$-finitely presented. Also we have the notion of uniformly $n$-finitely generated for any positive integer $n$ by defining that to mean under the existence of an admissible covering we have that each member $M_{[s_I,r_I]}$ in the family defined over $\Pi_{\mathrm{an},[s_I,r_I],I,A}$ is $n$-finitely generated. 
\end{definition}

\begin{lemma}\mbox{\bf{(After KPX \cite[Lemma 2.1.11]{KPX})}} Suppose that for a coherent sheaf $\{M_{s_I,r_I}\}_{[s_I,r_I]}$ defined above and for an admissible covering $\{[s_I,r_I]\}$ defined above in our noncommutative setting we have that any module $M_{s_I,r_I}$ is finitely generated uniformly by finite many elements of global section. Then the global section will be also finitely generated by these set of elements as well. 
	
\end{lemma}

\begin{proof}
As in \cite[Lemma 2.1.11]{KPX}, we use the set of elements to present a map from the free module generated by them to the global section, then after base change to each $\Pi_{\mathrm{an},[s_I,r_I],I,A}$ we do have that the map is corresponding presentation over $\Pi_{\mathrm{an},[s_I,r_I],I,A}$, then this will give rise to the corresponding presentation for the global section since we have the vanishing of the derived inverse limit functor.	
\end{proof}

\begin{lemma} \mbox{\bf{(After KPX \cite[Lemma 2.1.12]{KPX})}} \label{lemma6.16}
Let $H$ be a noncommutative Banach ring, and consider some finite module $M$ over $H$ generated by $\{f_1,...,f_n\}$ and any another set $\{f'_1,...,f'_n\}$ of elements such that $f_i-f_i'$ could be made sufficiently small uniformly for each $i=1,...,n$. Then we have that the set $\{f'_1,...,f'_n\}$ could also finitely generate the module $M$.	
\end{lemma}

\begin{proof}
As in \cite[Lemma 2.1.12]{KPX}, one looks at the corresponding matrix $G_{i,j}$ defined by $f_i-f_i'=\sum_{k}G_{i,k}f_k$. Then by using the open mapping theorem, one can control the norm of each entry $G_{i,j}$ to make them small enough below a uniform small bound, which makes the corresponding matrix $1+G$ invertible.	
\end{proof}

\begin{proposition} \mbox{\bf{(After KPX \cite[Proposition 2.1.13]{KPX})}} \label{prop1}
Assume now $I$ consists of exactly one index, consider any arbitrary coherent sheaf $\{M_{[s_I,r_I]}\}$ over $\Pi_{\mathrm{an},r_{I,0},I,A}$ with the global section $M$. Then we have the following corresponding statements:\\
I.	The coherent sheaf $\{M_{[s_I,r_I]}\}$ is uniformly finitely generated iff the global section $M$ is finitely generated;\\
II. The coherent sheaf $\{M_{[s_I,r_I]}\}$ is uniformly finitely presented iff the global section $M$ is finitely presented.
\end{proposition}

\begin{proof}
One could derive the the second statement by using the analog of \cref{lemma2.15} (see \cite[Lemma 4.54]{L}) and the first statement. Therefore it suffices for us to prove the first statement. Indeed one could see that one direction in the statement is straightforward, therefore it suffices now to consider the other direction. So now we assume that the coherent sheaf $\{M_{[s_I,r_I]}\}$ is uniformly finitely generated for some $n$ for instance. Then we are going to show that one could find corresponding generators for the global section $M$ from the given uniform finitely generatedness. 

Now by definition we consider an admissible covering taking the form of $\{[s_I,r_I]\}$. In \cite[Proposition 2.1.13]{KPX} the corresponding argument goes by using special decomposition of the covering in some refined way, mainly by reorganizing the initial admissible covering into two single ones, where each of them will be made into the one consisting of those well-established intervals with no overlap. We can also use instead the corresponding upgraded nice decomposition of the given admissible covering in our setting, which is extensively discussed in \cite[2.6.14-2.6.17]{KL2}. Namely as in \cite[2.6.14-2.6.17]{KL2} there is a chance for us to extract $2^{|I|}$ families $\{\{[s_{I,\delta_i},r_{I,\delta_i}]\}_{i=0,1,...}\}_{\{1_i\},\{2_i\},...,\{N_i\},}$ ($N=2^{|I|}$) of intervals, where there is a chance to have the situation where for each such family the corresponding intervals in it could be made to be disjoint in pairs.

Then for a chosen family $\{[s_{I,\delta_i},r_{I,\delta_i}]\}_{i=0,1,...}$, we now for each $i$ involved we consider the corresponding generating set $\mathbf{e}_{\delta_i,1},...,\mathbf{e}_{\delta_i,n}$ where $i=0,1,...$ and $n\in \{1,2,...,n\}$. Note that these should be generated over the integral rings, which is not the case when we work with the infinite radii. In the latter situation one should first invert $T_I$ then consider some large powers which could eliminate the different denominators. Then by the previous results we know that one could actually arrange these generators to be global sections by the density of the global sections in the sections over any chosen interval. Now we consider the following step where we for each $j=1,...,n$ form the following sum with coefficients to be chosen:
\begin{displaymath}
\mathbf{e}_j:=\sum_{i=0}^\infty a_{\delta_i,j}T_I^{\lambda_{I,\delta_i,j}}\mathbf{e}_{\delta_i,j}.
\end{displaymath}
Note that in our situation actually $I$ is a singleton, so we can actually regard $\delta_i$ as $2i$.

To check the suitable convergence in this situation we consider the following estimate with bounds given as below to be chosen from the conditions:
\begin{align}
\forall i''<\delta_i-1, \left\|a_{\delta_i,j}T^{\lambda_{I,\delta_i,j}}\mathbf{e}_{\delta_i,j}\right\|_{[s_{I,i''},r_{I,i''}]}\leq p^{-i},\\
\forall i'<i, \left\|a_{\delta_i,j}T^{\lambda_{I,\delta_i,j}}\mathbf{e}_{\delta_i,j}/(a_{\delta_{i'},j}T^{\lambda_{I,\delta_{i'},j}})\right\|_{[s_{I,\delta_{i'}},r_{I,\delta_{i'}}]} \leq \varepsilon_{\delta_{i'},j}\\
\forall i'<i, \left\|a_{\delta_{i'},j}T^{\lambda_{I,\delta_{i'},j}}\mathbf{e}_{\delta_{i'},j}/(a_{\delta_i,j}T^{\lambda_{I,\delta_i,j}})\right\|_{[s_{I,\delta_{i}},r_{I,\delta_{i}}]} \leq \varepsilon_{\delta_i,j}.	
\end{align}
As in the commutative situation in \cite[Proposition 2.1.13]{KPX} in order to get the desired quantity expressed in the estimate above we seek the following situation from what we presented above:
\begin{align}
\forall i''<\delta_i-1, \left\|a_{\delta_i,j}T^{\lambda_{I,\delta_i,j}}\mathbf{e}_{\delta_i,j}\right\|_{[s_{I,i''},r_{I,i''}]}\leq p^{-i},\\
\forall i'<i, \left\|a_{\delta_i,j}T^{\lambda_{I,\delta_i,j}}\mathbf{e}_{\delta_i,j}\right\|_{[s_{I,\delta_{i'}},r_{I,\delta_{i'}}]} \leq \varepsilon_{\delta_{i'},j} \left\|a_{\delta_{i'},j}T^{\lambda_{I,\delta_{i'},j}}\right\|_{[s_{I,\delta_{i'}},r_{I,\delta_{i'}}]}\\
\forall i'<i, \left\|a_{\delta_{i'},j}T^{\lambda_{I,\delta_{i'},j}}\mathbf{e}_{\delta_{i'},j}\right\|_{[s_{I,\delta_i},r_{I,\delta_i}]} \leq \varepsilon_{\delta_i,j} \left\|a_{\delta_i,j}T^{\lambda_{I,\delta_i,j}}\right\|_{[s_{I,\delta_i},r_{I,\delta_i}]}.	
\end{align} 	
Then what we can achieve up to now is to consider sufficiently large powers on the both sides in order to make the inequalities hold, which is possible due to our initial choice on the intervals, where the first statement could also be proved in the similar way (also see \cite[Proposition 2.1.13]{KPX}). To be more precise, one chooses the corresponding coefficients and the powers of the monomials by induction, where the initial coefficient for $i=0$ is set to be 1 with order 0. Then one enlarges the corresponding difference between the power of the monomials on the both sides of the inequalities above while evaluating the first and second inequalities at the left most radii and evaluating the last one at right most radii. Then actually this will prove the statement since up to some powers the difference between $\mathbf{e}_{\delta_i,j}$ and $\mathbf{e}_j$ will have sufficiently small norm, which could give rise to suitable finitely generatedness through the \cref{lemma6.16}. Finally one could then finish the whole proof in the odd situation in the same fashion.

\end{proof}

\begin{proposition} \mbox{\bf{(After KPX \cite[Proposition 2.2.7]{KPX})}}
In the situation where $I$ consists of exactly one element. We have that natural functor from coherent $\varphi_I$-modules (in the sense of the previous \cref{definition6.4}) over $\Pi_{\mathrm{an},r_{I,0},I,A}(\pi_{K_I})$ to the corresponding coherent $\varphi_I$-bundles (in the sense of the previous \cref{definition6.4}) is an equivalence.	
\end{proposition}

\begin{proof}
See \cref{proposition3.5}, by using the previous proposition.	
\end{proof}

\begin{definition} \mbox{\bf{(After KPX \cite[Notation 4.1.2]{KPX})}}
Now we consider the following derived categories. All modules will be assumed to be left. The first one is the sub-derived category consisting of all the objects in $\mathbb{D}_\mathrm{left}(A)$ which are quasi-isomorphic to those bounded above complexes of finite projective modules over the ring $A$. We denote this category (which is defined in the same way as in \cite[Notation 4.1.2]{KPX}) by $\mathbb{D}^-_\mathrm{perf,left}(A)$. Over the larger ring $\Pi_{\mathrm{an},\infty,I,A}(\Gamma_{K_I})$ we also have the sub-derived category of $\mathbb{D}_\mathrm{left}(\Pi_{\mathrm{an},\infty,I,A}(\Gamma_{K_I}))$ consisting of all those objects in $\mathbb{D}_\mathrm{left}(\Pi_{\mathrm{an},\infty,I,A}(\Gamma_{K_I}))$ which are quasi-isomorphic to those bounded above complexes of finite projective modules now over the ring $\Pi_{\mathrm{an},\infty,I,A}(\Gamma_{K_I})$. We denote this by $\mathbb{D}^-_\mathrm{perf,left}(\Pi_{\mathrm{an},\infty,I,A}(\Gamma_{K_I}))$. Similar we define the bounded derived categories $\mathbb{D}^\flat_\mathrm{perf,left}(A)$, $\mathbb{D}^\flat_\mathrm{perf,left}(\Pi_{\mathrm{an},\infty,I,A}(\Gamma_{K_I}))$. We then use the notation as following $D_\mathrm{perf,left}(A)$, $D_\mathrm{perf,left}(\Pi_{\mathrm{an},\infty,I,A}(\Gamma_{K_I}))$, $D^-_\mathrm{perf,left}(A)$, $D^-_\mathrm{perf,left}(\Pi_{\mathrm{an},\infty,I,A}(\Gamma_{K_I}))$, as well as $D^\flat_\mathrm{perf,left}(A)$, $D^\flat_\mathrm{perf,left}(\Pi_{\mathrm{an},\infty,I,A}(\Gamma_{K_I}))$ to denote the $(\infty, 1)$-enhancement of these derived categories.
\end{definition}

\indent Then we will investigate the complexes attached to any finite projective or coherent $(\varphi_I,\Gamma_I)$-module $M$ over $\Pi_{\mathrm{an},\mathrm{con},I,A}(\pi_{K_I})$:
\begin{displaymath}
C^\bullet_{\varphi_I,\Gamma_I}(M), C^\bullet_{\psi_I,\Gamma_I}(M),C^\bullet_{\psi_I}(M),	
\end{displaymath}
which are living inside the corresponding derived categories:
\begin{displaymath}
\mathbb{D}_\mathrm{left}(A),\mathbb{D}_\mathrm{left}(A),\mathbb{D}_\mathrm{left}(\Pi_{\mathrm{an},\infty,I,A}(\Gamma_{K_I})).	
\end{displaymath}

\begin{conjecture} \label{conjecture5.22}
With the notation and setting up as above we have $h^\bullet(C_{\psi_I}(M))$ are coadmissible over $\Pi_{\mathrm{an},\infty,I,A}(\Gamma_{K_I})$.	
\end{conjecture}


\begin{conjecture} \label{conjecture5.23}
With the notations above, we have that 
\begin{displaymath}
C^\bullet_{\varphi_I,\Gamma_{I}}(M)\in \mathbb{D}_\mathrm{perf,left}^-(A), C^\bullet_{\psi_I,\Gamma_{I}}(M)\in \mathbb{D}_\mathrm{perf,left}^-(A).
\end{displaymath}
\end{conjecture}

\begin{conjecture} \label{conjecture5.24}
Let $A_\infty(H)$ be the Fr\'echet-Stein algebra attached to any compact $p$-adic Lie group $H$ which could be written as inverse limit of noncommutative noetherian Banach affinoid algebras over $\mathbb{Q}_p$. Suppose $\mathcal{F}$ is a family of finite projective left modules over the Robba ring $\Pi_{\mathrm{an},\mathrm{con},I,A_\infty(H)}(\pi_{K_I})$, where we use the same notation to denote the global section and we assume that the global section is finite projective, carrying mutually commuting actions of $(\varphi_I,\Gamma_{K_I})$. Then we have that $C^\bullet_{\psi_I}(\mathcal{F})$ lives in $\mathbb{D}^\flat_{\mathrm{perf,left}}(\Pi_{\mathrm{an},\infty,I,A_\infty(H)}(\Gamma_{K_I}))$ and $C^\bullet_{\varphi_I,\Gamma_I}(\mathcal{F})$ lives in $\mathbb{D}^\flat_{\mathrm{perf,left}}(A_\infty(H))$. 
\end{conjecture}


\begin{remark}
We conjecture that the corresponding Cartan-Serre approach holds in the noncommutative setting as in \cite[Lemma1.10]{KL3}, which may possibly provide proof of the conjectures above.	
\end{remark}

\newpage

\subsection*{Acknowledgements} 

The inspiration we acquired from \cite{CKZ18} is obvious to the readers. We are also quite inspired by the deep philosophy of Burns-Flach-Fukaya-Kato. We would like to thank Professor Z\'abr\'adi for helpful discussion. We would like to thank Professor Kedlaya for explaining us the method of Cartan-Serre which is a key suggestion to this project, and for suggestion around the triangulation as well as helpful discussion along the preparation which directly leads to the current presentation. During the preparation, under Professor Kedlaya the author was supported by NSF Grant DMS-1844206.

\newpage

\bibliographystyle{ams}

\end{document}